\documentclass[10pt]{article}

\usepackage{amsfonts}
\usepackage{amsmath}   
\usepackage{amssymb}
\usepackage{amsthm}
\usepackage{mathtools}
\usepackage{tikz-cd}
\usepackage{xcolor}
\usepackage[colorlinks=true]{hyperref}
 
\newtheorem*{Th*}{Theorem}
\newtheorem{Th}{Theorem}[section]
\newtheorem{Prop}{Proposition}[section]   
\newtheorem{Lem}{Lemma}[section]   
\newtheorem{Coro}{Corollary}[section]   

\newtheorem{Rem}{Remark}[section]

\newcommand{\R}{\mathbb{R}}
\newcommand{\Z}{\mathbb{Z}}
\newcommand{\C}{\mathbb{C}}
\newcommand{\T}{\mathbb{T}}
\newcommand{\LL}{\mathcal{L}}

\newcommand{\hu}{{\widehat u}}
\newcommand{\h}{\mathfrak{h}}
\newcommand{\1}{\langle} %left arrow bracket
\newcommand{\2}{\rangle} %right arrow bracket

\newcommand{\tb}[2]{\genfrac{}{}{0pt}{}{#1}{#2}} %top/bottom

\newcommand{\sign}{\mathop{\rm sign}}

\newcommand{\re}{\mathop{\rm Re}}

\newcommand{\Wert}{\mathop{\rm Vert}\nolimits}
\newcommand{\bcdot}{\boldsymbol{\cdot}}

\begin{document}

\title{On the analyticity of the nonlinear Fourier transform of the Benjamin-Ono equation on $\T$}   
 
\author{P. G\'erard, T. Kappeler\footnote{T.K.  partially supported by the Swiss National Science Foundation.} 
\,and P. Topalov\footnote{P.T.  partially supported by the Simons Foundation,  Award \#526907.}}
  
\maketitle

\begin{abstract}  
We  prove that  the nonlinear Fourier transform of the Benjamin-Ono equation on $\T$, also referred to as Birkhoff map,
is a real analytic diffeomorphism from the scale of Sobolev spaces $H^{s}_{0}(\T, \R)$, $s > -1/2$, to the scale of weighted $\ell^2-$sequence spaces, 
$\h^{s +1/2}_{r,0}(\mathbb N, \C)$, $s >-1/2$. 
As an application we show that for any $-1/2<s<0$, the flow map of the Benjamin-Ono equation 
$\mathcal{S}_0^t : H^{s}_{0}(\T, \R)\to H^{s}_{0}(\T, \R)$ is {\em nowhere locally uniformly continuous} in 
 $H^{s}_{0}(\T, \R)$.
\end{abstract}    

\smallskip


\smallskip


\tableofcontents

\section{Introduction}\label{Introduction}
In this paper we study the Benjamin-Ono equation on the torus $\T := \R/2\pi\Z$,
\begin{equation}\label{eq:BO}
\partial_t u=H\big(\partial_x^2u\big)-2 u\partial_xu, %\quad u|_{t=0}=u_0,
\end{equation}
where $u\equiv u(x,t)$, $x\in\T$, $t\in\R$, is real valued and $H$ denotes the Hilbert transform
$H : H^\beta_c\to H^\beta_c$, $\beta\in\R$,
\[
\sum_{n\in\Z}\widehat{v}(n) e^{i n x}\mapsto -i\sum_{n\in\Z}\sign(n)\widehat{v}(n) e^{i n x},
\]
with $\sign(n)$ being defined as $\sign(\pm n)=\pm 1$ for $n> 0$, $\sign(0)=0$, and $\widehat{v}(n)$, $n\in\Z$, being
the Fourier coefficients of the element $v$ in the Sobolev space $H^\beta_c\equiv H^\beta(\T,\C)$. 
Equation \eqref{eq:BO} was introduced by Benjamin \cite{Ben1967} and Davis $\&$ Acrivos \cite{DA1967} as 
a model for an asymptotic regime of internal gravity waves at the interface of two fluids. 
It is well known that \eqref{eq:BO} possesses a Lax pair representation (cf. \cite{Na1979}) 
%that leads to an infinite sequence of conserved quantities (cf. \cite{Na1979}, \cite{BK1979})
and that by the use of the Gardner bracket, it can be written in Hamiltonian form.
%(see e.g. \cite{KLM} {\color{red}(Question: reference for $\T$?)} ).  
Equation \eqref{eq:BO} has been extensively studied, in particular its wellposedness in Sobolev spaces
- see \cite{GKT1}, \cite{Saut2019} and references therein. 
%We refer to  \cite{Saut2019} for an excellent survey and a derivation of \eqref{eq:BO}.

Using the Lax pair representation of \eqref{eq:BO}, it was  proved in \cite{GK,GKT1} that 
%for any $s > -1/2$,
the Benjamin-Ono equation on $\T$ admits a nonlinear Fourier transform, also referred to as Birkhoff map. More precisely, according to
\cite[Theorem 1]{GK} and \cite[Theorem 6, Proposition 5]{GKT1}, there exists a map
\begin{equation}\label{eq:Phi-introduction}
\Phi : \bigsqcup_{s > -1/2} H^s_{r,0} \to  \bigsqcup_{s > -1/2} h_{r,0}^{s+1/2}, \ 
u \mapsto \Phi(u) = (\Phi_n(u))_{n \ne 0} \, , 
\end{equation}
$$
\Phi_{-n}(u) =  \overline{ \Phi_n(u)} , \qquad \forall \, n \ge 1 ,
$$
so that for any $s > -\frac12$,  
$\Phi : H^s_{r,0} \to h^{s+1/2}_{r,0}$ is a homeomorphism,
which transforms the trajectories of \eqref{eq:BO} into straight lines. Furthermore,
these trajectories evolve on the isospectral sets of potentials
of the Lax operator (see \eqref{eq:L} below).
%$\Phi$ and its inverse map bounded subsets to bounded ones.
Here $H^\beta_r\equiv H^\beta(\T,\R)$, $\beta\in\R$, is the Sobolev space with exponent $\beta$ of real valued functions
on the torus,
\[
H^\beta_{r,0}:= H^\beta_r \cap H^\beta_{c,0}, \qquad
H^\beta_{c,0}:=\big\{u\in H^\beta_c\,\big| \, \hu(0)=0\big\} ,
\]
and
\begin{equation}\label{eq:h_{r,0}}
\h^\beta_{r,0}:=\big\{z\in\h^\beta_{c, 0}\,\big| \ z_{-n}=\overline{z}_n \ \forall n\ge 1\big\}
\end{equation}
is a real subspace of the complex weighted $\ell^2$-sequence space 
%$\h^\beta_{c,0}$
\begin{equation}\label{eq:h_{c,0}}
\h^\beta_{c,0} =\big\{ (z_n)_{n\ne 0}\,\big| \,
\sum_{n\ne 0} |n|^{2\beta}|z_n|^2 < \infty\big\} .
\end{equation}
The map $\Phi$ can be considered as a nonlinear Fourier transform since it allows to solve \eqref{eq:BO} by quadrature. 
In particular, we used it to prove
that for any $-1/2<s<0$,  \eqref{eq:BO} is globally $C^0$-wellposed in $H^{s}_{r,0}$ (\cite[Theorem 1]{GKT1})
and illposed in $H^{s}_{r,0}$ for $s \le -1/2$ (\cite[Theorem 2]{GKT1}),
 improving previously known wellposedness results (see \cite{Mo, MoP, AH}). 
In addition, the Birkhoff map has been applied to show
the almost periodicity of the solutions of \eqref{eq:BO} and the orbital stability of the traveling waves 
of the Benjamin-Ono equation on $\T$ (see \cite[Theorem 3 and Theorem 4]{GKT1}).
We refer to the components $\Phi_n(u)$, $n \ne 0$, of $\Phi$ as Birkhoff coordinates.

The main goal of the present paper is to prove the following theorem.

\begin{Th}\label{th:Phi}
For any $s>-1/2$, the nonlinear Fourier transform of the Benjamin-Ono equation
\[
\Phi : H^{s}_{r,0}\to\h^{s + 1/2}_{r,0},
\]
introduced in \cite{GK,GKT1}, is a real analytic diffeomorphism.
\end{Th}

Important applications of Theorem \ref{th:Phi} concern the study of perturbations of \eqref{eq:BO} 
(cf.  \cite{KP-book}, \cite{Kuk} (KdV equation), \cite{BKM} and references therein (NLS equation))
and the study of regularity properties of the solution map of  \eqref{eq:BO}. 

In this paper, we apply Theorem \ref{th:Phi} to prove the following result on the regularity of the solution map
of \eqref{eq:BO}. To state it we first have to introduce some more notation. Assume that $s> -1/2$.
For $u_0 \in H^{s}_{r, 0}$, denote by  $t \mapsto u(t) \equiv u(t, u_0)$ the solution of \eqref{eq:BO} 
with initial data $u_0$, constructed in \cite{GKT1}. 
For given $t\in\R$ and $T>0$, consider the {\em flow map}
\[
\mathcal{S}_0^t : H^{s}_{r, 0}\to H^{s}_{r, 0},\quad u_0\mapsto u(t, u_0),
\]
and the {\em solution map}
\[
\mathcal{S}_{0,T} : H^{s}_{r, 0}\to C\big([-T, T], H^{s}_{r, 0}\big),\quad u_0\mapsto u(\cdot, u_0)|_{[-T,T]},
\]
of \eqref{eq:BO}. 
%To state our results, we need one more definition.  
A continuous map $F : X\to Y$ between two Banach spaces $X$ and $Y$ is said to be
{\em nowhere locally uniformly continuous} in an open neighborhood $U$ in $X$ if
the restriction $F\big|_V : V\to Y$ of $F$ to any open neighborhood 
 $V\subseteq U$ is {\em not} uniformly continuous. In a similar way one defines the notion of a {\em nowhere locally Lipschitz} map
in an open neighborhood $U\subseteq X$.

\begin{Th}\label{th:well-posedness}
\begin{itemize}
\item[(i)] For any $-1/2<s<0$ and $t\ne 0$, the flow map $\mathcal{S}_0^t : H^{s}_{r,0}\to H^{s}_{r,0}$
of the Benjamin-Ono equation \eqref{eq:BO} is nowhere locally uniformly continuous in $H^{s}_{r,0}$. 
In particular, $\mathcal{S}_0^t : H^{s}_{r,0}\to H^{s}_{r,0}$ is nowhere locally Lipschitz in $H^{s}_{r,0}$.
\item[(ii)] For any $s\ge 0$ and $T>0$, the solution map $\mathcal{S}_{0,T} : H^s_{r, 0}\to C\big([-T, T], H^s_{r, 0}\big)$ is real analytic.
\end{itemize}
\end{Th}

\noindent{\bf Addendum to Theorem \ref{th:well-posedness}(ii).}\label{rem:improved_well-posedness}
{\em 
%The following complementary version of Theorem \ref{th:well-posedness} (ii) holds 
%(see Remark \ref{rem:improved_dependence_on_the_initial_data}): 
For any $k\ge 1$, $s>-1/2+2k$, and $T > 0$, the solution map
\[
\mathcal{S}_{0,T} : H^s_{r, 0}\to\bigcap_{j=0}^k C^j\big([-T, T], H^{s-2j}_{r, 0}\big)
\]
is well-defined and real analytic in $H^{s}_{r,0}$. 
}

\begin{Rem}
Item (i) of Theorem \ref{th:well-posedness} improves on the  result by Molinet in \cite[Theorem 1.2]{Mo},
saying that for any $s<0$, $t\in\R$, the flow map $\mathcal{S}^t_0$ (if it exists at all) is not of class 
$C^{1, \alpha}$ for any $\alpha>0$.
Item (ii)  of Theorem \ref{th:well-posedness} improves on the result by Molinet, saying that for any $s\ge 0$, $t\in\R$, 
the flow map $\mathcal S^t_0$ is real analytic near zero (\cite[Theorem 1.2]{Mo}, \cite{Mo1}).
\end{Rem}
\begin{Rem}
Any solution $u$ of \eqref{eq:BO} in $H^s_{r, 0}$, $s> -1/2$, constructed in \cite{GKT1}, has the property 
that for any $c\in\R$, $u(t, x - 2 ct) + c$ is again a solution with constant mean value $c$. 
It is straightforward to see that for any $c\in\mathbb R$, Theorem \ref{th:well-posedness} holds on the affine space 
$\big\{u\in H^s_r \,\big| \hu(0)=c\big\}$. 
We remark that the solution map of \eqref{eq:BO} is known to be nowhere locally
uniformly continuous on the Sobolev space $H^s_r$ for {\em any} $s > -1/2$. (A transparent proof
of this fact can be obtained using Birkhoff coordinates -- cf. \cite[Appendix A]{KT1}.)
\end{Rem}
\begin{Rem}
In \cite{GKT3}, we proved a local version of Theorem \ref{th:Phi} and Theorem \ref{th:well-posedness} near $0$. 
The version of Theorem \ref{th:Phi} near $0$ says that for any $s > -1/2$, 
there exists a neighborhood $U$ of $0$ in $H^s_{c,0}$ so that the Birkhoff map $\Phi : U \cap H^{s}_{r,0} \to \h^{s+1/2}_{r,0}$
extends to an analytic diffeomorphism $\Phi : U \to \h^{s+ 1/2}_{c,0}$ onto its image $\Phi(U)$. 
As an application it is shown in \cite{GKT3} that equation \eqref{eq:BO} admits an {\em analytic Birkhoff normal form near $0$}.
%a notion which plays an important role in the theory of finite dimensional integrable systems. Equation \eqref{eq:BO} can thus be viewed 
%as an instance of an integrable system of {\em infinite dimension}, admitting such a normal form. 
%It turns out that the proof 
%of the version of Theorem \ref{th:Phi} near $0$ simplifies since the differential of $\Phi$ at $0$ is a weighted Fourier transform
%and hence a linear isomorphism. At the same time, \cite{GKT3} can be viewed as a preparation of the current paper, since the Vanishing Lemma,
%established in \cite{GKT3}, will also play an important role in the proof of Theorem \ref{th:Phi}.
\end{Rem}

\noindent
{\em Ideas of the proof of Theorem \ref{th:Phi}.} 
The proof of Theorem \ref{th:Phi} 
%uses, in broad terms, the strategy developed in  \cite{GKT3} to prove a local version of Theorem \ref{th:Phi} near $0$, but is considerably more technical, 
 relies  on the approximation of any element in $H^{s}_{r,0}$ by finite gap potentials
and on properties of the spectrum of the Lax operator, associated to finite gap potentials.
To explain all this in more detail, let us first recall the construction of the Birkhoff map $\Phi : H^{s}_{r,0}\to\h^{s+1/2}_{r,0}$
as described in \cite{GK, GKT1}.
%$, $u \mapsto (\Phi_n(u))_{n \ne 1}$,
It is based on the Lax operator $L_w = -i\partial_x - T_w$, an unbounded operator on the Hardy space $H_+$. Here $H_+$ is defined in \eqref{eq:H_+} below
and $T_w$ denotes the Toeplitz operator with symbol $w\in  H^{s}_{r,0}$, $s > -1/2$. 
We refer to Section \ref{Lax operator} below for a review of the terminology and the results about this operator, established
in our previous papers \cite{GK}, \cite{GKT1}, and \cite{GKT2}.
At this point, we only mention that the spectrum of $L_w$ is discrete and consists of a sequence of simple, real  eigenvalues,
bounded from below, $\lambda_0(w) < \lambda_1(w) < \cdots$. In addition, the eigenvalues are separated from each other by a distance of at least one, 
$\gamma_n(w):= \lambda_n(w) - \lambda_{n-1}(w)  -1 \ge 0$ for any $n \ge 1$.
Associated to these eigenvalues are $L^2-$normalized  eigenfunctions $f_n \equiv f_n(w)$, which are uniquely determined by the normalization conditions
\begin{equation}\label{norm'zation f_n}
\sqrt[+]{\kappa_0(w)} := \1 f_0 | 1  \2 > 0, \qquad  \sqrt[+]{\mu_n(w)} := \1 e^{ix} f_{n-1} | f_n \2 > 0, \ \  \forall n \ge 1 .
\end{equation}
The components of the Birkhoff map $\Phi$ are then defined as 
\begin{equation}\label{def Birkhoff real}
\Phi_n(w) =\frac{ \1 1| f_n(w)  \2}{ \sqrt[+]{\kappa_n(w)}} , \qquad
\Phi_{-n}(w) =\overline{\Phi_n(w)} , \qquad \forall n \ge 1,
\end{equation}
where $\kappa_n(w) > 0$, $n \ge 1$, are scaling factors (cf. Section \ref{sec.normalized eigenfunctions}), determined by
\begin{equation}\label{kappa gamma}
| \1 1 | f_n(w) \2 |^2 = \gamma_n(w) \kappa_n(w) , \qquad \forall n  \ge 1.
\end{equation}
%We point out that the normalization conditions  \eqref{norm'zation f_n} are defined inductively.
%It is this fact which complicates establishing that $\Phi$ extends to an analytic map.
For $w = 0$, one has
\[
\lambda_n(0) = n , \qquad f_n(0) = e^{inx} , \qquad \kappa_n(0) = 1 , \qquad  \forall n \ge 0 ,
\]
and hence 
$$
\gamma_n(0) = 0 , \qquad \mu_n(0) = 1,  \qquad \forall n \ge 1.
$$
Note that $ f_n(0)$, $n \ge 0$, is the Fourier basis of the Hardy space $H_+$.
%The proof in \cite{GKT3} of the local version of Theorem \ref{th:Phi} near zero is based on perturbation arguments.
In order to prove Theorem \ref{th:Phi}, we need to show in a first step that for any $w \in H^s_{r,0}$, $s > -1/2$, 
$\Phi$ extends to an analytic map on a neighborhood of $w$ in $H^s_{c,0}$. To this end, we approximate
$w$ by a finite gap potential in $\mathcal U_N$ with $N \ge 1$ sufficiently large. We recall from \cite{GK} that an element $w \in H^s_{r,0}$ 
is said to be a finite gap potential if $ \{ n \ge 1 \, : \, \gamma_n(w) > 0 \}$ is finite and that for any $N \ge 1$, 
$\mathcal U_N$ is the set of finite gap potentials, satisfying $\gamma_N > 0$ and $\gamma_n = 0$ for any $n > N$.
It is shown in \cite{GK} that finite gap potentials are $C^\infty$-smooth and that $\cup_{N \ge 1} \mathcal U_N$ is dense in $H^s_{r,0}$
for any $s > -1/2$. Furthermore, for any $w \in \mathcal U_N$, $N \ge 1$, (cf. \cite{GK}, \cite[Appendix A]{GKT2})
\[
\lambda_n(w) = n , \qquad f_n(w) = g_\infty(w) e^{inx},  \qquad   \forall n \ge N ,
\]
where $g_\infty(w) := e^{i\partial^{-1}_x w}$. Hence up to finitely many eigenfunctions, $f_n(w)$, $n \ge 0$, is the Fourier basis of $H_+$, 
perturbed by $g_\infty(w)$. Note that $g_\infty(w)$ is a $C^\infty$-smooth function, does not depend on $n$, 
and takes values in the circle $\{ z \in \C \, : \, |z| =1 \}$.

For any given $w \ne 0$ in $H^s_{r,0}$, 
we choose $w_N \in \mathcal U_N$ with $N \ge 1$ sufficiently large,  to be the potential, uniquely determined by 
$$
\Phi_n(w_N) = \Phi_n(w), \quad \forall \, 1 \le n \le N, \qquad
\Phi_n(w_N) = 0, \quad \forall \, n > N.
$$ 
Uisng that $f_n(w_N)$, $n \ge 0$, is a perturbed Fourier basis of $H_+$ in the sense described above, we show that one can apply
perturbation theory to prove that $\Phi$ extends to an analytic map on a neighborhood of $w$ in $H^s_{c,0}$.

\smallskip
\noindent
{\em Outline of the proof of Theorem \ref{th:Phi}.}
Let us outline the main steps of the proof of Theorem \ref{th:Phi} in the case where $-1/2 < s \le 0$.
In \cite{GKT2} we proved that for any  $w \in H^s_{r,0}$, $-1/2 < s \le 0$, there exists a neighborhood $U^s\equiv U^s_w$ of $w$ in $H^s_{c,0}$
so that for any $u\in U^s$, the spectrum of $L_u$ consists of simple eigenvalues $\lambda_n(u)$, $n \ge 0$.
These eigenvalues are analytic functions on $U^s$. By shrinking $U^s$, if needed, one infers from standard perturbation theory, 
that for any given $0 \le n_0 < \infty $, the first $n_0$ eigenfunctions $f_n$, $0 \le n < n_0$, extend as $ H^{s+1}_+$-valued analytic functions to $U^s$  
(cf. Lemma \ref{estim2 kappa+mu} below). One of the main difficulties is to show that one can choose  the neighborhood $U^s$  independently of $n$
so that the components of the Birkhoff map extend to $U^s$ with the appropriated decay in $n$ as $|n| \to \infty$.
 For any $n \ge n_0$ with $n_0$ sufficiently large, the Riesz projectors 
$$
P_n : U^s \to \mathcal L(H^s_+, H^{s+1}_+), u \mapsto P_n(u) ,
%H^s_+ := H^s_{c} \cap H_+ ,
$$
onto the one dimensional eigenspace, corresponding to the eigenvalue $\lambda_n$, is analytic and admits a Taylor expansion at $w_N$ in terms of $v= u - w_N$
where $w_N$ is the approximating finite gap potential described above.
See \eqref{eq:H_+} below for the definition of the Hardy spaces $H^\beta_+$, $\beta \in \R$, and
Proposition \ref{prop:L-global} and Proposition \ref{prop:L-local} in Section \ref{Lax operator} for a review of these results.
Choosing $N$ and $n_0 > N$ sufficiently large, one infers that for any $n \ge n_0$, $h_n(u) := P_n(u) f_n(w_N)$  is a nonvanishing analytic function $U^s \to H^{s+1}_+$
and hence an eigenfunction of $L_u$, corresponding to the eigenvalue $\lambda_n(u)$ of $L_u$. In general, the eigenfunctions $h_n(u)$ 
will not satisfy the normalization conditions of the eigenfunctions $f_n(u)$.
In Section \ref{sec:Psi} we introduce the pre-Birkhoff map 
$$
\Psi(u) = (\Psi_n(u))_{n \ge n_0}, \qquad 
\Psi_n(u) := \1 h_n(u) | 1 \2 , \quad \forall n \ge n_0,
$$
where $n_0 > N$ will be chosen sufficiently large.
With the help of the Taylor expansion of $P_n(u)$, $n \ge n_0$, at $w_N$, we show that for any $u \in U^s$, $\Psi(u) \in\h^{s+1}_{\ge n_0}$
and that $\Psi : U \to \h^{s+1}_{\ge n_0}$ is analytic. Here $\h^{s+1}_{ \ge n_0}$ is the sequence space defined in \eqref{def frak h_+} below.  
%Using the estimates for the eigenfunctions $f_n(w_0)$, established in Section  \ref{Lax operator}, we prove that
% for any $s > -1/2$ (and not just for $-1/2 < s \le 0$), there exists a neighorhood $U$ of $w$ in $H^s_{c,0}$ so that 
% the pre-Birkhoff map extends to an analytic map $\Psi : U \to \h^{s+1}_{ \ge n_0}$.
In order to relate the Birkhoff map $\Phi$ to the pre-Birkhoff map $\Psi$, we record in Section \ref{sec.normalized eigenfunctions}
results about the scaling factors $\kappa_n(u)$, $n \ge 0$, and $\mu_n(u)$, $n \ge 1$, introduced in \cite{GK} and further analyzed 
in \cite{GKT2}. 
Using these results together with the ones proved for the pre-Birkhoff map, we then show in Section \ref{sec:the_Birkhof_map} 
that the Birkhoff map $H^s_{r,0} \to \h^{s+1/2}_+$ is real analytic, where $\h^{s+1/2}_+$ is considered 
as a real Hilbert space.\footnote{This convention of the notion of a real analytic map is used throughout the paper.}
An important ingredient is Proposition \ref{prop:delta_n-analyticity} (Vanishing Lemma).
%in \cite{GKT3}.
The delicate and quite technical proof of this proposition is given in Section \ref{sec:the_delta_map}.
Theorem  \ref{th:Phi} is then proved in Section \ref{proof main results}.
To show that $\Phi$ is a local diffeomorphism at a given potential $w \in H^s_{r,0}$ 
%(cf. Theorem  \ref{th:Phi}), 
we use the canonical relations of the Birkhoff coordinates (cf. \cite{GK, GKT1}) to prove that $d_w\Phi$  is onto 
and we approximate $w$ by a sequence of finite gap potentials to show that $d_w\Phi$ is Fredholm.

\smallskip
\noindent
{\em Outline of the proof of Theorem \ref{th:well-posedness}.} Theorem \ref{th:well-posedness} is proved in Section \ref{proof main results}
with the help of Theorem  \ref{th:Phi}. 
%The latter allows to extend the proof of the local version of Theorem \ref{th:well-posedness} near $0$
%in \cite{GKT3} to the entire space $H^s_{r,0}$, $-1/2 < s < 0$.

\smallskip
\noindent
{\em Related work.} 
The result saying that  the Birkhoff map $\Phi : H^{s}_{r,0}\to\h^{s+1/2}_{r,0}$ of the Benjamin-Ono equation
 is a real analytic diffeomorphism for the appropriate range of $s$, shows that similarly as for the KdV equation (cf. \cite{KP-book}, \cite{KT1}) 
and the defocusing nonlinear Schr\"odinger (NLS) equation (cf. \cite{GK-book}), the Benjamin-Ono equation on $\T$ is 
integrable in the strongest possible sense. We point out that the proof of the analyticity of the Birkhoff map in the case of 
the Benjamin-Ono equation significantly differs from the one in the case of the KdV and NLS equations due to the fact 
that the Benjamin-Ono equation is not a partial differential equation.

%It turns out that the analysis of the solution map of the KdV equation, expressed in Birkhoff coordinates, 
%can also be used to prove Theorem  \ref{th:well-posedness} and its extension.

%Similarly as in the case of the KdV equation (cf. \cite{KP-book}, \cite{Kuk}) and of the NLS equation (cf. \cite{BKM} and references therein),  
%a major application of the result, saying that $\Phi : H^{s}_{r,0}\to\h^{s+1/2}_{r,0}$ is a real analytic diffeomorphism for 
%any $s > -1/2$, concerns its use to study (Hamiltonian) perturbations of the Benjamin-Ono equation by KAM methods near finite gap 
%solutions of arbitrary large amplitude. 

\smallskip

\noindent{\em Notation.} In this paragraph we summarize the most frequently used notations in the paper. 
For any $\beta\in\R$, $H^\beta_c$ denotes the Sobolev space $H^\beta\big(\T,\C\big)$
of complex valued functions on the torus $\T = \R/2\pi\Z$ with regularity exponent $\beta$. 
The norm in $H^\beta_c$ is given by 
\[
\|u\|_\beta:=\Big(\sum_{n\in\Z}\1 n\2^{2\beta}|\hu(n)|^2\Big)^{1/2} , \quad \1 n \2 := \max\{1,|n|\} ,
\]
where $\hu(n)$, $n\in\Z$, are the Fourier coefficients of $u\in H^\beta_c$.
For $\beta=0$ and $u\in H^0_c\equiv L^2(\T,\C)$ we set $\|u\|\equiv \|u\|_0$.
By $H^\beta_{c,0}$ we denote the complex subspace in $H^\beta_c$ of functions with mean value zero,
\[
H^\beta_{c,0}=\big\{u\in H^\beta_c\,\big|\,\hu(0)=0 \big\}.
\]
For $f\in H^\beta_c$ and $g\in H^{-\beta}_c$, define the sesquilinear and bilinear pairings,
\[
\1 f|g\2:=\sum_{n\in\Z}\widehat{f}(n)\overline{\widehat{g}(n)} ,
\qquad \quad
\1 f , g\2:=\sum_{n\in\Z}\widehat{f}(n)\widehat{g}(-n).
\]
The positive Hardy space $H^\beta_+$ with regularity exponent $\beta\in\R$ is defined as
\begin{equation}\label{eq:H_+}
H^{\beta}_+:=\big\{f\in H^{s}_c\,\big|\,\widehat{f}(n)=0\,\,\,\forall n<0\big\} , \qquad
H_+ \equiv H^0_+ .
%H^\beta_-:=\big\{f\in H^\beta_c\,\big|\,\widehat{f}(n)=0\,\,\,\forall n>0\big\}.
\end{equation}
For any $1\le p<\infty$, denote by
$
\ell^p_+\equiv\ell^p\big(\Z_{\ge 1},\C\big)
$
the Banach space of complex valued sequences $ z=(z_n)_{n\ge 1} $ with finite norm
\[
\| z \|_{\ell^p_+} :=\Big(\sum_{n\ge 1}|z_n|^p\Big)^{1/p}<\infty
\]
and by $\ell^\infty_+\equiv\ell^\infty\big(\Z_{\ge 1},\C\big)$ the Banach space of  
complex valued sequences with finite supremum  norm $ \| z\|_{\ell^\infty_+} :=\sup_{n\ge 1}|z_n|$. 
More generally, for $1\le p\le\infty$ and $m\in\Z$, we introduce the Banach space
\[
\ell^p_{\ge m}\equiv\ell^p\big(\Z_{\ge m},\C\big),\qquad \Z_{\ge m}:=\{n\in\Z\,| \, n\ge m\},
\]
as well as the space $\ell^p_c\equiv\ell^p\big(\Z,\C\big)$, defined in a similar way.
Furthermore, denote by $\h^\beta_{ \ge m}$, $\beta\in\R$, $m \in \Z$, the Hilbert space 
of complex valued sequences $ z = (z_n)_{n \ge m}$ with 
\begin{equation}\label{def frak h_+}
\|  z\|_{\h^\beta_{ \ge m}} :=\Big(\sum_{n\ge m} \1 n \2 ^{2\beta}|z_n|^2\Big)^{1/2}<\infty.
\end{equation}
For convenience of notation, we often write $\| z\|_\beta$ for $\|  z\|_{\h^\beta_{ \ge m}}$
and in the case where $m=1$,  $\h^\beta_+$ for  $\h^\beta_{\ge 1}$.  \\
Finally, for $z\in\C\setminus(-\infty,0]$, 
we denote by $\sqrt[+]{z}$ or $\sqrt{z}$ the principal branch of the  square root of $z$, defined by
$\re\big(\sqrt[+]{z}\big)>0$.

%%%%%%%%%%%%%%%%%%%%%%%%%%%%%%%%%%%%%%%%%%%%%%%%%%%%%%%
%%%%%%%%%%%%%%%%%%%%%%%%%%%%%%%%%%%%%%%%%%%%%%%%%%%%%%%
\section{Premlinaries}\label{Lax operator}
In this section we review results about the Lax operator $L_u$ of the Benjamin-Ono equation,
needed in this paper.
Most of these results were established in \cite{GKT2}.
%for potentials $u$ in $H^{s}_c$ with $ -1/2 < s \le 0$, which were established in \cite{GKT2}.
%In the second part we prove bounds for the eigenfunctions $f_n(w)$, $n \ge 1$,
%of $L_w$ for potentials $w \in H^s_{r,0}$ with $s > -1/2$. 
We point out that in the context of this paper, it turned out to be advantageous to alter some of the notations
introduced in \cite{GKT2}.
%Throughout this section we assume that  $-1/2 < s \le 0$.

\smallskip

For $u\in H^{s}_c$ with $-1/2 < s \le 0,$ consider the pseudo-differential expression
\begin{equation}\label{eq:L}
L_u:=D-T_u , \qquad D = -i\partial_x ,
\end{equation}
acting on the Hardy space $H_+^s$, where 
$T_u$ is the Toeplitz operator with symbol $u$, $T_u f:=\Pi(u f)$,
and $\Pi\equiv\Pi^+ : H^{\beta}_c \to H^{\beta}_+$, $\beta \in \R$, is the Szeg\H o projector
\[
\Pi : H^{s}_c\to H^{s}_+,\quad \sum_{n\in\Z}\widehat v(n) e^{i n x}\mapsto\sum_{n\ge 0}\widehat v(n) e^{i n x} .
\]
It follows from Lemma 1 in \cite{GKT2} that $L_u$ defines an operator in $H^{s}_+$ with domain
$H^{s+1}_+$ so that the map $L_u : H^{s+1}_+\to H^{s}_+$ is bounded.
The following result follows from \cite[Theorem 1]{GKT2} and \cite{GKT1}.

\begin{Prop}\label{prop:L-global}
For any $-1/2 < s \le 0$, there exists an open neighborhood $W\equiv W^{s}$ of $H^{s}_{r,0}$ in $H^{s}_{c,0}$ so that 
for any $u\in W$, the operator $L_u$ is a closed operator in $H^{s}_+$ with domain $H^{s+1}_+$. 
The operator has a compact resolvent and all its eigenvalues are  {\em simple}. When appropriately listed,  $\lambda_n(u)$, $n\ge 0$, 
 satisfy 
 $$
 \re(\lambda_n(u))<\re(\lambda_{n+1}(u)), \quad \forall n \ge 0, \qquad \quad  \lim_{n \to \infty} |\lambda_n(u)-n| = 0.
 $$ 
For $u\in H^{s}_{r,0}$, the eigenvalues are real valued and 
$\gamma_n(u):= \lambda_n(u)-\lambda_{n-1}(u)-1\ge 0$ for $n\ge 1$.
\end{Prop}

In addition to the information provided by Proposition \ref{prop:L-global}, we need rather precise estimates on the spectrum of $L_u$,
established in \cite{GKT2}. 
First, we record estimates of the eigenfunctions 
of the Lax operator $L_w$ for real valued potentials  $w \in H^s_{r,0}$ with $s > -1/2$.  
Recall that in \cite{GK} ($s=0$), \cite{GKT1} ($-1/2 < s < 0$), we introduced for any $w \in H^s_{r,0}$, $-1/2 < s \le 0$, 
the $L^2$-normalized eigenfunctions $f_n \equiv f_n(w)$ , $n \ge 0$, of $L_w$, corresponding to the eigenvalues $\lambda_n(w)$, uniquely determined by
 the normalization conditions (cf. \eqref{norm'zation f_n})
\begin{equation}\label{norm'zation f_n version2}
 \1 f_0 | 1  \2 > 0, \qquad   \1 e^{ix} f_{n-1} | f_n \2 > 0, \ \  \forall n \ge 1 .
\end{equation}
The eigenfunctions $f_n$, $n \ge 0$, are in $ H^{s+1}_+ $ and hence
$$
g_n(w) := f_n(w) e^{-inx} \in H^{1+1}_c, \qquad \forall \, n \ge 0. 
%\ \forall w \in H^s_{r,0}, -1/2 < s \le 0 .
$$
By Proposition 9 in \cite[Appendix A]{GKT1},  $g_\infty(w) := e^{i\partial_x^{-1} w}$ is in $H^{s+1}_c$.
Furthermore, $g_n \equiv g_n(w)$, $n \ge 0$, and $g_\infty \equiv g_\infty(w)$ satisfy the equations (cf. \cite[(80), (82)]{GKT1})
\begin{equation}\label{equ for g_n, g_infty}
Dg_n = (\lambda_n - n)g_n + \Pi_{\ge -n}(wg_n)  , \qquad
Dg_\infty = w g_\infty ,
\end{equation}
where for any $k \in \Z$, $\Pi_{\ge k}$ denotes the projection 
$$
H^\beta_c \to H^\beta_c, \, h = \sum_{n \in \Z} \widehat h(n) e^{inx} \mapsto  \Pi_{\ge k}(h) =  \sum_{n \ge k} \widehat h(n) e^{inx} .
$$
Our first result concerns the regularity of the functions $g_n(w)$ and $g_\infty(w)$
for $w \in H^s_{r,0}$ with $s > -1/2$ arbitrary. 

\begin{Lem}\label{regularity g_n}
For any $w \in H^s_{r,0}$ with $s > -1/2$, 
$$
g_n(w) \in H^{s+1}_c , \quad \forall n \ge 0,  \qquad  g_\infty(w) \in H^{s+1}_c .
$$
\end{Lem}
\begin{proof}
The claimed statements follow in a straightforward way
from the discussion above, \eqref{equ for g_n, g_infty}, and well known results on the multiplication of elements in Sobolev spaces,
recorded below in a form convenient for our purposes.
\end{proof}
The following lemma is well known  (cf. e.g. \cite[Lemma 1]{GKT2}).
\begin{Lem}\label{multi of functions}
For any $s > -1/2$, there exists a constant $C_{s, 1} \ge 1$, only depending on $s$, so that the following holds:\\
(i) For any $s >  -1/2$, 
\begin{equation}\label{multi of functions 1}
\| f g\|_s \le C_{s,1} \|f \|_s \| g \|_{1 + s - \tau} , \qquad \forall \, f \in H^s_c, \ g \in H^{s+ 1 - \tau }_c ,
\end{equation}
where $\tau:= \min \{ \frac 12 (\frac 12 + s), \frac 14 \}$. \\
(ii) For any $ s > 1/2$,
\begin{equation}\label{multi of functions 2}
\| f g\|_s \le C_{s, 1} \|f \|_s \| g \|_s , \qquad \forall \, f, g \in H^s_c .
\end{equation}
We choose $C_{s,1}$ in such a way that it is nondecreasing in $s$ on the interval $[0, \infty)$.
\end{Lem}
\begin{Rem}\label{multi of functions 3}
For any $ s_1, s_2 \ge 0$ and any $s \le \min(s_1, s_2) $ with $s_1 + s_2 - s> 1/2$,
$$
H^{s_1}_c \times H^{s_2}_c \to H^s_c , \,  (f_1, f_2) \mapsto f_1 f_2
$$
is a bounded bilinear map.
\end{Rem}
%\begin{Rem}\label{def tau}
%We remark that for any $-\frac 12 < s \le \frac 12$, 
%$$
%\frac 34 + \frac{s}{2} = 1+s - \tau(s) > 1/2 , \qquad \tau(s)  := (\frac 12 + s)/2 > 0.
%$$ 
%\end{Rem}
%\begin{proof} In view of the discussion above it remains to treat the case where $s > 0$.
%Since the proofs of the claimed statements for $ g_n(w)$ and $g_\infty(w)$ are similar, we concentrate on $g_n \equiv g_n(w)$ only.
%Let us fix $n \ge 0$.
%First we consider the case where $0 < s \le 1$. Since $w \in L^2_{r,0}$ and hence $g_n \in H^1_c$ by the considerations above,
%it follows that $wg_n \in H^s_c$ and in turn $\Pi_{\ge -n}(wg_n)\in H^s_c$. Hence by \eqref{equ for g_n, g_infty} $Dg_n \in H^{s}_c$
%and thus $g_n \in H^{s+1}_c$. To prove the case $s > 1$, we argue by induction. Assume that for any given integer $k \ge 0$
%and any $w \in H^k_{r,0}$, $g_n$ is in $H^{k+1}_{c}$. We then claim that for any $w \in H^s_{r,0}$ with $k < s \le k+1$,
%$g_n$ is in $H^{s+1}_c$. Indeed, by \eqref{equ for g_n, g_infty}, $Dg_n$ is in $H^s_c$ and hence $g_n \in H^{s+1}_c$.
%\end{proof}

In the sequel, we need bounds of $\|g_n(w)\|_{1+s}$ in terms of $\| w \|_s$, which are uniform in $n\ge 0$.
\begin{Lem}\label{estimates g_n}
For any $s > -1/2$, there exists a constant $C_{s, 2} \ge 1$, only depending on $s$, so that for any  $n \ge 0$ and any $w \in H^s_{r,0}$,
$$
\|g_n(w)\|_{1+s} \le C_{s,2} (1+ \|w\|_s)^{\eta(s)}, \qquad   \|g_\infty(w) \|_{1+s} \le C_{s,2} (1+ \|w\|_s)^{\eta(s)}\, ,
$$
where 
\begin{equation}\label{def eta(s)}
\eta(s) := \begin{cases} 
\frac{4(1+s)}{1+2s} \ \  \  \forall  -\frac 12 < s \le 0 \\
4 \qquad  \quad  \forall \, 0 < s \le 1\\
2 + 2s \quad \  \forall \, s > 1 \, .
\end{cases}
\end{equation}
We choose $C_{s,2}$ in such a way that it is nondecreasing in $s$ on the interval $[0, \infty)$.
\end{Lem}
\begin{Rem}
We have chosen the exponant $\eta(s)$ to be nondecreasing in $s$ on the interval $[0, \infty)$. This is needed for our scheme of proof of Theorem \ref{th:Phi}, 
whereas the size of $\eta(s)$ is of no importance for our purposes.  
\end{Rem}
\begin{proof} The proof of the claimed estimates of $\|g_n(w)\|_{1+s}$ and $\| g_\infty(w)\|_{1 +s}$ are similar,
hence we concentrate on the ones of $\|g_n(w)\|_{1+s}$. 
For notational convenience, we do not indicate the dependence  of quantities such as $\lambda_n$ and $g_n$ on $w$ in this proof
and we fix $n \ge 0$.

\noindent
{\em Case $-1/2 < s \le 0$.} First note that for any $w \in H^s_{r,0}$, 
\begin{equation}\label{L^2 norm}
\|g_n \|_s \le \| g_n \| = \|f_n e^{-inx} \| = 1, \qquad
| \1 g_n | 1 \2 | \le \|g_n\| = 1.
\end{equation}
By \eqref{equ for g_n, g_infty}, one has
%and Lemma \ref{regularity g_n} 
$$
\| Dg_n \|_s \le |\lambda_n - n | \|g_n \|_s + \|wg_n \|_s \, .
$$ 
Hence using \eqref{L^2 norm} and the estimate
$n - \lambda_n  = \sum_{k \ge n+1} \gamma_k \le  |\lambda_0|$ (cf. \cite[(3.13)]{GK}, \cite[Section 3]{GKT1}),
it follows that
\begin{equation}\label{1 norm g_n}
\|g_n\|_{1+s} \le 1 + \|Dg_n\|_s \le 1 + |\lambda_0| + \| wg_n\|_s .
\end{equation}
By Lemma \ref{multi of functions} with $C_1:=  C_{s, 1}$, 
$$
\| wg_n \|_s \le C_{1} \|w\|_s \|g_n\|_{1 + s - \tau}, \qquad \tau \equiv \tau(s) = (\frac 12 + s)/2 .
$$
Interpolating between $L^2_c$ and $H^{s+1}_c$, and since $\|g_n\|_0 = 1$, one obtains
$$
\|g_n\|_{1+s-\tau} \le C_2 (\|g_n\|_{1+s})^{1 - \frac 1q} , \qquad \frac1q := \frac{\tau}{1+s} ,
$$
where $C_2 \ge 1$ is a constant, only depending on $s$.
By Young's inequality, 
$$
\| wg_n \|_s \le C_1 C_2 \|w\|_s ( \|g_n\|_{1+s})^{1 - \frac 1q} \le \frac1q (C_1C_2 \|w\|_s)^q + (1- \frac 1q) \|g_n\|_{1+s}
$$
and in turn
$$
\| g_n \|_{1+s} \le q + q |\lambda_0|  + (C_1C_2 \|w\|_s)^q .
$$
By \cite[(19), (31)-(32)]{GKT1}, there exists $C_3 \ge 1$, only depending on $s$, so that 
$$
|\lambda_0| \le 1 +C_3 (1 + \|w\|_s)^{\frac {2}{1 + 2s}} .
$$
(For $s=0$, one has by the trace formula $|\lambda_0| \le \sum_{k \ge 1} (k \gamma_k)) = \frac 12 \|w\|$ (cf. \cite[Proposition 3.1]{GK}.)
Since $\frac {2}{1 + 2s} < 4\frac{1+s}{1+2s} =q$ and $q = \eta(s)$, 
we thus conclude that there exists a constant $C_4 \ge 1$, only depending on $s$,
so that
$$
\|g_n \|_{1+s} \le C_4 (1 + \|w\|_s)^{\eta(s)} .
$$
\noindent
{\em Case $0 < s \le1/2$.} Note that by the trace formula, mentioned above,
one has $| \lambda_0 | \le \frac 12 \|w\|^2$  and $g_n \in H^1_c$. Hence by \eqref{equ for g_n, g_infty},
\begin{equation}\label{2 norm g_n}
\|g_n\|_{1+s} \le 1 + \frac 12 \|w\|^2 \|g_n\|_s + \| wg_n\|_s .
\end{equation}
By Remark \ref{multi of functions 3} (use  $\frac{1+s}{2} > \frac 12$ for $0 < s \le 1/2$), there exists $C_5 \ge 1$, only depending on $s$, so that
$$
\|wg_n \|_s
 \le C_5 \|w\|_s \|g_n\|_{\frac{1+s}{2}} .
 $$
 Interpolating between $L^2_c$ and $H^{s+1}_c$, and since $\|g_n\| = 1$, there exists $C_6 \ge 1$, only depending on $s$, so that
$$
 \|g_n\|_{\frac{1+s}{2}} \le  C_6 (\|g_n\|_{1+s})^{1/2} 
$$ 
and, since $0 < s \le 1/2$, $ \|w\|^2 \|g_n\|_s \le \|w\|_s^2 \|g_n\|_{\frac{1+s}{2}}$, one obtains
\begin{align*}
\|g_n\|_{1+s}&\le 1 + \big(\frac 12 \|w\|^2_s + C_5 \|w\|_s \big) \|g_n\|_{\frac{1+s}{2}}\\
&\le 1+ C_6\big(\frac {1}{2} \|w\|^2_s + C_5 \|w\|_s\big) \big(\|g_n\|_{1+s}\big)^{1/2}
\end{align*}
 and in turn
 $$
 \begin{aligned}
  \|g_n\|_{1+s} & \le 1 + \big(\frac 12 \|w\|^2_s + C_5 \|w\|_s \big) \|g_n\|_{\frac{1+s}{2}} \\
& \le 1 + \frac 12 C_6^2 \big( \frac {1}{2} \|w\|^2_s + C_5 \|w\|_s\big)^2
 + \frac 12 \|g_n\|_{1+s} .
 \end{aligned}
 $$
We conclude that there exists a constant $C_7 \ge 1$, only depending on $s$, so that
$$
\|g_n \|_{1+s} \le C_7(1 +  \|w\|_s)^4 .
$$ 
\noindent
{\em Case $ s >1/2$.} By \eqref{equ for g_n, g_infty} and Lemma \ref{multi of functions} (ii) with $C_1= C_{s, 1}$
$$
\| g_n\|_{1+s} \le 1 + \big( \frac12 \|w\|_s^2  +  C_1 \|w\|_s \big) \|g_n\|_s , \qquad  .
$$
Interpolating between $L^2_c$ and $H^{s+1}_c$, it follows that there exists a constant $C_8 \ge 1$, only depending on $s$, so that
$\| g_n \|_s \le C_8 \|g_n \|_{1+s}^{\frac{s}{1+s}}$. By Young's inequality one then infers that
$$
\| g_n \|_{1+s} 
 \le 1 + \frac{1}{1+s}(C_8 \frac{1}{2} \|w\|_s^2  + C_8 C_1 \|w \|_s)^{1 + s} + \frac{s}{1+s} \| g_n \|_{1+s} 
  $$
and in turn concludes  that there exists a constant $C_{9} \ge 1$, only depending on $s$, so that 
$$
\| g_n \|_{1+s} \le C_{9}(1 + \|w\|_s)^{2 + 2s} .
$$
The claimed estimates now follow by setting $C_{s,2} := \max (C_4, C_7, C_{9})$ ($\forall \, -1/2 < s \le 0$)
and $C_{s,2} := \sup\{C_4, C_7, C_9, C_{s'} : \, s' < s\}$ ($\forall \, s > 0$).
\end{proof} 

To state the refined estimates on the spectrum of $L_u$ we introduce some more notation.
For any $M > 0$ and $s > -1/2$, let
\begin{equation}\label{def C_{M,s}}
C_{M, s} := \rho \, C_{s,1}^3 C_{s,2}^2 (2 +M)^{2\eta(s)} ,  \qquad  \rho\equiv \rho_{M,s}:= 8 C_{s,2}^2(2 + M)^{2\eta(s)} \ge 8\cdot 2^8 ,
\end{equation}
where $\eta(s)$ is given by \eqref{def eta(s)}, $C_{s,1}$ by Lemma \ref{multi of functions},  
and $C_{s,2}$ by Lemma \ref{estimates g_n}. Note that $C_{M,s}$, $\rho_{M,s}$, and $8C_{M,s}/ \rho_{M,s} = C^3_{s,1}$
are nondecreasing functions in $s$ on the interval $[0, \infty)$. Furthermore,
\begin{equation}\label{estimates constants}
C_{s,1} \ge 1 \, , \qquad  C_{s,2} \ge 1 \, , \qquad  C_{M,s} \ge 1 \, .
\end{equation}
For any $w \in H^s_{r,0} \setminus \{ 0 \}$, $-1/2 < s \le 0$, and $N \ge 1$ with $\gamma_N(w) > 0$, denote by $w_N$
the element in $\mathcal U_N$, defined by 
\begin{equation}\label{def w_N}
\Phi_n(w_N) = \Phi_n(w), \quad \forall n \le N, \qquad  \Phi_n(w_N) = 0, \quad \forall n > N.
\end{equation}
For any $M > 0$ and $w \in  H^s_{r,0}$ with $0 < \|w\|_s \le M$, 
we choose $N \equiv N(w) \ge 1$ so large that
\begin{equation}\label{choice N}
\| w - w_N \|_{s} \le \frac{1}{2 \, C_{M, s}}
\end{equation}
where $C_{M, s}$ is the constant introduced in \eqref{def C_{M,s}}.
% It satisfies $C_{M,s} \ge 4 C_{M, s, 1/4}$ with $C_{M, s, 1/4}$ being the constant introduced in \cite[(31)]{GKT2}.

Finally, for any $v \in H^s_{c, 0}$, we denote by $B^{s}_{c,0}(v, r)$ the ball in $H^s_{c,0}$
of radius $r > 0$, centered at $v$, Furthermore, 
 for any $r > 0$ and $n \ge 0$, we introduce the following subsets of $\C$,
\begin{equation}\label{eq:Vert}
\Wert_n(r):=\big\{\lambda\in\C\,\big|\,|\lambda-n|\ge r,\,|\re(\lambda)-n|\le 1/2\big\}, \qquad \qquad
\end{equation}
\begin{equation}\label{def:D_n}
D_n(r): =\big\{\lambda\in\C\,\big|\,|\lambda-n|<r \big\} , \qquad \qquad \qquad \qquad \qquad \qquad \quad
\end{equation}
and denote by $\partial D_n(r)$ the counterclockwise oriented boundary of $D_n(r)$ in $\C$ 
and by $\Wert_n^0(r)$ the interior of $\Wert_n(r)$ in $\C$.

Arguing as in the proof of Lemma 16 in \cite{GKT2}, one gets the following
\begin{Lem}\label{Lemma 16 GKT2} 
Let $-1/2 < s \le 0$ and $M>0$. Then for any $w \in H^s_{r,0}$ 
with $0 < \|w\|_s \le M$, $w_N$ defined as in \eqref{def w_N},
and $N$ chosen so that 
$\| w - w_N\|_s < 1/ 2C_{M,s}$ (cf. \eqref{choice N}), there exists $n_0 > N$ so that for any
$v \in H^s_{c,0}$ and $n \ge n_0$,
$$
\sup_{\lambda \in \rm{Vert}_n(1/4)} \| T_v (L_{w_N} - \lambda)^{-1}\|_{H^{s;n}_+ \to H^{s;n}_+} \le \frac{8}{\rho} C_{M, s} \|v\|_s ,
$$
where $H^{s;n}_+$ is the Hilbert space $H^s_+$, endowed with the inner product, associated with
the shifted norm $\| \cdot \|_{s; n}$.
\end{Lem}
\begin{Rem}
We recall that the 
shifted norm is defined as
$$
\|f\|_{s;n}:=\big(\sum_{k\ge 0}\1n-k\2^{2s}|\widehat{f}(k)|^2\big)^{1/2} .
$$
It is equivalent to the standard norm in $H^s_+$ (cf. \cite[Lemma 3]{GKT2}). 
\end{Rem}

%and for any $\tau \in \C$, $\nu, \nu' > 0$, 
% $$ 
% {\rm Vert}_\tau(r; \nu, \nu'):= \{ \lambda \in \C \, : \, |\lambda - \tau| \ge r; \,\,  
% \tau - \nu \le \Re \lambda  \le \tau + \nu' \} \, ,
% $$
% $$
%  {\rm Vert}_\tau(r;  \infty, \nu'):= \{ \lambda \in \C \, : \, |\lambda - \tau| \ge r; \,\, 
%  \Re \lambda  \le \tau + \nu' \}\, , \qquad \quad
%  $$

The Counting Lemma  \cite[Theorem 4 with $\rho = 1/4$]{GKT2} then implies the following

 \begin{Prop}\label{prop:L-local}
 Let $- 1/2 < s \le 0$ and $M > 0$. For any $w \in H^{s}_{r,0}$ with $0 < \|w\|_{s} \le M$, 
 choose $w_N \in \mathcal U_N$ as in \eqref{def w_N} with $N\ge 1$ satisfying \eqref{choice N}, 
 and $n_0 > N$ as in Lemma \ref{Lemma 16 GKT2}. Then there exists a neighborhood
 $U^s \equiv U^{s}_w$ of $w$ in $H^{s}_{c,0}$ with $U^s \subset B^{s}_{c,0}(w, 1/2C_{M, s})$ $( \subset B^{s}_{c,0}(w_N, 1/C_{M, s}))$ and the following properties:
 
 \noindent 
(i) For any $u\in U^s$, the operator $L_u$ is a closed operator in $H^{s}_+$ with domain $H^{s+1}_+$
and has a compact resolvent. The spectrum ${\rm spec}(L_{u})$ of $L_{u}$ is discrete and consists of simple eigenvalues only.
%  and for any $ n \ge 0 $, one has
% $$
%  \#\big({\rm spec}(L_{u}) \cap D_{\tau_n}(r_n+ 1/4) \big) = 1\,
%  $$
%  and
% $$
%   {\rm spec}(L_{u}) \cap {\rm Vert}_{\tau_n}(r_n + 1/4; \nu_{n}, \nu_{n+1}) = \emptyset . 
%$$
% Here $\tau_0 := \lambda_0(w)$, $\nu_{0} =  \infty$, $r_0:= 0$  and for any $ 1 \le n <  n_0 $,
%$$
% r_n := \gamma_n(w)/2, \quad  \tau_n := \lambda_n(w) - \gamma_{n}(w)/2 \, , \quad \nu_n := ( \tau_{n} - \tau_{n-1})/2\, ,
%$$
% whereas
%$$
% r_{n_0} := 0\, , \quad \tau_{n_0 } := n_0 \, , \quad  \nu_{n_0} := \frac 12 \big( 1 + \gamma_{n_0 - 1}(w)/2 + \sum_{k \ge n_0 } \gamma_k(w) \big) \, ,
%$$  
%and for any $n \ge n_0 + 1$, 
%$$
%  r_n:= 0\, , \qquad  \tau_n := n\, , \qquad \nu_n := 1/2 \, . \qquad \qquad
%$$

\noindent
(ii) For any $n \ge n_0$ and any $u \in U^s$,
%$B^{s}_{c,0}(w_N, C_{M, s})$
$$
  \#\big({\rm spec}(L_{u}) \cap D_{n}(1/4) \big) = 1 , \qquad 
   {\rm spec}(L_{u}) \cap {\rm Vert}_{n}(1/4) = \emptyset . 
$$
Furtherover, for any $n\ge n_0$, the resolvent map 
\[
U^s \times\Wert_n^0(1/4)\to\LL\big(H^{s}_+,H^{s+1}_+\big),\quad(u,\lambda)\mapsto(L_u-\lambda)^{-1},
\]
is analytic and
\begin{equation}\label{L_u as a perturbation}
(L_u-\lambda)^{-1}=(L_{w_N} - \lambda)^{-1} \big({\rm{Id}} -T_v(L_{w_N} -\lambda)^{-1}\big)^{-1} , \quad v: = u - w_N .
\end{equation}
The Neumann series 
\begin{equation}\label{eq:neumann_series}
\big({\rm Id} -T_v(L_{w_N}-\lambda)^{-1}\big)^{-1}=\sum_{m\ge 0} \big(T_v( L_{w_N}-\lambda)^{-1}\big)^m
\end{equation}
converges 
uniformly on $U^s\times\Wert_n(1/4)$ with respect to the operator norm, induced by the shifted norm $\|\cdot\|_{s;n}$ in $H^s_+$
as well as with respect to the standard operator norm.
\end{Prop}
Proposition \ref{prop:L-local} yields the following 
\begin{Coro}\label{def P_n} 
 Let $- 1/2 < s \le 0$,  $M > 0$, and $w \in H^s_{r,0}$ with $0 < \|w\|_s \le M$.
Assume that  $U^s \equiv U^s_w$ and $n_0 > N$ are given as in Proposition  \ref{prop:L-local}.
Then  for any $u \in U^s$ and $ n\ge n_0$, the Riesz projector $P_n(u)$ is well-defined,
\begin{equation}\label{eq:P_n-local}
P_n(u):=
-\frac{1}{2\pi i}\oint_{\partial D_n}(L_u-\lambda)^{-1}\,d\lambda\in\LL\big(H^{s}_+,H^{s+1}_+\big),
\quad D_n := D_n(1/3), 
\end{equation}
and the map
\begin{equation}\label{P_n anal near 0}
P_n : U^s \to \LL\big(H^{s}_+,H^{s+1}_+\big),\quad u\mapsto P_n(u).
\end{equation}
is analytic.
\end{Coro}

%%%%%%%%%%%%%%%%%%%%%%%%%%%%%%%%%%%%%%%%%%%%%%%%%%%%%%%%%%%%%%%%%%%%%%%%%%%
%%%%%%%%%%%%%%%%%%%%%%%%%%%%%%%%%%%%%%%%%%%%%%%%%%%%%%%%%%%%%%%%%%%%%%%%%%%

\section{Analytic extension of the pre-Birkhoff map}\label{sec:Psi}
In this section we introduce the pre-Birkhoff map and study its properties. 
Throughout the section, we assume that $M >0$, $s >-1/2$, and 
$w$ is a potential in $H^s_{r,0}$, satisfying $0 < \|w\|_s \le M$.
The finite gap potential $w_N$ is defined as in \eqref{def w_N} 
and $N$ chosen so that (cf. \eqref{choice N})
\begin{equation}\label{def w_N for s}
\| w - w_N\|_s < \frac{1}{2C_{M,s} }
\end{equation}
where $C_{M,s}$ is given by \eqref{def C_{M,s}}.
Furthermore,  define
\begin{equation}\label{def sigma}
\sigma\equiv \sigma(s) := \min (s, 0)
\end{equation}
and let $n_0 > N$ be given as in Lemma \ref{Lemma 16 GKT2} so that for any $v \in H^\sigma_{c, 0}$,
$$
\sup_{\lambda \in {\rm Vert}_n(1/4)} \| T_v (L_{w_N} - \lambda)^{-1}\|_{H^{\sigma; n}_+ \to H^{\sigma; n}_+} \le \frac{8}{\rho_{M, \sigma}} C_{M, \sigma} \|v\|_\sigma .
$$
By Corollary \ref{def P_n}, for any $s>-1/2$, $u\in U^\sigma$, and $n\ge n_0$,
the Riesz projector
\begin{equation*}
P_n(u)=
-\frac{1}{2\pi i}\oint\limits_{\partial D_n}(L_u-\lambda)^{-1}\,d\lambda\in\LL\big(H^{\sigma}_+,H^{\sigma+1}_+\big) ,
\end{equation*}
is well-defined and the map $U^\sigma \to \LL\big(H^{\sigma}_+,H^{\sigma+1}_+\big)$, $u\mapsto P_n(u)$, is analytic. 
Here $U^\sigma \equiv U^\sigma_w$ is the open neighborhood of $w$ in $H^{\sigma}_{c,0}$, given by Corollary \ref{def P_n}
and $\partial D_n$  is the counterclockwise oriented boundary of $D_n$, defined as (cf. \eqref{def:D_n})
$$
D_n\equiv D_n(1/3) \, .
$$
Hence, for any $n\ge n_0$, the map
\begin{equation}\label{eq:Psi_n-map}
U^\sigma \to H^{\sigma+1}_+,\quad u\mapsto h_n( u),
\end{equation}
is analytic where
\begin{equation}\label{eq:h_n}
h_n( u):=P_n(u) f_n\in H^{\sigma+1}_+ ,\qquad f_n := f_n(w_N) .
\end{equation}
where $w_N$ is the finite gap potential, defined as in \eqref{def w_N for s}.

We point out that in the course of our arguments, the neighborhood might be shrunk.
The main objective in this section is to prove the following 
\begin{Prop}\label{prop:Psi}
For any $w \in H^s_{r,0}$ with $s>-1/2$ and $0 < \|w\|_s \le M$, there exists an open neighborhood 
$ U^{s} \equiv U^s_w$ of $w$ in $H^{s}_{c,0}$ so that the following holds:
\begin{itemize}
\item[(i)] The map
\begin{equation}\label{eq:Psi}
\Psi : U^s \to  \h^{s+1}_{\ge n_0}, \quad u \mapsto \big(\1 h_n(u) | 1\2\big)_{n\ge n_0},
\end{equation}
is analytic where $h_n(u)$, $n \ge n_0$ are given by \eqref{eq:h_n} and $n_0$ by Corollary \ref{def P_n}. 
We refer to $\Psi$ as the (local) pre-Birkhoff map near $w$.
\item[(ii)] For any $n\ge n_0$, the map $U^\sigma \to\C$, $u\mapsto  \1 h_n( u) | f_n(w_N) \2$,
is analytic and 
\[
\big|\1 h_n( u) | f_n(w_N)  \2  - 1 \big| \le 1/2, \qquad \forall \, n \ge n_0,  \ u \in U^\sigma .
\] 
\end{itemize}
The neighborhood $U^s$ can be chosen as $ U^\sigma_w \cap B^s_{c,0}(w_N, 1/C_{M,s})$
where $U^\sigma_w$ is the open neighborhood of $w$ in $H^{\sigma}_{c,0}$, given by Corollary \ref{def P_n}.
\end{Prop}

To prove Proposition \ref{prop:Psi}, we first need to make some preliminary considerations.
By the definition of the set $\mathcal U_N$ of finite gap potentials in \cite{GK}, the eigenvalues $\lambda_n \equiv \lambda_n(w_N)$
and eigenfunctions $f_n \equiv f_n(w_N)$ of $L_{w_N}$ satisfy
\begin{equation}\label{spec theory 1}
\lambda_n = n - \sum_{k > n} \gamma_k,  \ \ \forall n \ge 0, \qquad \gamma_n \equiv \gamma_n(w_N) = \lambda_n - \lambda_{n-1} - 1, \ \ \forall n \ge 1 ,
\end{equation}
\begin{equation}\label{spec theory 2}
\quad \lambda_n  = n, \ \   \forall \, n \ge N, \qquad \qquad  \quad      \gamma_N  > 0 , \quad  \gamma_n = 0, \ \   \forall \, n > N ,  \quad \qquad  \qquad
\end{equation}
\begin{equation}\label{spec theory 3}
f_n = g_n e^{inx}, \ \  \forall n \ge 0, \qquad \qquad     \1 f_n | 1 \2 = 0, \ \  \forall n > N , \qquad \qquad \qquad \ \
\end{equation}
\begin{equation}\label{spec theory 4}
g_n \equiv g_n(w_N) = g_\infty , \ \  \forall n \ge N, \quad 
g_\infty \equiv g_\infty(w_N) = e^{i\partial_x^{-1} w_N} . \qquad \qquad
\end{equation}
Since $\1 f_n | 1 \2 = 0$ for any $n > N$ and $(f_n)_{n \ge 0}$ is an orthonormal basis of $H_+$,
\begin{equation}\label{expansion of one}
1 = \sum_{0 \le \alpha \le N} \1 1 | f_\alpha \2 f_\alpha
\end{equation}
and since by \eqref{eq:Psi_n-map}, $U^\sigma \to \C, u \mapsto \1 h_n(u) | 1 \2$ is analytic,
item (i) of Proposition \ref{prop:Psi} follows once we show that there exist $n_1 \ge n_0 \, (> N)$
and a neighborhood $U^s \equiv U^s_w$ of $w$ in $H^s_{c,0}$ so that for any $0 \le \alpha \le N$,
$$
\sum_{n \ge n_1} n^{2+2s} |\1 h_n(u) | f_\alpha \2 |^2 < \infty
$$
locally uniformly on $U^s$. By the definitions of $h_n(u)$ and $P_n(u)$,
$$
\1 h_n(u) | f_\alpha \2 = - \frac{1}{2\pi i} \oint\limits_{\partial D_n} \1  (L_u-\lambda)^{-1} f_n | f_\alpha \2 \,d\lambda , \qquad \forall n \ge n_0 ,
$$
or,  expanding $(L_u-\lambda)^{-1}  =  (L_{w_N} - \lambda)^{-1} \big({\rm Id} -  T_v(L_{w_N} - \lambda)^{-1} \big)^{-1}$
as a Neumann series (cf. \eqref{L_u as a perturbation}, \eqref{eq:neumann_series}),
$$
\1 h_n(u) | f_\alpha \2 =  - \frac{1}{2\pi i} \oint\limits_{\partial D_n}
\frac{ \1  \sum_{m \ge 0} \big( T_v(L_{w_N} - \lambda)^{-1} \big)^m f_n | f_\alpha \2}{\lambda_\alpha - \lambda}  \,d\lambda .
$$
Since $\1 f_n | f_\alpha \2 = 0$ for any $n > N$ and $0 \le \alpha \le N$, the term $m=0$ in the latter sum vanishes, yielding
after changing the index of summation from $m$ to $m -1$,
$$
\1 h_n(u) | f_\alpha \2 =  \sum_{m \ge 0} \frac{1}{2\pi i} \oint\limits_{\partial D_n}
\frac{ \1  \big( T_v(L_{w_N} - \lambda)^{-1} \big)^{m+1} f_n | f_\alpha \2}{\lambda - \lambda_\alpha }  \,d\lambda .
$$
To show Proposition \ref{prop:Psi} (i) it suffices to show that there exists $n_1 \ge n_0 \, (> N)$ 
%and a neighborhood $U^s \equiv U^s_w$ of $w$ in $H^s_{c,0}$ 
so that for any $0 \le \alpha \le N$ and any $u = w_N + v \in B^s_{c, 0}(w_N, 1/C_{M,s})$,
$$
  \sum_{m \ge 0} \Big( \frac{1}{2\pi i} \oint\limits_{\partial D_n}
\frac{ \1  \big( T_v(L_{w_N} - \lambda)^{-1} \big)^{m+1} f_n | f_\alpha \2}{\lambda - \lambda_\alpha }  \,d\lambda \Big)_{n \ge n_1} 
$$
is absolutely and uniformly summable in $\h^{s+1}_{\ge n_1}$, i.e.,
\begin{equation}\label{sum of norms of sequences}
  \sum_{m \ge 0} \Big(\sum_{n \ge n_1}  n^{2 + 2s} \big| \frac{1}{2\pi i} \oint\limits_{\partial D_n}
\frac{ \1  \big( T_v(L_{w_N} - \lambda)^{-1} \big)^{m+1} f_n | f_\alpha \2}{\lambda - \lambda_\alpha }  \,d\lambda \big|^2 \Big)^{1/2} < \infty 
\end{equation}
uniformly for $u = w_N + v \in B^s_{c, 0}(w_N, 1/C_{M,s})$.
For $m = 0$, the integrand in \eqref{sum of norms of sequences} can be computed as
$$
\frac{ \1   T_v(L_{w_N} - \lambda)^{-1}  f_n | f_\alpha \2}{\lambda - \lambda_\alpha }
=\frac{ \1 vf_n | f_\alpha \2 }{(\lambda - \lambda_\alpha)(n - \lambda)}
=\frac{ \widehat{ vg_\infty \overline{g_\alpha}} (\alpha- n) }{(\lambda - \lambda_\alpha)(n - \lambda)} \, .
$$
Since $|n- \lambda| = \frac 13$ for $\lambda \in \partial D_n$ and
$\lambda_\alpha  \le \alpha \le N$  (cf.  \eqref{spec theory 1} - \eqref{spec theory 2}), one has
\begin{equation}\label{estimate eigenvalues}
|\lambda - \lambda_\alpha| \ge n - \frac 13 - N \ge \frac 23 , \qquad \forall \, n \ge n_0, \ \forall  \, \lambda \in \partial D_n ,
\end{equation}
yielding
$$
\begin{aligned}
\big| \frac{1}{2\pi i} \oint\limits_{\partial D_n}
& \frac{ \1   T_v(L_{w_N} - \lambda)^{-1}  f_n |  f_\alpha \2}{\lambda - \lambda_\alpha }  \,d\lambda \big| \\
& \le \frac 13  \sup_{\lambda \in \partial D_n} \frac{ | \1  T_v(L_{w_N} - \lambda)^{-1}  f_n | f_\alpha \2 |}{|\lambda - \lambda_\alpha| } 
 \le  \frac{ | \widehat{ vg_\infty \overline g_\alpha} | (\alpha- n) }{ n - \frac13 - N }
\end{aligned}
$$
and in turn
$$
\Big( \sum_{n \ge n_0} \frac{n^{2 + 2s}}{(n - \frac 13 -N)^2}  
 \big| \widehat{ vg_\infty \overline g_\alpha} (\alpha- n) \big|^2\Big)^{1/2} 
 \le  \frac{n_0}{n_0 - \frac 13 -N} \|v g_\infty \overline g_\alpha \|_s .
$$
By Lemma \ref{multi of functions} and Lemma \ref{estimates g_n}  and using that $\| w_N\|_s \le \|w\|_s + 1 \le M + 1$
$$
\|v g_\infty \overline g_\alpha \|_s 
%\le C_{s,1} \|v\|_s \|  g_\infty \overline g_\alpha \|_{1+s - \tau}
\le C_{s, 1}^2 C_{s,2}^2  (1+ \|w_N\|_s)^{2\eta(s)} \|v\|_s
\le C_{s, 1}^2 C_{s,2}^2  (2+ M)^{2\eta(s)} \|v\|_s \le \frac 12
$$
for any $u = w_N + v$ with $\|v\|_s \le 1/C_{M,s}$ (cf. \eqref{def C_{M,s}}, \eqref{estimates constants}).

It thus remains to consider the terms in \eqref{sum of norms of sequences} with $m \ge 1$. One computes,
using the basis $(f_n)_{n \ge 0}$ of $H_+$,
$$
\begin{aligned}
& \1  \big( T_v(L_{w_N} - \lambda)^{-1} \big)^{m+1} f_n | f_\alpha \2 \\
&  = \sum_{\tb{k_j \ge 0}{ 1 \le j \le m}} 
 \frac{\widehat{ vg_\infty \overline g_{k_1}} (k_1- n) \,  \widehat{ vg_{k_1} \overline g_{k_2}} (k_2- k_1) \, \cdots \,  \widehat{ vg_{k_m} \overline g_\alpha} (\alpha - k_m)}
 {(n- \lambda) \prod_{1 \le j \le m} (\lambda_{k_j} -  \lambda ) } .
 \end{aligned}
$$
Since $\lambda_k \le  k$ for any $0 \le k \le N$ and $\lambda_k = k$ for any $k \ge N$ (cf. \eqref{spec theory 1} \eqref{spec theory 2}), one has
$$
\inf_{\lambda \in \partial D_n} | \lambda_k - \lambda | \ge \frac{|k-n| + 1}{5} .
$$
Together with the estimate \eqref{estimate eigenvalues}, the two latter estimates lead to 

\begin{align}\label{estimate contour integral 1}
&\Big| \frac{1}{2\pi i} \oint\limits_{\partial D_n} \frac{ \1  \big( T_v(L_{w_N} - \lambda)^{-1} \big)^{m+1} f_n | f_\alpha \2 }{\lambda - \lambda_\alpha} d \lambda \Big| \\
&  \le  \frac{5^m}{n - \frac 13 - N}
\sum_{\tb{k_j \ge 0}{ 1 \le j \le m}} 
 \frac{|\widehat{ vg_\infty \overline g_{k_1}} (k_1- n)| \,  |\widehat{ vg_{k_1} \overline g_{k_2}} (k_2- k_1)| \, \cdots \,  |\widehat{ vg_{k_m} \overline g_\alpha} (\alpha - k_m)|}
 { \prod_{1 \le j \le m} (|k_j - n| + 1 ) } . \nonumber
 \end{align}
To establish \eqref{sum of norms of sequences} it thus suffices to prove the following lemma, which in view of Section \ref{sec:the_delta_map}
is stated in a slightly stronger form than needed here.
\begin{Lem}\label{key lemma}
There exists $n_1 \ge n_0$
%and a neighborhood $U^s \equiv U^s_w$ of $w$ in $H^s_{c,0}$ 
so that for any $0 \le \alpha \le N$,
$$
\sum_{m \ge 1}8^m \big( \sum_{n \ge n_1} n^{2s}\big(
\sum_{\tb{k_j \ge 0}{ 1 \le j \le m}}  S_u(n; k_1, \dots , k_m, \alpha) \big)^2 \big)^{1/2}  < \infty
$$
uniformly for $u = w_N + v \in B^s_{c,0}(w_N, 1/C_{M,s})$ where
$$
S_u(n; k_1, \dots , k_m, \alpha)\!:= \!
\frac{|\widehat{ vg_\infty \overline g_{k_1}} (k_1- n)| \,  |\widehat{ vg_{k_1} \overline g_{k_2}} (k_2- k_1)|...
|\widehat{ vg_{k_m} \overline g_\alpha} (\alpha - k_m)|}
 { \prod_{1 \le j \le m} (|k_j - n| + 1 ) } .
$$
\end{Lem}
\noindent
{\em Proof of Lemma \ref{key lemma}} 
We fix $0 \le  \alpha \le N$ throughout the proof. Let us first consider the summands with $m=1$ and $m=2$.

\smallskip
\noindent
{\em Summand with $m=1$.}  Since for $k_1 \ge N$, $g_{k_1} = g_\infty$, we split the sum 
\begin{equation}\label{term m=1}
\sum_{k_1 \ge 0} S_u(n; k_1, \alpha) = \sum_{k_1 \ge 0}  \frac{|\widehat{ vg_\infty \overline g_{k_1}} (k_1- n)|   |\widehat{ vg_{k_1} \overline g_\alpha} (\alpha - k_1)|}
 {  (|k_1 - n| + 1 ) }  
\end{equation}
into two parts,  $\sum_{k_1 \ge 0} =  \sum_{k_1 \ge N} + \sum_{k_1 < N}$. 
Introducing $\ell_1 := k_1 -n$ as summation index instead of $k_1$
and using that $g_\infty(x) \overline g_{k_1}(x) =  g_\infty(x) \overline g_\infty(x) = 1$ for $k_1 \ge N$,
one has
\begin{equation}\label{part 1 of sum}
{\sum}_{k_1 \ge N} \le  \sum_{\ell_1 \in \Z} \frac{|\widehat{ v} (\ell_1)|   |\widehat{ vg_\infty \overline g_\alpha} (\alpha - n -  \ell_1)|} {  |\ell_1| + 1  } ,
\end{equation}
\begin{equation}\label{part 2 of sum}
{\sum}_{k_1< N} \le  \frac{1}{(n - N)^\tau} \sum_{0 \le k_1 < N}  \sum_{\ell_1 \in \Z}  
\frac{| \widehat{ vg_\infty \overline g_{k_1}} (\ell_1)|   |\widehat{ vg_{k_1} \overline g_\alpha} (\alpha - n -  \ell_1)|} {  (|\ell_1| + 1 )^{1-\tau} }  ,
\end{equation}
where $\tau \equiv \tau(s)  > 0$ is defined in Lemma \ref{multi of functions} (i) and where we used that
for any  $0 \le k_1 < N$, 
$$
|k_1 - n| + 1 = (|k_1 - n| + 1)^\tau (|k_1 - n| + 1)^{1 - \tau} \ge (n - N)^\tau (|k_1 - n| + 1)^{1 - \tau}  .
$$
%$$ \tau:= \min\{ \frac 12 (\frac 12 + s) , \frac 14 \} $$
(We remark that since we have no control over $N$ in terms of $\|w\|_s$,  the terms $|k_1 - n| + 1$ with $k_1 < N$ in the denominator of \eqref{term m=1} are split up as above.)
We now choose $n_1 \ge n_0$ so large that $(n_1 -N - \frac 23)^\tau \ge N $ and hence
\begin{equation}\label{def n_1}
\frac{1}{(n -N - \frac 23)^\tau} \le \frac{1}{N} , \qquad \forall \, n \ge n_1 \, .
\end{equation}
As a consequence, ${\sum}_{k_1 < N}\le \frac{1}{N} \sum_{0 \le \beta_1 < N}    \mathfrak S_{\beta_1}$ 
where
\begin{equation}\label{term sum_<N}
 \mathfrak S_{\beta_1} := \sum_{\ell_1 \in \Z}  
\frac{| \widehat{ vg_\infty \overline g_{\beta_1}} (\ell_1)|   |\widehat{ vg_{\beta_1} \overline g_\alpha} (\alpha - n -  \ell_1)|} {  (|\ell_1| + 1 )^{1-\tau} }  .
\end{equation}
We then conclude that
$$
\begin{aligned}
& \big( \sum_{n \ge n_1} n^{2s} \big( \sum_{k_1 \ge 0}  S_u(n; k_1, \alpha)\big)^2\big)^{1/2} \\
& \le \big( \sum_{n \ge n_1} n^{2s} \big( {\sum}_{k_1 \ge N}  \big)^2\big)^{1/2} + 
 \sup_{0 \le \beta_1 < N} \big( \sum_{n \ge n_1} n^{2s} \big( \mathfrak S_{\beta_1}  \big)^2\big)^{1/2}
\end{aligned}
$$
and, using Lemma \ref{multi of functions}, Lemma \ref{estimates g_n}, and $\|w_N\|_s \le \|w_N - w\|_s + \|w\|_s  \le 1 + M$,
$$
\begin{aligned}
\big( \sum_{n \ge n_1} n^{2s} \big( {\sum}_{k_1 \ge N}  \big)^2\big)^{1/2} 
&\le  \big( \sum_{n \ge n_1} n^{2s} \big(  \sum_{\ell_1 \in \Z} \frac{|\widehat{ v} (\ell_1)|   |\widehat{ vg_\infty \overline g_\alpha} (\alpha - n -  \ell_1)|} {  |\ell_1| + 1  }  \big)^2\big)^{1/2} \\
& \le C_{s, 1} \|v\|_s \|v g_\infty \overline g_\alpha \|_s \\
%&\le C_{s, 1}^3 \|v\|_s^2 \|g_\infty\|_{1+s} \|g_\alpha \|_{1+ s} \\
& \le C_{s,1}^3 C_{s, 2}^2  (1+ \|w_N\|_s)^{2\eta(s)} \|v\|_s^2 \\
& \le \big( C_{s,1}^3 C_{s, 2}^2  (2+ M)^{2\eta(s)}  \|v\|_s  \big)^2
\end{aligned}
$$
and
$$
\begin{aligned}
\big( \sum_{n \ge n_1} n^{2s} \big( \mathfrak S_{\beta_1}  \big)^2\big)^{1/2} 
& \le   \big( \sum_{n \ge n_1} n^{2s} 
\big(  \sum_{\ell_1 \in \Z}   \frac{| \widehat{ vg_\infty \overline g_{\beta_1}} (\ell_1)|   |\widehat{ vg_{\beta_1} \overline g_\alpha} (\alpha - n -  \ell_1)|} {  (|\ell_1| + 1 )^{1-\tau} }   \big)^2
\big)^{1/2}\\
&\le C_{s, 1} \| v g_\infty \overline g_{\beta_1} \|_s \| v g_{\beta_1} \overline g_\alpha \|_s\\
%& \le C_{s,1}^5  \|v\|_s^2 \|g_\infty\|_{1+s}  \|g_{\beta_1} \|_{1+ s}^2  \|g_\alpha \|_{1+ s} \\
& \le C_{s,1}^5 C_{s, 2}^4  (1+ \|w_N\|_s)^{4\eta(s)}  \|v\|_s^2 \\
& \le \big( C_{s,1}^3 C_{s, 2}^2  (2+ M)^{2\eta(s)}  \|v\|_s  \big)^2 \, .
\end{aligned}
$$
Altogether, we thus have proved that
\begin{align}\label{estimate 1 for Lemma}
\big( \sum_{n \ge n_1} n^{2s} \big( \sum_{k_1 \ge 0}  & S_u(n; k_1, \alpha)\big)^2\big)^{1/2}
\le 2 \big( C_{s,1}^3 C_{s, 2}^2  (2+ M)^{2\eta(s)}  \|v\|_s  \big)^2  .
% & \le \big(2 C_{s,1}^3 C_{s, 2}^2  (2+ M)^{2\eta(s)}  \|v\|_s  \big)^2 .
\end{align}
\smallskip
\noindent
{\em Summand with $m=2$.}  
The term $S_u(n; k_1, k_2, \alpha)$ is given by
$$
S_u(n; k_1, k_2, \alpha) = 
 \frac{|\widehat{ vg_\infty \overline g_{k_1}} (k_1- n)| \,  |\widehat{ vg_{k_1} \overline g_{k_2}} (k_2- k_1)| \,  |\widehat{ vg_{k_2} \overline g_\alpha} (\alpha - k_2)|}
 { \prod_{1 \le j \le 2} (|k_j - n| + 1 ) } .
$$
We split the sum $\sum\limits_{k_1, k_2 \ge 0} S_u(n; k_1, k_2, \alpha)$
into two parts  $\sum\limits_{k_1 \ge 0, k_2 \ge N}$ and $\sum\limits_{k_1 \ge 0, k_2 < N}$. 
Similarly, we write $\sum\limits_{k_1 \ge 0, k_2 \ge N} = \sum\limits_{k_1, k_2 \ge N} + \sum\limits_{k_1<N, k_2 \ge N}$.
Arguing in the same way as in the case of the summand with $m=1$, one sees that $\sum\limits_{k_2, k_1 \ge N}$ can be 
bounded as
$$
\begin{aligned}
{\sum}_{k_2, k_1 \ge N}  
&\le  \sum_{k_2 \ge N} \Big( 
\sum_{\ell_1 \in \Z} \frac{|\widehat v (\ell_1)| \,  |\widehat v (k_2- n - \ell_1)| \,  |\widehat{ vg_\infty \overline g_\alpha} (\alpha - k_2)|}{ (|\ell_1| + 1 ) (|k_2 - n| + 1 ) }  \Big) \\
 & \le \sum_{\ell_2, \ell_1 \in \Z} \frac{|\widehat v (\ell_1)| \,  |\widehat v (\ell_2 - \ell_1)|}{ |\ell_1| + 1 } \, 
 \frac{ |\widehat{ vg_\infty \overline g_\alpha} (\alpha - n -  \ell_2)|}{|\ell_2 | + 1  } ,
\end{aligned}
$$
and $ {\sum}_{ k_1<N, k_2 \ge N}$ by $\frac{1}{N} \sum_{0 \le \beta_1 < N}  \mathfrak S_{k_2 \ge N, \beta_1}$  where
$$
\mathfrak S_{k_2 \ge N, \beta_1} :=  
\sum_{\ell_2, \ell_1 \in \Z} \frac{|\widehat {v g_\infty \overline g_{\beta_1}}(\ell_1)|  \,  |\widehat{ v g_{\beta_1} \overline g_\infty }(\ell_2 - \ell_1)| }{ (|\ell_1| + 1 )^{1-\tau} } 
\frac{|\widehat{ vg_\infty \overline g_\alpha} (\alpha - n -  \ell_2)|}{|\ell_2 | + 1  } .
$$
To estimate ${\sum}_{k_2 < N} $, we write
${\sum}_{k_2 < N}  = {\sum}_{k_2 <N, k_1 \ge N} + {\sum}_{k_2, k_1 < N}$. The term $ {\sum}_{k_2 <N, k_1 \ge N} $ is bounded by
$
{\sum}_{k_2 < N, k_1 \ge N} \
\le  \frac{1}{N}\sum_{0 \le \beta_2 < N} \mathfrak S_{\beta_2, k_1 \ge N}
$ 
where
$$
\mathfrak S_{\beta_2, k_1 \ge N} :=  
\sum_{\ell_2, \ell_1 \in \Z} \frac{|\widehat v (\ell_1)| \,  |\widehat{ v g_\infty \overline g_{\beta_2}} (\ell_2 - \ell_1)| }{ |\ell_1| + 1 } 
\frac{ |\widehat{ vg_{\beta_2} \overline g_\alpha} (\alpha - n -  \ell_2)|}{ (|\ell_2 | + 1 )^{1 - \tau} }, 
$$
and ${\sum}_{k_2, k_1 < N} \le \frac{1}{N^2}\sum_{0 \le \beta_2, \beta_1 <N} \mathfrak S_{\beta_2, \beta_1}$ with
$$
\mathfrak S_{\beta_2, \beta_1} :=  
\sum_{\ell_2, \ell_1 \in \Z} \frac{|\widehat{ v g_\infty  \overline g_{\beta_1}  }(\ell_1)| \,  |\widehat{ v g_{\beta_1} \overline g_{\beta_2}} (\ell_2 - \ell_1)|}{ (|\ell_1| + 1 )^{1 - \tau} }
\frac{ |\widehat{ vg_{\beta_2} \overline g_\alpha} (\alpha - n -  \ell_2)|} { (|\ell_2 | + 1 )^{1 - \tau} }.
$$
Similarly as in the case of the summand with $m=1$, one sees that
$$
\begin{aligned}
\big( \sum_{n \ge n_1} n^{2s} \big( {\sum}_{k_2 \ge N, k_1 \ge N} \big)^2 \big)^{1/2}
& \le C_{s,1}^2 \|v\|_s^2 \|vg_\infty \overline g_\alpha\|_s  \qquad \qquad \\
& \le C_{s,1}^4 C_{s,2}^2 (1+ \|w_N\|_s)^{2\eta(s)}  \|v\|_s^3  \qquad \qquad\\
&\le \big(C_{s,1}^3 C_{s, 2}^2 (2+ M)^{2\eta(s)} \|v\|_s \big)^3 ,
\end{aligned}
$$
$$
\begin{aligned}
\big( \sum_{n \ge n_1} n^{2s} (\frak S_{k_2 \ge N, \beta_1})^2 \big)^{1/2} 
& \le  C_{s,1}^2  \|v g_\infty \overline g_{\beta_1}\|_s \, \| v g_{\beta_1} \overline g_\infty\|_s  \,  \|vg_\infty \overline g_\alpha\|_s \\
& \le  C_{s,1}^8 C_{s, 2}^6 (1+ \|w_N\|_s)^{6\eta(s)} \|v\|_s^3\\
& \le \big(C_{s,1}^3 C_{s, 2}^2 (2+ M)^{2\eta(s)} \|v\|_s \big)^3 ,
\end{aligned}
$$
$$
\begin{aligned}
\big( \sum_{n \ge n_1} n^{2s} (\mathfrak S_{\beta_2, k_1 \ge N})^2 \big)^{1/2} 
& \le  C_{s,1}^2  \| v \|_s \|v g_\infty \overline g_{\beta_2}\|_s \, \| v g_{\beta_2} \overline g_\alpha\|_s  \qquad  \\
&\le  C_{s,1}^6 C_{s, 2}^4 (1+ \|w_N\|_s)^{4\eta(s)} \| v \|_s^3 \\
&\le \big(C_{s,1}^3 C_{s, 2}^2 (2+ M)^{2\eta(s)} \|v\|_s \big)^3 ,
\end{aligned}
$$
$$
\begin{aligned}
\big( \sum_{n \ge n_1} n^{2s} (\mathfrak S_{\beta_2, \beta_1})^2 \big)^{1/2} 
& \le  C_{s,1}^2  \|v g_\infty \overline g_{\beta_1}\|_s \, \| v g_{\beta_1} \overline g_{\beta_2}\|_s  \,  \|vg_{\beta_2} \overline g_\alpha\|_s \\
&\le \big(C_{s,1}^3 C_{s, 2}^2 (2+ M)^{2\eta(s)} \|v\|_s \big)^3 ,
\end{aligned}
$$
We then conclude that
$$
\begin{aligned}
 \big( \sum_{n \ge n_1} n^{2s}  \big( \sum_{k_2, k_1 \ge 0}  & S_u(n; k_1, k_2, \alpha)\big)^2\big)^{1/2}   \le 4 \big(C_{s,1}^3 C_{s, 2}^2 (2+ M)^{2\eta(s)} \|v\|_s \big)^3 \\
 & \le \big( 2C_{s,1}^3 C_{s, 2}^2 (2+ M)^{2\eta(s)} \|v\|_s \big)^3
\end{aligned}
$$
\smallskip
\noindent
{\em Summand with $m\ge 3$.} Our computations for the summands with $m=1$ and $m=2$ suggest to introduce the map
$Q :  ( \h^{s}_c )^{N+1} \to  ( \h^{s}_c )^{N+1}$,
defined as follows. Writing an element in $( \h^{s}_c )^{N+1}$ in the form 
$\widehat \xi = (\widehat \xi_\beta)_{0 \le \beta \le N}$
where for any $0 \le \beta \le N$, $\widehat \xi_\beta$ is a sequence of Fourier coefficients $(\widehat \xi_\beta (\ell))_{\ell \in \Z}$
of the function $\xi_\beta \in H^s_c$. Then
\begin{equation}\label{def Q part0}
Q :  ( \h^{s}_c )^{N+1} \to  ( \h^{s}_c )^{N+1}, \ \widehat \xi = (\widehat \xi_\beta)_{0 \le \beta \le N} \mapsto 
\widehat \zeta = (\widehat \zeta_\beta)_{0 \le \beta \le N}
\end{equation}
is defined as follows:  $\widehat \zeta_N = (\widehat \zeta_N (\ell))_{\ell \in \mathbb Z}$ is given by
\begin{equation}\label{def Q part1}
\widehat \zeta_N(\ell) :=  \sum_{\ell_1 \in \Z} \frac{\widehat \xi_N(\ell_1) |\widehat v(\ell - \ell_1)|}{|\ell_1| + 1}
+ \frac{1}{N} \sum_{0 \le \beta_1 < N}  \sum_{\ell_1 \in \Z} \frac{\widehat \xi_{\beta_1}(\ell_1) |\widehat{ v g_{\beta_1} \overline g_\infty}(\ell - \ell_1)|}{(|\ell_1| + 1)^{1 - \tau} }
\end{equation}
and for any $0 \le \beta < N$,  $\widehat \zeta_\beta = (\widehat \zeta_\beta(\ell))_{\ell \in \Z}$ is given by
\begin{equation}\label{def Q part2}
\widehat \zeta_\beta (\ell ):= \sum_{\ell_1 \in \Z} \frac{\widehat \xi_N(\ell_1) |\widehat{ v g_\infty \overline g_\beta}(\ell - \ell_1)|}{|\ell_1| + 1}
+  \frac{1}{N} \sum_{0 \le \beta_1 < N} \sum_{\ell_1 \in \Z}\frac{\widehat \xi_{\beta_1}(\ell_1) |\widehat{ v  g_{\beta_1} \overline g_\beta }(\ell - \ell_1)|}{(|\ell_1| + 1)^{1 - \tau} } 
\end{equation}
For any $m \ge 2$, the term $\sum_{k_j \ge 0}  S_u(n; k_1, \dots , k_m, \alpha) $ is bounded as
$$
\begin{aligned}
\sum_{\tb{k_j \ge 0}{ 1 \le j \le m}} &  S_u(n; k_1, \dots , k_m, \alpha) 
 \le  \sum_{\ell \in \Z} \big(Q^{m-1}[\widehat \xi^{(0)}] \, \big)_N(\ell) \, \frac{ | \widehat{v g_\infty \overline g_\alpha}(\alpha - n -  \ell) | }{|\ell| +1}\\
 & + \frac{1}{N} \sum_{0 \le \beta < N}  \sum_{\ell \in \Z}  \big(Q^{m-1}[\widehat \xi^{(0)}] \, \big)_\beta (\ell) \, \frac{ | \widehat{v g_\beta \overline g_\alpha}(\alpha-n - \ell) | }{(|\ell| +1)^{1- \tau}}
 \end{aligned}
$$
where the starting point $\widehat \xi^{(0)} = (\widehat \xi_{\beta}^{(0)})_{0 \le \beta \le N}$ is given by (cf. \eqref{part 1 of sum}, \eqref{part 2 of sum})
\begin{equation}\label{starting point}
\widehat \xi_N^{(0)}:= \widehat v, \qquad \qquad \widehat \xi_\beta^{(0)}:= \widehat{v g_\infty \overline g_\beta} , \ \  0 \le \beta < N .
\end{equation}
For any $m \ge 2$, we then obtain from Lemma \ref{multi of functions} and Lemma \ref{estimates g_n},
\begin{align}\label{estimate 2 for Lemma}
& \Big( \sum_{n \ge n_1} n^{2s}\big( \sum_{\tb{k_j \ge 0}{ 1 \le j \le m}}   S_u(n; k_1, \dots , k_m, \alpha) \big)^2 \Big)^{1/2} \nonumber\\
&\le C_{s,1} \| \big( Q^{m-1}[\widehat \xi]  \big)_N \|_s  \| v g_\infty \overline g_\alpha \|_s 
+ C_{s,1} \sup_{0 \le \beta < N}  \| \big( Q^{m-1}[\widehat \xi^{(0)} ]  \big)_\beta \|_s  \| v g_\beta \overline g_\alpha \|_s \nonumber\\
& \le 2  C_{s,1}^3  C_{s,2}^2 (2 + M)^{2\eta(s)}  \|v\|_s   \sup_{0 \le \beta \le N}  \| \big( Q^{m-1}[\widehat \xi^{(0)}]  \big)_\beta \|_s  .
\end{align}
Now let us establish bounds for the map $Q$.
By Lemma \ref{multi of functions} and Lemma \ref{estimates g_n},
$$
\begin{aligned}
\|\widehat \zeta_N \|_s & \le C_{s, 1}  \|v\|_s \|\widehat \xi_N \|_s +  C_{s, 1}^3 C_{s,2}^2 (2 + M)^{2 \eta(s)} \|v\|_s \sup_{0 \le \beta_1 <N} \| \xi_{\beta_1}\|_s \\
& \le 2 C_{s, 1}^3 C_{s, 2}^2  (2 + M)^{2 \eta(s)}   \|v\|_s \cdot   \sup_{0 \le \beta_1 \le N} \| \widehat \xi_{\beta_1}\|_s ,
\end{aligned}
$$
and similarly, for any $0 \le \beta < N$,
$$
\| \widehat \zeta_\beta \|_s \le  2 C_{s, 1}^3 C_{s, 2}^2  (2 + M)^{2 \eta(s)}   \|v\|_s \cdot    \sup_{0 \le \beta_1 \le N} \| \widehat \xi_{\beta_1} \|_s .
$$
Altogether, we have proved that
\begin{equation}\label{estimate zeta_beta}
   \sup_{0 \le \beta \le N} \| \widehat \zeta_{\beta} \|_s 
 \le  2 C_{s, 1}^3 C_{s, 2}^2  (2 + M)^{2 \eta(s)}   \|v\|_s \cdot   \sup_{0 \le \beta_1 \le N} \| \widehat \xi_{\beta_1} \|_s .
\end{equation}
Similiarly,  by \eqref{starting point}, one has for any $0 \le \beta < N$,
$$
\| \widehat \xi_\beta^{(0)} \|_s \le C_{s,1} \|v\|_s \|g_\infty \overline g_\beta \|_{1+s} \le  C_{s,1}^3 C_{s,2}^2  (2 + M)^{2 \eta(s)}  \|v\|_s ,
$$
and hence
\begin{equation}\label{starting point 2}
 \sup_{0 \le \beta \le N} \| \widehat \xi_{\beta}^{(0)} \|_s \le  C_{s,1}^3 C_{s,2}^2  (2 + M)^{2 \eta(s)}  \|v\|_s .
\end{equation}
By induction, it then follows from \eqref{estimate zeta_beta} and \eqref{starting point 2} that for any $m \ge 2$,
\begin{equation}\label{estimate 3 for Lemma}
\sup_{0 \le \beta \le N} \|\big( Q^{m-1} [\widehat \xi^{(0)} ] \big)_\beta \|_s \le 
\big( 2 C_{s,1}^3 C_{s,2}^2  (2 + M)^{2 \eta(s)}  \|v\|_s \big)^{m} .
\end{equation}
Combining \eqref{estimate 1 for Lemma}, \eqref{estimate 2 for Lemma}, and \eqref{estimate 3 for Lemma} we have thus shown that for any $u= w_N + v \in B^s_{c,0}(w_N, 1/C_{M,s})$, 
$$
\begin{aligned}
& \sum_{m \ge 1}8^m \big( \sum_{n \ge n_1} n^{2s}\big(
\sum_{\tb{k_j \ge 0}{ 1 \le j \le m}}  S_u(n; k_1, \dots , k_m, \alpha) \big)^2 \big)^{1/2}  \\
& \le  \sum_{m \ge 1} ( 16 C_{s,1}^3 C_{s,2}^2  (2 + M)^{2 \eta(s)}  \|v\|_s)^{m+1}  \le \frac 12 ,
\end{aligned}
$$
where we used definition \eqref{def C_{M,s}} of $C_{M,s}$.
This proves Lemma \ref{key lemma}. \hfill $\square$

\smallskip

For later reference we record the estimate \eqref{estimate zeta_beta} of the operator Q, obtained in the proof of Lemma \ref{key lemma}.
Recall that $s > -1/2$, $w \in H^s_{r,0}$ with $0 < \|w\|_s \le M$, 
and $w_N$ is the finite gap potential, defined at the beginning of this section. 
\begin{Lem}\label{estimate Q}
The operator $Q$, defined by \eqref{def Q part0}-\eqref{def Q part2}, satisfies for any $v \in H^s_{c,0}$
and any $\widehat \xi = (\widehat \xi_\beta)_{0 \le \beta \le N} \in (\h^s_c)^{N+1}$ the following estimate,
$$
   \sup_{0 \le \beta \le N} \| \big( Q[ \widehat \xi] \big)_\beta \|_s 
 \le  2 C_{s, 1}^3 C_{s, 2}^2  (2 + M)^{2 \eta(s)}   \|v\|_s \cdot   \sup_{0 \le \beta_1 \le N} \| \widehat \xi_{\beta_1} \|_s .
$$
\end{Lem}
With these preparations, we are now ready to prove Proposition \ref{prop:Psi}.
Without further reference, we use the notation introduced in this section.
\begin{proof}[Proof of Proposition \ref{prop:Psi}]
(i)  The claimed statements follow from the arguments above: the neighborhood 
$U^\sigma \equiv U^\sigma_w$ of $w$ in $H^s_{c,0}$ (cf. \eqref{eq:Psi_n-map} - \eqref{eq:h_n}) has the property that
$U^\sigma \to H^{\sigma+1}_+,\ u\mapsto h_n( u)$ is analytic for any $n \ge n_0$ and so is the map
$U^\sigma : \to \C, u \mapsto \1 h_n(u) | \, f_\alpha \2$  for any $n \ge n_0$ and $0 \le \alpha \le N$.
Since by \eqref{expansion of one}, $1 = \sum_{0 \le \alpha \le N} \1 1 | f_\alpha \2 f_\alpha$,
it then remains to show that there exists $n_1 \ge n_0$ with the property that for any $0 \le \alpha \le N$, 
$\sum_{n \ge n_1} n^{2+2s} |\1 h_n(u) | f_\alpha \2 |^2 < \infty$ is uniformly bounded in some neighborhood
of $w$ in $H^s_{c,0}$, contained in $U^\sigma$. 
But this follows from \eqref{sum of norms of sequences}, \eqref{estimate contour integral 1},
and Lemma \ref{key lemma} for any $u= w_N + v$ in $U^s_w:= U^\sigma \cap B^s_{c,0}(w_N, 1/C_{M,s})$.

\smallskip
\noindent
(ii) First note that without loss of generality we can assume that $-1/2 < s \le 0$.
Recall that by \eqref{eq:h_n} one has for any $n \ge n_0$,
$$
\1 h_n(u) | f_n  \2 = \1 P_n(u) f_n | f_n  \2
 = - \frac{1}{2\pi i} \oint\limits_{\partial D_n} \1  (L_u-\lambda)^{-1} f_n | f_n \2 \,d\lambda .
$$
Since $(L_u-\lambda)^{-1}  =  (L_{w_N} - \lambda)^{-1} \big({\rm Id} -  T_v(L_{w_N} - \lambda)^{-1} \big)^{-1}$ and by \eqref{eq:neumann_series}
$$ 
\big({\rm Id} -  T_v(L_{w_N} - \lambda)^{-1} \big)^{-1} = {\rm Id} +  \sum_{k \ge 1} (T_v(L_{w_N} - \lambda)^{-1})^k
$$
uniformly on $U^s\times\Wert_n(1/4)$
and since by Cauchy's theorem, for any $n \ge n_0$,
$$
 - \frac{1}{2\pi i} \oint\limits_{\partial D_n} \1  (L_{w_N}-\lambda)^{-1} f_n | f_n \2 \,d\lambda = 
 - \frac{1}{2\pi i} \oint\limits_{\partial D_n} \frac{\1 f_n | f_n \2 }{n - \lambda}\,d\lambda   =1
$$
it follows that for any $n \ge n_0$,
$$
\begin{aligned}
| \1 h_n(u) | f_n  \2  - 1 | & \le \sum_{m \ge 1} \big| \frac{1}{2\pi i} \oint\limits_{\partial D_n} 
\frac{\1 \big( T_v(L_{w_N} - \lambda)^{-1}\big)^{m} f_n | f_n \2}{n - \lambda} \,d\lambda \big| \\
& \le \sum_{m \ge 1} \sup_{\lambda \in \partial D_n}  \| e^{-inx} \big(T_v(L_{w_N} - \lambda)^{-1} \big)^{m} f_n \|_s \|g_\infty \|_{-s} ,
\end{aligned}
$$
where we used that $f_n = g_\infty e^{inx}$ for any $n \ge n_0 \, (> N)$.
For any $\lambda \in \partial D_n$, one has by the definition of the shifted norm 
$$
\begin{aligned}
 \| e^{-inx} \big(T_v(L_{w_N} - \lambda)^{-1} \big)^{m} f_n \|_s 
& =  \| \big(T_v(L_{w_N} - \lambda)^{-1} \big)^{m} f_n \|_{s;n} \\
& \le \| T_v(L_{w_N} - \lambda)^{-1} \|^m_{H^{s;n}_+ \to H^{s;n}_+} \|f_n \|_{s;n} .
 \end{aligned}
 $$
 Since we assume $-1/2 < s \le 0,$
 $$
\|f_n \|_{s;n} = \|e^{-inx} f_n\|_s = \|g_\infty \|_s \le  \|g_\infty \| = 1 ,
$$
it then follows that
 $$
 \| e^{-inx} \big(T_v(L_{w_N} - \lambda)^{-1} \big)^{m} f_n \|_s
 \le  \| T_v(L_{w_N} - \lambda)^{-1} \|^m_{H^{s;n}_+ \to H^{s;n}_+} .
 $$
 By Lemma \ref{Lemma 16 GKT2}, one has for any $u = w_N + v \in B^s_{c,0}(w_N, 1/C_{M,s})$ that
 $$
 \sum_{m \ge 1}\| T_v(L_{w_N} - \lambda)^{-1} \|^m_{H^{s;n}_+ \to H^{s;n}_+}  \le  \frac{8/\rho}{1-8/\rho} = \frac{8}{\rho - 8} 
 $$
% Furthermore, since  $\|g_\infty \|_{-s} \le \|g_\infty \|_{1+ s}$ and
% by Lemma \ref{estimates g_n} and \eqref{def C_{M,s}}
% $$
% \|g_\infty \|_{1+ s} \le C_{s,2} (2 + M)^{\eta(s)} = \sqrt{\rho}/\sqrt{8}
% $$
% it then follows that for any $u = w_N + v \in B^s_{c,0}(w_N, 1/C_{M,s})$
and hence by the estimate $\rho \ge 2^{11}$ (cf. \eqref{def C_{M,s}})
 $$
 | \1 h_n(u) | f_n  \2  - 1 | \le  \frac{8}{\rho - 8} \le \frac 12 .
$$
\end{proof}
%%%%%%%%%%%%%%%%%%%%%%%%%%%%%%%%%%%%%%%%%%%%%%%%%%%%%%%%%%%%%%
%%%%%%%%%%%%%%%%%%%%%%%%%%%%%%%%%%%%%%%%%%%%%%%%%%%%%%%%%%%%%

\section{Poisson relations}\label{sec.normalized eigenfunctions}

The goal of this section is to show that for any $w \in H^s_{r, 0}$, $-1/2 < s \le 0,$ there exists a neighborhood
$U^s \equiv U^s_w$ of $w$ in $H^s_{c,0}$ with the property that the normalized eigenfunctions $f_n(u)$, $n \ge 0$,
of $L_u$ extend to analytic functions on $U^s$. Since the case $w=0$ has been treated in \cite{GKT3}, we concentrate
on the case $w \ne 0$.  As an application we show that the components of the Birkhoff map $\Phi$ satisfy canonical relations.

First recall that for any $u \in H^{s}_{r,0}$, $s > -1/2$,
the eigenfunctions $f_n(u)$, $n \ge 1$, satisfy the normalisation conditions  \eqref{norm'zation f_n},  
\begin{equation}\label{extension f_n 0}
f_{n}(u) =  \frac{1}{ \sqrt[+]{\mu_{n}(u)}}  P_{n }(u)\big(Sf_{n -1}(u) \big) \, .
\end{equation}
See \cite[Remark 4.4]{GK}) ($u \in L^2_{r,0}$) and \cite[Proposition 5]{GKT2} ($u \in H^{s}_{r,0}$, $ -1/2 < s < 0$.

As in Section \ref{sec:Psi}, we assume that $M >0$, $s >-1/2$, and $w$ is a potential in $H^s_{r,0}$, satisfying $0 < \|w\|_s \le M$.
The finite gap potential $w_N$ is defined as in \eqref{def w_N} with $N$ chosen so that 
$\| w - w_N\|_s < \frac{1}{2C_{M,s} }$ (cf. \eqref{choice N}) where $C_{M,s}$ is given by \eqref{def C_{M,s}}.
Furthermore,  we define $\sigma\equiv \sigma(s) := \min (s, 0)$
and let $n_0 > N$ be given as in Lemma \ref{Lemma 16 GKT2}.
We recall that for any $s > 0$, $C_{M,s} \ge C_{M, \sigma}$ 
(cf. \eqref{def C_{M,s}} and the choice of the constants $C_{s,1}$ and $C_{s,1}$ in Lemma \ref{multi of functions} and respectively, Lemma \ref{estimates g_n}).
Denote by $U^\sigma \equiv U^\sigma_w$ an open ball in $H^\sigma_{c,0}$, centered at $w$, with the following properties,

\begin{itemize}
\item[{\em (NBH1)}] Proposition \ref{prop:L-local} (Counting Lemma) and \cite[Theorem 5]{GKT2}\footnote{See also arXiv:2006.11864 for an updated version of \cite{GKT2}.} 
\big(Moment Map $\Gamma: u \mapsto (\gamma_n(u))_{n \ge 1}$\big) hold on $U^\sigma \subset B^\sigma_{c,0}(w, 1/ 2C_{M,\sigma})$.
\end{itemize}
%{\color{red}
%\begin{Rem}
%We remark that in the revised version of the paper \cite{GKT2} on arXiv, we have corrected a few misprints, 
%updated/added some references, slightly simplified the exposition in Section 6, and corrected a small error 
%in the proof of Proposition 6. The latter proposition is used in the proof of \cite[Theorem 5]{GKT2}.
%\end{Rem}
%}
We point out that in the course of our arguments, the neighborhood might be shrunk.
%We remark that by choosing $U^s$ to be a ball, it is invariant under complex conjugation, i.e., for any $u= w+v$ in $U^s$,
%also $\overline u = w + \overline v$ is in $U^s$. 
Furthermore, recall that $n_1 \ge n_0$ denotes the integer, introduced in \eqref{def n_1}.
For the convenience of the reader, we record the following results from \cite{GKT2} on the scaling factors $\kappa_n$ and $\mu_n$ 
(cf. \eqref{norm'zation f_n}, \eqref{kappa gamma}).
\begin{Lem}\label{estim1 kappa+mu}(\cite[Proposition 8]{GKT2})
There exists $n_2 \ge n_1$ so that for any $n > n_2$, the scaling factors $\kappa_n$, $\mu_n: U^\sigma \to \C$ are analytic maps
and satisfy the estimates
$$
| n \kappa_n(u) - 1 | \le \frac{7}{12} e^{1/3} < 1 ,
$$
$$
| \mu_n(u) - 1 | \le \frac{7}{6} e^{1/3} |\gamma_n(u)| \le \frac{7}{12} e^{1/3} < 1.
$$
\end{Lem}

By the arguments used in the proof of Lemma 5.2 in \cite{GK} in the case $s=0$ and by \cite[Proposition 8(ii)]{GKT2},
one shows that the following lemma holds.

\begin{Lem}\label{estim2 kappa+mu}
By shrinking the ball $U^\sigma$ of Lemma \ref{estim1 kappa+mu}, if needed, it follows that for any $0 \le n \le n_2$
with $n_2$ given as in Lemma \ref{estim1 kappa+mu}, the nth eigenfunction $f_n(u)$ of the Lax operator $L_u$
extends to an analytic map $f_n: U^\sigma \to H^{\sigma +1}_+$, satisfying
$$
L_u f_n(u) = \lambda_n(u) f_n(u), \quad f_n(u) \ne 0, \qquad \forall \, u \in U^\sigma .
$$
Furthermore, for any $1 \le n \le n_2$ and $u \in U^\sigma$,
$$
| \kappa_n(u) - \kappa_n(w) | < \kappa_n(w)/2 
$$
and the map $U^\sigma \to \C$, $u \mapsto 1/\sqrt[+]{\kappa_n(u)}$ is analytic.
\end{Lem}

In the remaining part of this section, $U^\sigma \equiv U^\sigma_w$ denotes an open ball in $H^\sigma_{c,0}$, centered at $w$,
satisfying {\em (NBH1)} and
\begin{itemize}
\item[{\em (NBH2)}] Lemma \ref{estim2 kappa+mu} holds on $U^\sigma$.
 \end{itemize}

Now we are ready to state the main result of this section.

\begin{Lem}\label{extension of f_n}
Let $w \in H^s_{r,0}$, $s > -1/2 $, and let $U^\sigma \equiv U^\sigma_{w}$ be a neighborhood of $w$ in $H^\sigma_{c,0}$,
satisfying {\em (NBH1)} - {\em (NBH2)}. Then for any $n \ge 0$, the nth eigenfunction $f_n(u)$ of $L_u$ extends to
an analytic map $f_n : U^\sigma \to H^{\sigma+1}_+$, satisfying $L_u f_n(u) = \lambda_n(u) f_n(u)$.
Furthermore, $f_n(u) \ne 0$ for any $0 \le n \le n_2$.
\end{Lem}
\begin{proof}
For $0 \le n \le n_2$, the claimed results follow by Lemma \ref{estim2 kappa+mu}. For $n > n_2$, we argue as follows.
By Corollary \ref{def P_n}  and since $n_2 \ge n_0 > N$, the Riesz projector 
$P_n : U^\sigma \to \LL\big(H^{\sigma}_+,H^{\sigma+1}_+\big)$ is analytic for any $n > n_2$.
By Lemma \ref{estim1 kappa+mu} and Lemma \ref{estim2 kappa+mu}, it then follows that \eqref{extension f_n 0}
extends to an analytic map 
$$
 f_{n_2 + 1}: U^\sigma \to H^{\sigma+1}_+ , u \mapsto f_{n_2 + 1}(u) = \frac{1}{ \sqrt[+]{\mu_{n_2 + 1}(u)}}  P_{n_2 + 1}(u)\big(Sf_{n_2}(u) \big) \, .
$$ 
We then argue by induction to conclude that in fact \eqref{extension f_n 0} extends for any $n > n_2$ to an analytic map 
\begin{equation}\label{extension f_n 1}
f_{n} :  U^\sigma \to H^{\sigma+1}_+, \, u \mapsto  \frac{1}{ \sqrt[+]{\mu_{n}(u)}}  P_{n }(u)\big(Sf_{n -1}(u) \big) .
\end{equation}
Since $P_n(u)$ is the projector of $H^\sigma_+$ onto the one dimensional
eigenspace of $L_u$, corresponding to the eigenvalue $\lambda_n(u)$, it follows that $\big( L_u - \lambda_n(u)\big) f_n(u) = 0$
for any $n > n_2$ and $u \in U^\sigma$. 
\end{proof}

To show that the Poisson brackets between components of the Birkhoff map $\Phi: H^s_{r,0} \to \h^{s+1/2}_{r,0}$ with $-1/2 < s \le 0$ 
are well defined, we need to extend two lemmas, obtained in \cite{GK} in the case $s=0$.
We recall that the gradient $\nabla_0 F$ of a $C^1$-smooth functional $F: H^s_{r,0} \to \C$ is defined with respect to the bilinear pairing
$\1 \cdot , \cdot \2$ between $H^s_{c, 0}$ and $H^{-s}_{c, 0}$,
$\1 \nabla_0 F(u) , v \2 := d_uF[v]$ for any $v \in H^s_{r,0}$. 
In case $F$ is the restriction of a $C^1$-smooth functional  $\tilde F: H^s_{r} \to \C$ to $H^s_{c, 0}$,
one has 
$$
\nabla_0 F(u) = \nabla \tilde F(u) - \1 \nabla \tilde F(u) | 1 \2 ,
$$
where $ \nabla \tilde F(u) $ denotes the 
gradient of $\tilde F$ with respect to the bilinear paring $\1 \cdot , \cdot \2$ between $H^s_{c}$ and $H^{-s}_{c}$.
\begin{Rem}\label{convention nabla}
In case the context permits, we sometimes write  $\nabla F(u)$ instead of  $\nabla_0 F(u)$
and do not distinguish between $F$ and $\tilde F$.
\end{Rem}

\begin{Lem}\label{nabla of f_n and kappa_n}
For any $-1/2 < s \le 0$ and any $n \ge 0$ the following holds:\\
(i) For any $u \in H^s_{r,0}$,
$ \nabla_0 \1 1 | f_n \2 (u)$ is  in $H^{s+1}_{c, 0}$ and 
 $\nabla_0  \1 1 | f_n \2 : H^s_{r,0} \to H^{s+1}_{c,0}$ is real analytic.\\
(ii) For any $u \in H^s_{r,0}$,
$\nabla_0 \kappa_n(u)$ is  in $H^{s+1}_{c, 0}$ and $\nabla_0 \kappa_n : H^s_{r,0} \to H^{s+1}_{r, 0}$ is real analytic.
\end{Lem}
\begin{proof}
The statements of items (i) and (ii) for $s=0$ are proved in \cite[Lemma 5.4]{GK} and  \cite[Lemma 5.7]{GK}, respectively. 
In view of Lemma \ref{estim1 kappa+mu}, Lemma \ref{estim2 kappa+mu}, and Lemma \ref{extension of f_n},
the arguments of the proof of \cite[Lemma 5.4, Lemma 5.7]{GK} can be used to show that the results in the case $-1/2< s < 0$.
\end{proof}

\begin{Coro}\label{gradients of Birkhoff coordinates}
For any $u \in H^s_{r,0}$, $-1/2 < s \le 0$, and any $n \ge 1$,  the gradient of the $n$th  component of $\Phi$, 
$\Phi_n(u) =  \frac{ \1 1 | f_n(u) \2}{\sqrt{\kappa_n(u)}}$ (cf. \eqref{def Birkhoff real}), is  in $H^{s+1}_{c, 0}$ and 
$$
\nabla_0 \Phi_n : H^s_{r,0} \to H^{s+1}_{c, 0} 
$$
is real analytic.
\end{Coro}
\begin{proof}
The stated result for $s=0$ is proved in \cite[Proposition 5.9]{GK}. For $-1/2 < s < 0$ one argues as follows: 
by the chain rule, one has for any $n \ge 1$,
$$
\nabla_0 \Phi_n(u) =  \frac{ 1 }{\sqrt{\kappa_n(u)}} \nabla_0 \1 1 | f_n \2 (u)
- \frac12 \frac{ \1 1 | f_n(u) \2 }{\kappa_n(u)^{3/2}} \nabla_0 \kappa_n(u)
$$
and the results follow from  Lemma \ref{nabla of f_n and kappa_n} together 
with Lemma \ref{estim1 kappa+mu}, Lemma \ref{estim2 kappa+mu}, and Lemma \ref{extension of f_n}.
\end{proof}
For any $-1/2 < s \le 0$, the Gardner bracket of functionals $F$, $G : H^{-s}_{r,0} \to \C$ with sufficiently regular $L^2$-gradients
is defined as (cf. Remark \ref{convention nabla})
$$
\{ F, G \} = \frac{1}{2\pi} \int_0^{2\pi} \partial_x \nabla F \cdot \nabla G \, d x \, .
$$
\begin{Prop}\label{canonical relations}
For any $-1/2 < s \le 0$, 
the Poisson brackets between the components of $\Phi : H^s_{r,0} \to \h^{s+1/2}_{r,0}$ are well-defined and for any $n, k \ge 1$
$$
\{ \Phi_n, \Phi_k \} = 0\, , \qquad \{ \Phi_n,  \overline{\Phi_{k}} \} = -i \delta_{nk}.
$$
\end{Prop}
\begin{proof}
For $s=0$, the stated results are the content of \cite[Corollary 7.3 ]{GK} and it remains to consider $-1/2 < s < 0$. Note that 
by Corollary \ref{gradients of Birkhoff coordinates}, the Poisson brackets $\{ \Phi_n, \Phi_k \}$ and $ \{ \Phi_n,  \overline{\Phi_k} \}$
are well-defined and real analytic for any $n, k \ge 1$. Since $L^2_{r,0}$ is dense in $H^s_{r,0}$, the claimed canonical relations then follow in the case $-1/2 < s < 0$ 
from the case $s=0$ by continuity.
 \end{proof}

%%%%%%%%%%%%%%%%%%%%%%%%%%%%%%%%%%%%%%%%%%%%%%%%%%%%%%%%%%%%%
%%%%%%%%%%%%%%%%%%%%%%%%%%%%%%%%%%%%%%%%%%%%%%%%%%%%%%%%%%%%%

\section{Analytic extension of the Birkhoff map}\label{sec:the_Birkhof_map}

The goal of this section is to show that for any $w \in H^s_{r,0} \setminus \{ 0 \}$ with $s > -1/2$, there exists a neighborhood $U^s_w$ of $w$ in $H^s_{c,0}$
so that $\Phi : U^s_w \cap H^s_{r,0} \to \h^{s+1/2}_{r,0}$ extends to an analytic map on $U^s_w$.

As in Section \ref{sec:Psi}, we assume that $M >0$, $s >-1/2$, and $w$ is a potential in $H^s_{r,0}$, satisfying $0 < \|w\|_s \le M$.
The finite gap potential $w_N$ is defined as in \eqref{def w_N} with $N$ chosen so that 
$\| w - w_N\|_s < \frac{1}{2C_{M,s} }$ (cf. \eqref{choice N}) where $C_{M,s}$ is given by \eqref{def C_{M,s}}.
Furthermore,  we define $\sigma\equiv \sigma(s) := \min (s, 0)$
and let $n_0 > N$ be given as in Lemma \ref{Lemma 16 GKT2}.
We recall that for any $s > 0$, $C_{M,s} \ge C_{M, \sigma}$.
Denote by $U^\sigma \equiv U^\sigma_w$ an open ball in $H^\sigma_{c,0}$, centered at $w$, satisfying {\em (NBH1)} and {\em (NBH2)} of Section \ref{sec.normalized eigenfunctions}.
For $s > 0$, let $U^s \equiv U^s_w$ be the neighborhood of $w$  in $H^s_{c,0}$, given by
\begin{equation}\label{def U^s}
U^s :=  U^\sigma \cap B^s_{c,0}(w_N, 1/C_{M,s}) .
\end{equation}
Note that $U^s$ is ball, centered at $w$, which invariant under complex conjugation, i.e., for any $u \in U^s$, one has $\overline u \in U^s$.

By \eqref{def Birkhoff real}, for any $u \in H^s_{r,0}$,  the nth component $\Phi_n(u)$, $n \ge 1$, of the Birkhoff map $\Phi$ 
is given 
$$
\Phi_n(u)=\frac{\big\1 1|f_n( u)\big\2}{\sqrt[+]{\kappa_n(u)}} = \frac{\big\1 \overline{ f_n( u)} | 1 \big\2}{\sqrt[+]{\kappa_n(u)}}.
$$
Since for any $u \in H^s_{r,0}$ and $n \ge 0$, the eigenvalue $\lambda_n(u)$  and the scaling factor $\kappa_n(u)$ 
are real valued, it then follows that for any $n \ge 1$,
\begin{equation}\label{eq:Phi_n}
\Phi_{-n}(u) =  \frac{\big\1  f_n(  u) | 1 \big\2}{\sqrt[+]{\kappa_n( u)}} , \qquad
\Phi_n(u)  = \overline{\Phi_{-n}(\overline u)} .
%{\rm c.c.} \big(\frac{\big\1 f_n( \overline u) | 1 \big\2}{\sqrt[+]{\kappa_n(\overline u)}} \big) ,
\end{equation}
%where c.c. stands for complex conjugation.
Lemma \ref{estim1 kappa+mu}, Lemma \ref{estim2 kappa+mu}, and Proposition \ref{extension of f_n}
then yield the following
\begin{Coro}\label{local extension1 Phi}
For any $w \in H^s_{r,0} \setminus \{ 0 \}$, $s > -1/2$, and $n \ge 1$,  the components $\Phi_n$ and $\Phi_{-n}$
extend to analytic maps  $U^\sigma \to \C$, given by
$$
\Phi_{-n}(u) =  \sqrt[+]{n} \frac{\big\1  f_n(  u) | 1 \big\2}{\sqrt[+]{n\kappa_n( u)}}  , \qquad
\Phi_n(u) = \overline{\Phi_{-n}(\overline u)} .
%=  {\rm c.c.} \big(  \sqrt[+]{n}\frac{\big\1 f_n( \overline u) | 1 \big\2}{\sqrt[+]{n\kappa_n(\overline u)}} \big) , \qquad
$$
\end{Coro}

Next we want to prove that for any $u \in U^s$, $(\Phi_n(u))_{n \ge 1}$ and $(\Phi_{-n}(u))_{n \ge 1}$
are in $\h^{s + 1/2}_+$ and that $(\Phi_n(u))_{n \ge 1}$, $(\Phi_{-n}(u))_{n \ge 1} : U^s \to \h^{s + 1/2}_+$
is analytic. By Lemma \ref{estim1 kappa+mu}, Lemma \ref{estim2 kappa+mu}, and Corollary \ref{local extension1 Phi},
this will hold if
\begin{equation}\label{bound1 on Phi}
\sum_{n > n_2} n^{2 + 2s} | \1 f_n(u) | 1\2 |^2 <\infty 
\end{equation}
locally uniformly for $u \in U^s$, where $n_2$ is given by Lemma \ref{estim1 kappa+mu}.
To prove \eqref{bound1 on Phi}, we use Proposition \ref{prop:Psi}, 
implying that $(\1 h_n(u) | 1 \2)_{n >n_2}$ is in $\h^{s+1}_{> n_2}$ locally uniformly for $u$ in $U^s$.
Note that $h_n(u)$ is an eigenfunction of $L_u$, corresponding to the eigenvalue $\lambda_n(u)$.
 Recall that by \eqref{extension f_n 1}
$$
f_{n}(u) =  \frac{1}{ \sqrt[+]{\mu_{n}(u)}}  P_{n }(u)\big(Sf_{n -1}(u) \big), \qquad \forall n > n_2, \ \forall \, u \in U^\sigma .
$$
and by Proposition \ref{prop:Psi}(ii),
$$
\big|\1 h_n( u) | f_n(w_N)  \2  - 1 \big| \le 1/2, \qquad \forall \, n \ge n_0,  \ \forall  \, u \in U^\sigma .
$$
Since $\lambda_n(u)$, $n \ge 0$, are simple eigenvalue of $L_u$, it follows that for any $n \ge n_2$
\begin{equation}\label{formula with f_n, a_n}
f_n(u) = a_n(u)h_n(u) , \qquad a_n(u) := \frac{\1 f_n( u) | f_n(w_N)  \2}{\1 h_n( u) | f_n(w_N)  \2} , 
\end{equation}
and for any $n > n_2$,
$P_{n }(u)\big(Sh_{n-1 }(u) \big) = \nu_{n}(u) h_{n}(u)$ where
$$
 \nu_{n}(u):= \frac{ \1 P_{n  }(u)\big(Sh_{n-1 }(u) \big) | f_n(w_N)  \2}{\1 h_n( u) | f_n(w_N)  \2} .
$$
Note that by Corollary \ref{def P_n}  and Proposition  \ref{prop:Psi}, $\nu_n : U^\sigma \to \C$ is analytic for any $n > n_2$.
The factors $ a_{n}(u)$ can be computed in terms of the $\nu_n(u)$ inductively. For notational convenience,
in the formulas below,
we will not indicate the dependence on $u$.  For any $n > n_2$ one obtains
$$
\sqrt[+]{\mu_{n}} a_{n} h_{n} = \sqrt[+]{\mu_{n}} f_{n} = P_{n }\big(Sf_{n-1 } \big)
= P_{n }\big(S a_{n-1} h_{n-1 } \big) = a_{n-1} \nu_{n} h_{n} .
$$
Since $h_{n} \ne 0$ (cf. Proposition  \ref{prop:Psi}) and $\mu_n \ne 0$ (cf. Lemma \ref{estim1 kappa+mu}),
one infers that
$$
a_n = \frac{\nu_n}{\sqrt[+]{\mu_n}} a_{n-1} .
$$
By induction it then follows that for any $n_3 \ge n_2$ and $n > n_3$
\begin{equation}\label{formula a_n}
a_n = a_{n_3} \prod_{k = n_3+1}^{n} \frac{\nu_k}{\sqrt[+]{\mu_k}} 
\end{equation}
and in turn, for any $n > n_2$ one has by \eqref{formula with f_n, a_n},
\begin{equation}\label{formula <f_n | 1>}
\1 f_n(u) | 1\2 = a_n(u) \1 h_n(u) | 1\2 , \qquad 
a_n(u) = a_{n_2}(u) \prod_{k = n_2+1}^{n} \frac{\nu_k(u)}{\sqrt[+]{\mu_k(u)}} .
\end{equation}
Note that for any $n \ge n_2$, $a_n : U^\sigma \to \C$ is analytic by \eqref{formula with f_n, a_n}
Corollary \ref{def P_n}, Lemma \ref{extension of f_n}, and Proposition  \ref{prop:Psi}(ii).
%and Lemma \ref{estim2 kappa+mu}
In view of the first identity in \eqref{formula <f_n | 1>} and Proposition  \ref{prop:Psi}(i),
the estimate \eqref{bound1 on Phi} will follow if we show 
that there exists $n_3 \ge n_2$ so that
$$
U^\sigma \to \ell^\infty_{> n_3} , \ u \mapsto (a_n(u))_{n > n_3}
$$
is locally uniformly bounded for $u$ in $U^\sigma$. We begin by estimating $\nu_n(u)$, $n > n_2$.
For notatational convenience, write $f_n \equiv f_n(w_N)$. Then 
\begin{equation}\label{rewrite nu}
\nu_n(u) = \frac{ \1 P_{n  }(u)(Sh_{n-1 }(u) ) \, | f_n  \2}{\1 h_n( u) | f_n  \2}
= 1 + \frac{\delta_n(u)}{\alpha_n(u)}
\end{equation}
where 
$$\alpha_n(u):= \1 h_n( u) | f_n  \2
$$ 
and
\begin{align}\label{64bis}
\delta_n(u) & := \1 P_{n  }(u)(Sh_{n-1 }(u))  - h_n(u) \, | f_n  \2  \nonumber\\
& =  \1 P_{n  }(u)\big(SP_{n-1}(u) f_{n-1 } - f_n\big) \, | f_n  \2 ,
\end{align}
where we used that $h_n(u) = P_n(u) f_n$.
By Corollary \ref{def P_n}  and Proposition  \ref{prop:Psi}, $\delta_n: U^\sigma \to \C$
is analytic for any $n > n_2$. 
\begin{Prop}\label{prop:delta_n-analyticity}
For any $u \in U^\sigma$, $(\delta_n(u))_{n > n_2} \in \ell^1_{> n_2}$ and
\begin{equation}\label{eq:delta}
\delta : U^\sigma \to \ell^1_{> n_2},\quad u\mapsto (\delta_n(u))_{n > n_2},
\end{equation}
 is analytic and bounded.
\end{Prop}
The proof of Proposition \ref{prop:delta_n-analyticity}, which is quite delicate and a bit technical,
is differred to  Section \ref{sec:the_delta_map}. We remark that a local version of Proposition \ref{prop:delta_n-analyticity} 
near $w=0$ is given in \cite[ Proposition 4.1]{GKT3}.

We now use  Proposition \ref{prop:delta_n-analyticity} to prove the following
\begin{Lem}\label{lem:a}
For any $u \in U^\sigma$, the sequence $(a_n(u))_{n > n_2} \in \ell^\infty_{> n_2}$ is bounded and the map
\begin{equation}\label{eq:a}
a : U^\sigma \to \ell^\infty_{n > n_2}, \  u\mapsto  (a_n(u))_{n > n_2}
\end{equation}
 is analytic.
\end{Lem}
\begin{proof}
Without loss of generality we assume that $-1/2 < s \le 0$ and hence $s = \sigma$.
Since for any $n > n_2$, $a_n: U^s \to \C$ is analytic, the claim follows by \cite[Theorem A.3]{KP-book},
once we show that there exists $n_3 \ge n_2$ so that $(a_n(u))_{n > n_3}$ is a sequence in
$\ell^\infty_{> n_3}$, which is bounded uniformly for $u \in U^s$.
In view of the formula \eqref{formula a_n} for $a_n(u)$, it suffices to show that there exists $n_3 \ge n_2$
so that the two sequences
$\big( \prod_{k = n_3+1}^{n} \nu_k(u) \big)_{n > n_3}$ and
$\big( \prod_{k = n_3+1}^{n}\frac{1}{\sqrt[+]{\mu_k(u)}}  \big)_{n > n_3}$
in  \eqref{formula a_n} are sequences in $\ell^\infty_{n > n_3}$, which are bounded uniformly for $u \in U^s$.
By Proposition \ref{prop:delta_n-analyticity}, there exists $n_3 \ge n_2$ and $C_1 \ge 1$
so that for any $u \in U^s$,
\begin{equation}\label{estimate sum delta_k}
\sum_{k = n_3+1}^\infty |\delta_k(u)| \le C_1, \qquad |\delta_n(u)| \le \frac 14, \quad \forall \, n > n_3 .
\end{equation}
Since by Proposition \ref{prop:Psi}, 
\begin{equation}\label{estimate alpha_n}
\alpha_n(u)| \ge 1/2 , \qquad \forall \, n \ge n_2 ,
\end{equation} 
it then follows that
$$
\big|  \frac{\delta_n(u)}{\alpha_n(u)}  \big| \le \frac 12 , \qquad \forall \, n > n_3, \ \forall \, u \in U^s .
$$
We then conclude from \eqref{rewrite nu}, \eqref{estimate sum delta_k}, and \eqref{estimate alpha_n}
that there exists $C_2 \ge 1$ so that for any $u \in U^s$ and any $k > n_3$,
$\log \nu_k(u) $ is well-defined and for any $n > n_3$,
$$
\sum_{k = n_3+1}^n |\log \nu_k(u) | = \sum_{k = n_3+1}^n |\log( 1 +  \frac{\delta_k(u)}{\alpha_k(u)}) | \le C_2 ,
$$
where $\log$ denotes the standard branch of the (natural) logarithm on $\C \setminus (- \infty, 0]$.
As a consequence,  for any $u \in U^s$ and $n > n_3$
$$
\big| \prod_{k = n_3+1}^{n} \nu_k(u)  \big| 
%\le \big| \exp \big( \sum_{k = n_3+1}^n \log \nu_k(u) \big) \big| 
\le \exp \big( \sum_{k = n_3+1}^n |\log \nu_k(u) |\big) \le e^{ C_2}.
$$
Let us now turn towards $\prod_{k = n_3+1}^{n}\frac{1}{\sqrt[+]{\mu_k(u)}} $.
By Lemma \ref{estim1 kappa+mu}, $\log \mu_n(u) =  \log(1 - (1- \mu_n(u)))$
is well-defined for any $n > n_3$ and any $u \in U^s$.
By the property {\em{(NHB1)}} of $U^s$ it follows that Theorem 5 of \cite{GKT2} holds.
Hence by the estimate for $|1- \mu_n(u)|$ of Lemma \ref{estim1 kappa+mu}, there exists $C_3 \ge 1$ so that 
$$
\sum_{k = n_3+1}^\infty  |\log \mu_k(u) | \le C_3, \qquad \forall \, u \in U^s .
$$
This implies that for any $u \in U^s$ and $n > n_3$
$$
\big|  \prod_{k = n_3+1}^{n}\frac{1}{\sqrt[+]{\mu_k(u)}} \big| 
\le \exp \big( \frac 12 \sum_{k = n_3+1}^n |\log \mu_k(u) |\big) \le e^{C_3 / 2} ,
$$
which completes the proof of the lemma.
\end{proof}

Combining Corollary \ref{local extension1 Phi}, formula \eqref{formula <f_n | 1>},
Proposition  \ref{prop:Psi}, and Lemma \ref{lem:a}, we arrive at
\begin{Prop}\label{Phi analytic}
Let $s  > -1/2 $, $M >0$, and $w \in H^s_{r,0}$ with $0 < \|w\|_s \le M$ and let
$U^s$ be the neighborhood of $w$ in $H^s_{c,0}$, introduced at the beginning of the section (cf. \eqref{def U^s}).
Then 
$$
\Phi : U^s \to  \h^{s+ \frac 12}_{c, 0}, \quad u \mapsto \big( (\Phi_{-n}(u))_{n \ge 1}, (\Phi_{n}(u))_{n \ge 1} \big) ,
$$
is analytic.
\end{Prop}

%%%%%%%%%%%%%%%%%%%%%%%%%%%%%%%%%%%%%%%%%%%%%%%%%%%%%%%%%%%%%
%%%%%%%%%%%%%%%%%%%%%%%%%%%%%%%%%%%%%%%%%%%%%%%%%%%%%%%%%%%%%

\section{Proof of Theorem \ref{th:Phi} and Theorem \ref{th:well-posedness}}\label{proof main results}
In this section we prove Theorem \ref{th:Phi}, stated in Section \ref{Introduction}, and then use it to show Theorem \ref{th:well-posedness}. 

By \cite[Theorem 1.1]{GKT3} (near zero) and 
 Proposition \ref{Phi analytic} (away from zero), the Birkhoff map
$$
\Phi: H^s_{r,0} \to \h^{s+1/2}_{r,0}, \, u \mapsto \big( (\overline{\Phi_{-n}(\overline u)} )_{n\le -1}, (\Phi_n(u))_{n\ge 1}\big)
$$ 
is real analytic for any $s > -1/2$, where by \eqref{def Birkhoff real}, 
\begin{equation}\label{recall Phi_n}
\Phi_n(u)=\frac{\1 1 | f_n(u) \2}{\sqrt[+]{\kappa_n(u)}}, \qquad n\ge 1.
\end{equation}
In a first step we prove the following
\begin{Prop}\label{loc diffeo}
For any $u \in H^s_{r,0}$, $s > -1/2$, the differential of $\Phi$ at $u$,
$d_u\Phi : H^s_{r,0} \to \h^{s+1/2}_{r,0}$, is a linear isomorphism.
\end{Prop}

The proof of Proposition \ref{loc diffeo} is split up into the following three lemmas.

\begin{Lem}\label{Lemma 1}
For any $u \in H^s_{r,0}$ with $s > -1/2$, 
$d_u\Phi : H^s_{r,0} \to \h^{s+1/2}_{r,0}$, is a Fredholm operator of index $0$.
\end{Lem}
\begin{Lem}\label{Lemma 2}
For any $u \in H^s_{r,0}$ with $ -1/2 < s \le 0$,
$d_u\Phi : H^s_{r,0} \to \h^{s+1/2}_{r,0}$, is a linear isomorphism.
\end{Lem}
\begin{Lem}\label{Lemma 3}
For any $u \in H^s_{r,0}$ with $s > 0$,
$d_u\Phi : H^s_{r,0} \to \h^{s+1/2}_{r,0}$, is a linear isomorphism.
\end{Lem}
Lemma \ref{Lemma 3} follows from Lemma \ref{Lemma 1} and Lemma \ref{Lemma 2}.
We prove it first.
\begin{proof}[Proof of Lemma \ref{Lemma 3}]
Let $u \in H^s_{r,0}$ with $s > 0$. First note that $d_u\Phi : H^s_{r,0} \to \h^{s+1/2}_{r,0}$
is the restriction of $d_u\Phi : L^2_{r,0} \to \h^{1/2}_{r,0}$ to $H^s_{r,0}$ and hence by Lemma \ref{Lemma 2},
$d_u\Phi : H^s_{r,0} \to \h^{s+1/2}_{r,0}$ is one-to-one.
By Lemma \ref{Lemma 1}, it then follows that $d_u\Phi : H^s_{r,0} \to \h^{s+1/2}_{r,0}$ is a linear isomorphism.
\end{proof}

Let us now turn to the proof of Lemma \ref{Lemma 2}. It uses Lemma \ref{Lemma 1} and the canonical relations,
recorded in Proposition \ref{canonical relations}.
\begin{proof}[Proof of Lemma \ref{Lemma 2}]
Let $u \in H^s_{r,0}$ with $ -1/2 < s \le 0$. According to Lemma \ref{Lemma 1}, 
$d_u\Phi : H^s_{r,0} \to \h^{s+1/2}_{r,0}$, is a Fredholm operator of index $0$.
Hence it is a linear isomorphism, once we show that it is onto, or equivalently,
that the transpose $(d_u\Phi)^\top$ of $d_u\Phi$ is one-to-one.
Note that $(d_u\Phi)^\top: \h^{-s-1/2}_{r,0} \to H^{-s}_{r,0}$ is given by
$$
(d_u\Phi)^\top  (\overline{z},   z )
= \sum_{n \ge 1} \overline{z_n} \, \nabla_0 \overline{\zeta_n}(u) + z_n \, \nabla_0 \zeta_n(u) ,
\qquad z = (z_n)_{n \ge 1} \in h_+^{-s -1/2} , 
$$
where $\overline z = (\overline z_n)_{n \ge 1} $ and  where $\nabla_0 \zeta_n$ is the $L^2_{r,0}$-gradient
of $\zeta_n = \Phi_n$. 
%It is related to the $L^2$-gradient $\nabla\zeta_n$ by $\nabla_0 \zeta_n = \nabla \zeta_n - \1 \nabla \zeta_n | 1 \2$. 
By Corollary \ref{gradients of Birkhoff coordinates},
 $\nabla_0 \zeta_n \in H^{s+1}_{c, 0}$ and hence $\partial_x \nabla_0 \zeta_n \in H^{s}_{c, 0}$ for any $n \ge 1$.
Now assume that $(\overline z, z) \in  \h^{-s-1/2}_{r,0}$ is in the kernel of $(d_u\Phi)^\top$,
$(d_u\Phi)^\top(\overline z, z) = 0$. Taking the inner product with $\partial_x \nabla_0 \zeta_k$, $k \ge 1$, one gets
\begin{equation}\label{kernel condition}
0 = \sum_{n \ge 1} \overline{z_n} \1 \nabla_0 \overline{\zeta_n} | \partial_x \nabla_0 \zeta_k \2  + 
z_n \1 \nabla_0 \zeta_n | \partial_x \nabla_0 \zeta_k \2 , \qquad \forall \, k \ge 1 \, .
\end{equation}
Since by Proposition \ref{canonical relations},
$$
\frac{1}{2\pi} \int_0^{2\pi}  \partial_x \nabla_0 {\overline \zeta_k}  \cdot \nabla_0 \zeta_n dx = \{\overline{ \zeta_k}, \zeta_n \} = i \delta_{kn}
 $$
 and
 $$
\frac{1}{2\pi} \int_0^{2\pi}  \partial_x \nabla_0 {\overline \zeta_k}  \cdot \nabla_0 \overline{\zeta_n} dx = \{\overline{ \zeta_k}, \overline{ \zeta_n} \} = 0 ,
 $$
it then follows from \eqref{kernel condition} that $z_k = 0$ for any $k \ge 1$ and hence $z=0$. This shows that 
$(d_u\Phi)^\top$ is one-to-one.
\end{proof}

It remains to prove Lemma \ref{Lemma 1}. First we need to make some preliminary considerations. We begin by computing
$$
d_u\Phi_n[v] = \1 \nabla_0 \zeta_n(u) , v \2  =  \1 \nabla_0 \zeta_n(u) | v \2, \qquad \forall \, v \in H^s_{r,0}, \ n \ge 1.
$$
By \eqref{recall Phi_n} we have
\begin{equation}\label{70bis}
 \1 \nabla_0 \zeta_n(u) | v \2 = -\frac 12 \frac{\delta \kappa_n}{\kappa_n(u)} \zeta_n(u) + \frac{\1 1 | \delta f_n \2}{\sqrt{\kappa_n(u)} },
\end{equation}
where $\delta \kappa_n = d_u \kappa_n[v]$ and similarly, $ \delta f_n = d_u f_n[v]$.
The latter can be computed as follows. From $(L_u - \lambda_n(u))f_n(u) = 0$, one gets
$$
(L_u - \lambda_n) \delta f_n -T_v f_n - (\delta \lambda_n) f_n = 0 ,
$$
where for notational convenience, we do not always indicate the dependence on $u$ and $v$.
Taking the inner product of both sides of the latter identity with $f_n $ and using that $L_u$ is self-adjoint, one gets
$$
0 - \1 T_v f_n | f_n \2 -  (\delta \lambda_n) \1 f_n | f_n \2 = 0 .
$$
Since $\1 f_n | f_n \2 = 1$ one concludes that $\delta \lambda_n = - \1 T_vf_n | f_n \2$ and therefore
$$
(L_u - \lambda_n) \delta f_n  = T_v f_n -  \1 T_v f_n | f_n \2 f_n  = P_n^\bot T_v f_n ,
$$
where $P_n^\bot \equiv P_n(u)^\bot = {\rm Id} - P_n(u) : H_+ \to H_+$. We conclude that
$$
P_n^\bot \delta f_n = (L_u - \lambda_n)^{-1} P_n^\bot(T_v f_n)
$$
and in turn, with $\delta f_n = P_n \delta f_n + P_n^\bot \delta f_n$,
$$
\1 1 | \delta f_n  \2 = \1 f_n | \delta f_n \2 \1 1 | f_n \2 + \1 P_n^\bot 1 |  (L_u - \lambda_n)^{-1} P_n^\bot(T_v f_n) \2 ,
$$
where we used that $ \1 f_n | P_n \delta f_n \2 =  \1 f_n | \delta f_n \2$.
Altogether, \eqref{70bis} then becomes
\begin{equation}\label{70ter}
\1 \nabla_0 \zeta_n | v \2 = - \frac 12 \frac{\delta \kappa_n}{\kappa_n} \zeta_n + \1 f_n | \delta f_n \2 \zeta_n
+ \frac{\1  (L_u - \lambda_n)^{-1} P_n^\bot1 | P_n^\bot(T_v f_n ) \2  }{\sqrt{\kappa_n} } .
\end{equation}
For $u \in \mathcal U_N$, the latter formula simplifies. For any $n > N$, one has
\begin{equation}\label{70quarto}
\lambda_n = n, \quad \gamma_n = 0, \quad \1 1 | f_n \2 = 0, \quad \zeta_n = 0, \quad   f_n = g_\infty e^{inx}, \quad  g_\infty(x) = e^{i\partial_x^{-1} u(x)} .
\end{equation}
In particular, for any $n > N$, 
\begin{equation}\label{70quinto}
1 =  \sum_{0 \le k \le N} \1 1 | f_k \2 f_k = P_n^\bot 1 .
\end{equation} 
Since
$$ \frac 1n + \frac{1}{ \lambda_k - n} = - \frac 1n \frac{\lambda_k}{n(1 - \lambda_k /n )}, \qquad \forall \, 0 \le k \le N ,
$$ 
it then follows from \eqref{70quinto} that
$$
 (L_u - \lambda_n)^{-1} P_n^\bot 1 = - \frac 1n \, -  \, \frac 1n \sum_{0 \le k \le N}  \frac{\lambda_k}{n(1 - \lambda_k /n )} \1 1 | f_k \2 f_k , \qquad \forall \, n > N .
$$
This together with \eqref{70ter} and \eqref{70quarto} then imply that for any $n > N$,
\begin{align}\label{variation zeta_n} 
\1 \nabla_0 \zeta_n | v \2  & = - \frac{1}{\sqrt{n}} \1 v \overline g_\infty | e^{inx} \2
+ \big(1 - \frac{1}{\sqrt{n\kappa_n}} \big) \frac{1}{\sqrt{n}} \1 v \overline g_\infty | e^{inx} \2 \nonumber \\
& - \frac 1n \frac{1}{\sqrt{n}} \frac{1}{\sqrt{n\kappa_n}} \sum_{0 \le k \le N} \frac{\lambda_k \1 1 | f_k \2}{1 - \lambda_k/n} \1 v\overline g_\infty f_k | e^{inx}\2 .
\end{align}
The latter formula suggests to introduce for any $u \in H^s_{r,0}$, $s > -1/2$,
\begin{equation}\label{75bis}
G(u) : H^s_{r,0} \to  \h^{s + 1/2}_+ , \, v \mapsto \big( \frac{1}{\sqrt{n}} \widehat{ v \overline g_\infty}(n) \big)_{n \ge 1} .
\end{equation}
\begin{Rem}
In \cite{GKT3}, we computed the differential of $\Phi$ at $u=0$, to be the weighted Fourier transform
$$
d_0 \Phi : H^s_{r,0} \to \h^{s+1/2}_{r,0}, \, v \mapsto ( - \frac{1}{\sqrt{|n|}} \widehat v(n) )_{n\ne 0} .
$$
Note that for $u=0$,  $g_\infty \equiv 1$ and hence $(d_0 \Phi_n)_{n \ge 1} = - G(0)$.
\end{Rem}

One verifies in a straightforward way that $G(u) \in \mathcal L(H^s_{r,0}, \h^{s + 1/2}_+)$ and that
$$
G: H^s_{r,0} \to \mathcal L(H^s_{r,0}, \h^{s + 1/2}_+), \, u \mapsto G(u) ,
$$
is continuous with respect to the operator norm topology on $ \mathcal L(H^s_{r,0}, \h^{s + 1/2}_+)$.
Define 
$\pi_{r,0} : \h^s_{r,0} \to  \h^s_+, \, (z_n)_{n \ne 0} \mapsto (z_n)_{n \ge 1}$, and for any $u \in H^s_{r,0}$, $s > -1/2$,
\begin{equation}\label{75ter}
A(u) :=  \pi_{r,0} \, d_u \Phi + G(u) .
\end{equation}
It follows that $A(u) \in \mathcal L(H^s_{r,0}, \h^{s + 1/2}_+)$ and that 
$$
A: H^s_{r,0} \to \mathcal L(H^s_{r,0}, \h^{s + 1/2}_+), \, u \mapsto A(u)
$$
is continuous.
\begin{Lem}\label{A compact}
For any $u \in H^s_{r,0}$ with $s > -1/2$, $A(u): H^s_{r,0} \to \h^{s+1/2}_+$ is a compact operator.
\end{Lem}
\begin{proof}
Let $s > -1/2$. Since the set of compact operators  $H^s_{r,0} \to \h^{s + 1/2}_+$ is a closed subset of $ \mathcal L(H^s_{r,0}, \h^{s + 1/2}_+)$,
the operator $A: H^s_{r,0} \to \mathcal L(H^s_{r,0}, \h^{s + 1/2}_+) $ is continuous, 
and $\cup_{N \ge 0}\,  \mathcal U_N$ is dense in $H^s_{r,0}$ (cf. \cite[Theorem 1]{GK}, \cite[Theorem 6, Proposition 5]{GKT1}), 
 it suffices to prove that for any $u \in \mathcal U_N$, $N \ge 1$, $A(u) : H^s_{r,0} \to \h^{s+1/2}_+$
is compact. Let $u \in \mathcal U_N$ for some $N \ge 1$.
By \eqref{variation zeta_n} and \eqref{75ter}, the components $A_n(u)[v]$ of $A(u)[v] = (A_n(u)[v])_{n \ge 1}$ with $n > N$ can be written as
$$
A_n(u)[v] = \big( 1 - \frac{1}{\sqrt{n \kappa_n}}  \big) G_n(u)[v]
-  \frac 1n \frac{1}{\sqrt{n}} \frac{1}{\sqrt{n\kappa_n}} \sum_{0 \le k \le N} \frac{\lambda_k \1 1 | f_k \2}{1 - \lambda_k/n} \1 v\overline g_\infty f_k | e^{inx}\2 .
$$
We estimate the two terms separately.
By the product representation of $\kappa_n$ (cf. \cite[Corollary 3.4]{GK}) and since $\gamma_n = 0$ for $n > N$,
$$
n\kappa_n = \frac{n}{n - \lambda_0} \prod_{1 \le p \le N} (1 + \frac{\gamma_p}{n - \lambda_p}) , \qquad \forall \, n > N .
$$
Since $- \lambda_0 = | \lambda_0|$ (cf. \eqref{70quarto}), one gets for any $n > N$
$$
\begin{aligned}
\frac{1}{\sqrt{n \kappa_n}} & = \Big( \frac{n +|\lambda_0|}{n} \Big)^{1/2} \Big( \prod_{1 \le p \le N}  \big( 1 + \frac{\gamma_p}{n - \lambda_p}  \big) \Big)^{-1/2}\\
& = \exp\Big( \frac 12 \log \big(1 + \frac{|\lambda_0|}{n} \big) - \frac 12 \sum_{1 \le p \le N}  \log \big(1 + \frac{\gamma_p}{n - \lambda_p} \big)  \Big) .
\end{aligned}
$$
By the elementary estimates for $x \in \R$ and $y \ge 0$,
$$
|e^x - 1 | = \big|  \int_0^1\frac{d}{dt} e^{tx}  d t \big| \le |x|e^{|x|}, \qquad 
\big| \log(1 + y)  \big| = \big| \int_0^y  \frac{1}{1+t} dt\big| \le y ,
$$
it then follows with 
$$
x = \frac 12 \log (1 + \frac{|\lambda_0|}{n}) - \frac 12 \sum_{1 \le p \le N}  \log(1 + \frac{\gamma_p}{n - \lambda_p})
$$ 
that
$$
\big|  \frac{1}{ \sqrt{ n\kappa_n}} -1 \big| \le |x| e^{|x|} \le \big( \, \frac 12 \frac{|\lambda_0|}{n} 
+ \frac 12  \frac 1n \sum_{1 \le p \le N}   \frac{\gamma_p}{1 - p/n} \,  \big) \, e^x , 
$$
where we used that for $0 \le p \le N$,  $\lambda_p \le p$ (cf. \eqref{70quarto}).
Hence there exists a constant $C_1 \ge 1$ so that 
$$
\big|  \frac{1}{ \sqrt{ n\kappa_n}} -1 \big| \le \frac{C_1}{n}, \qquad \forall \, n > N .
$$
This implies that for $u$ given as above, 
$$
H^s_{r,0} \to \h^{s+1/2 + 1}_{> N}, \, v \mapsto 
\big( \big( 1 - \frac{1}{\sqrt{n \kappa_n}}  \big) G_n(u)[v] \big)_{n > N}
$$
is a bounded map.
Furthermore, for any $0 \le k \le N$,
$$
H^s_{r,0} \to \h^{s+1/2 + 1}_{> N}, \, v \mapsto 
\big( \frac 1n \frac{1}{\sqrt{n}} \frac{1}{\sqrt{n\kappa_n}}  \frac{\lambda_k \1 1 | f_k \2}{1 - \lambda_k/n} \1 v\overline g_\infty f_k | e^{inx}\2 \big)_{n > N}
$$
is a bounded map as well. Since the embedding $ \h^{s+1/2 + 1}_+ \to  \h^{s+1/2}_+$ is compact,   we then conclude that
$$
H^s_{r,0} \to \h^{s+1/2}_+, \, v \mapsto \big(  A_n(u)[v]  \big)_{n \ge 1}
$$
is compact.
\end{proof}

For the proof of Lemma \ref{Lemma 1} we need to make some further preparations. 
For any $u \in H^s_{r,0}$ with $s > -1/2$, let
\begin{equation}\label{Toeplitz/Hankel}
T_{\overline g_\infty} : H^s_+ \to H^s_+ , \, f \mapsto \Pi(\overline g_\infty f) \, ,
\qquad  
H_{\overline g_\infty} : H^s_- \to H^s_+, \, f \mapsto  \Pi( \overline g_\infty f) \, ,
\end{equation}
denote the Toeplitz operator 
%with symbol $\overline g_\infty$ 
and respectively,
the Hankel operator with symbol $\overline g_\infty$. Here  $H^s_- $ denotes
the complex conjugate of the Hardy space $H^s_+$,
$$
H^s_- :=  \{ f \in H^s_{c} : \, \widehat f (n) =0 \ \forall \,  n \ge 1 \}.
$$
Furthermore let
$$
H^s_{+,0} := \{ f \in H^s_{+} : \, \langle f | 1 \rangle =0 \} 
$$ 
and introduce the inclusion $\iota_+ : H^s_{+,0} \to H^s_+$, the projection
$$
\Pi_{\ge 1} : H^s_+ \to H^s_{+,0} \, , \ \sum_{n \ge 0} \widehat v(n) e^{inx} \mapsto  \sum_{n \ge 1} \widehat v(n) e^{inx} \, , 
$$
and the following version of the restriction of the Toeplitz operator to $H^s_{+,0}$,
\begin{equation}\label{Toeplitz restricted}
T^{(1)}_{\overline g_\infty} := \Pi_{\ge 1} T_{\overline g_\infty} \iota_+ : \ H^s_{+, 0} \to H^s_{+, 0}.
\end{equation}
It follows from the well known fact that the Toeplitz operator $T_{\overline g_\infty}: H^s_+ \to H^s_+$
is a linear isomorphism (cf. e.g. \cite[Theorem 10, p 309]{Lax}) that the same is true for $T^{(1)}_{\overline g_\infty}$.
For the convenience of the reader we include a proof.
\begin{Lem}\label{Fredholm}
For any $u \in H^s_{r,0}$ with $s > -1/2$,  the following holds:\\
$(i)$ $T_{\overline g_\infty}: H^s_+ \to H^s_+$ is a linear isomorphism.\\
$(ii)$ $T^{(1)}_{\overline g_\infty}: H^s_{+, 0} \to H^s_{+, 0}$ is a linear isomorphism.\\
$(iii)$ $H_{\overline g_\infty} : H_-^s \to H_+^s$ is a compact operator.
%$(iv)$ $G(u): H^s_{r,0} \to h^{s+ 1/2}_+$ is a Fredholm operator of index zero.
\end{Lem}
\begin{Rem}
$(i)$ The Hankel operator $H_{\overline g_\infty}$ can be shown to have smoothing properties which imply
that it is compact -- see e.g. \cite[Appendix A]{GKT4}. However, for our purposes, such smoothing properties are not needed.\\
$(ii)$ The inverse $(T^{(1)}_{\overline g_\infty})^{-1} : H^s_{+, 0} \to H^s_{+, 0}$ of $T^{(1)}_{\overline g_\infty}$ is given for any $f \in H^s_{+,0}$ by
$$
(T^{(1)}_{\overline g_\infty})^{-1} (f)  = f_1 - \langle f_1 | 1 \rangle g_+\, ,  \qquad f_1 := g_+ \Pi [g_- f] \, , 
$$
where $g_\pm$ are defined in \eqref{def g_pm} below.
\end{Rem}
\begin{proof}
Let $u$ be in $H^s_{r,0}$ with $s > -1/2$. Then $g_\infty = e^{i \partial_x^{-1}u(x)}$ is an element in $H^{s+1}_c$, 
hence in particular it is continuous.
For notational convenience, we write $g$ instead $g_\infty$ 
%$h$ instead $\overline g_\infty$. 
and decompose it as
\begin{equation}\label{def g_pm}
g = g_+ \cdot g_-, \qquad g_+(x) := e^{ i\partial_x^{-1} \Pi u(x)},  \qquad g_-(x) := e^{ i\partial_x^{-1} ({\rm Id} - \Pi) u(x)} .
\end{equation}
Since $u$ is real valued one has
$$
\overline g_\infty = e^{- i \partial_x^{-1}u(x)} = e^{ -i\partial_x^{-1} \Pi u(x)} e^{- i\partial_x^{-1} ({\rm Id} - \Pi) u(x)} = \frac{1}{g_+} \frac{1}{g_-} .
$$
Note that  $|g_\pm (x) | =1$ for any $x \in \T$. Furthermore, using the Taylor expansion of $e^{i y}$ at $y=0$ and Lemma \ref{multi of functions}(ii), one sees that 
$g_{\pm} \in H^{s+1}_\pm$ and $\frac{1}{g_\pm} \in H^{s+1}_\pm$. \\
$(i)$ 
Since $T_{\overline g_\infty}$ is bounded, it suffices to show that it is bijective.
First let us show that $T_{\overline g_\infty}$ is one-to-one. Assume that $T_{\overline g_\infty} f = 0$ for some $f$ in $H^s_+$. 
It means that there exists $h\in H^s_c$ with $\Pi (h) = 0$
so that ${\overline g_\infty} f = h$. Multiplying both sides of the latter equation by $g_-$, yields $\frac{1}{g_+} f = g_-h $. 
Since $g_-$ is in $H^{s+1}_-$
it follows that $\Pi( g_-h ) = 0$. On the other hand, $\Pi( \frac{1}{g_+} f) = \frac{1}{g_+} f$, implying that $\frac{1}{g_+} f = 0$ and hence $f=0$.

To show that $T_{\overline g_\infty}$ is onto, let $f_0 \in H^s_+$ and define
$$
f := g_+ \Pi \big[ g_- f_0 \big] \in H^s_+ .
$$ 
We claim that $f$
is a solution of $T_{\overline g_\infty} f = f_0$. Indeed, by the definition of $f$ there exists $h \in H^s_c$ with $\Pi (h) = 0$ so that
$\frac{1}{g_+} f =  g_- f_0  + h$. Multiplying  both sides of the equation by $\frac{1}{g_-}$, one obtains
${\overline g_\infty} f = f_0 + \frac{1}{g_-} h$.  Since $\Pi(\frac{1}{g_-} h) = 0$ it then follows that 
$$
T_{\overline g_\infty} f = \Pi({\overline g_\infty} f) = \Pi f_0 + 0 = f_0 .
$$
$(ii)$ Using the Taylor expansion of $e^{iy}$ and taking into account that $\langle u | 1 \rangle = 0$,
one sees that $ \Pi g_-^{-1} = 1$ and $ \Pi g_-= 1$. It then follows from item $(i)$ that
\begin{equation}\label{80bis}
T_{\overline g_\infty} g_+ = \Pi (\frac{1}{g_-}) = 1\, , \qquad
T_{\overline g_\infty}^{-1} 1 = g_+ \Pi g_- = g_+. 
\end{equation}
Let $E^s_+ := T_{\overline g_\infty}^{-1}(H^s_{+,0})$. Since by item $(i)$, $T_{\overline g_\infty}^{-1}: H^s_{+} \to H^s_{+}$ is a linear isomorphism, any element
$f \in H^s_{+,0}$ can be written in a unique way as a linear combination of $g_+$ and an element $f_1$ of $E^s_+$, 
$f = cg_+ +  f_1$ where $c = - \langle f_1| 1 \rangle $. 

First let us show that $T^{(1)}_{\overline g_\infty}$ is one-to-one.  Assume that $f \in H^s_{+,0}$ with $T^{(1)}_{\overline g_\infty} f = 0$.
Representing $f$ as $f = - \langle f_1 | 1 \rangle g_+ + f_1$ with $f_1 \in E^s_+$,
it then follows from \eqref{80bis} that 
$T_{\overline g_\infty} f =  - \langle f_1| 1 \rangle + T_{\overline g_\infty} f_1$, implying that
$0 = T^{(1)}_{\overline g_\infty} f = T_{\overline g_\infty} f_1$. Hence by item (i), $f_1 = 0$ and in turn $f = 0$.

To see that $T^{(1)}_{\overline g_\infty}$ is onto, take $h \in H^s_{+, 0}$ and let $f_1 := T_{\overline g_\infty}^{-1} h$.
Then $f := - \langle f_1 | 1 \rangle g_+ + f_1$ is an element in $H^s_{+, 0}$, and 
$T^{(1)}_{\overline g_\infty} f = T_{\overline g_\infty} f_1 = h$. 

\noindent
$(iii)$ To prove that the Hankel operator $H_{\overline g_\infty}$ is compact, we  approximate $\overline g_\infty$ by 
$\overline g_{\infty, j} := \sum_{k \le  j} \widehat{\overline g}_\infty (k) e^{ikx}$, $j \ge 1$. One has 
 $\overline g_\infty = \lim_{j \to \infty} \overline g_{\infty, j}$ in $H^{s+1}_c$ and hence $H_{\overline g_\infty, j} \to H_{\overline g_\infty}$ in operator norm
 (cf. Lemma \ref{multi of functions}, Remark \ref{multi of functions 3}).
 Since $H_{\overline g_\infty, j}$ is an operator of rank $j+1$, it is compact and hence so is $H_{\overline g_\infty}$.
\end{proof}
We are now ready to prove Lemma \ref{Lemma 1}.
 \begin{proof}[Proof of Lemma \ref{Lemma 1}]
By Lemma \ref{A compact}, it remains to prove that for any given $u \in H^s_{r,0}$ with $s > -1/2$,
$G(u): H^s_{r,0} \to \h^{s+1/2}_{r,0}$ is a Fredholm operator of index $0$. 
For any $v \in H^s_{r,0}$, write $v = \Pi v +({\rm Id} - \Pi) v$. Hence by the definition \eqref{75bis} of $G(u)[v]$, 
one has
%where $\Pi^- : H^s_{r,0} \to H^s_-$ denotes the Szeg\H{o} projection onto the Hardy space $H^s_-$,
$$
G(u)[v] =  \mathcal F^+_{1/2} (\Pi [ \overline g_\infty \Pi v])
+ \mathcal F^+_{1/2} (\Pi [ \overline g_\infty ({\rm Id} - \Pi) v ])
$$
where $\mathcal F^+_{1/2}$ denotes the weighted partial Fourier transform
$$
\mathcal F^+_{1/2} : H^s_c \to \h^{s+1/2}_+ , \, f \mapsto \big( \frac{1}{\1 n \2^{1/2}} \widehat f(n)  \big)_{n \ge 1} .
$$
In view of the definitions \eqref{Toeplitz/Hankel} and \eqref{Toeplitz restricted}, $G(u)[v]$ can be expressed in terms of 
the restriction $T^{(1)}_{\overline g_\infty}$ of the Toeplitz operator $T_{\overline g_\infty}$ and the Hankel operator $H_{\overline g_\infty}$,
$$
G(u)[v]= \mathcal F^+_{1/2} (T^{(1)}_{\overline g_\infty}[\Pi v])  + \mathcal F^+_{1/2} (H_{ \overline g_\infty}[({\rm Id} - \Pi) v]) .
$$
By Lemma \ref{Fredholm}(ii), $T^{(1)}_{\overline g_\infty} : H^s_{+,0} \to H^s_{+,0} $ is a linear isomorphism and hence so is
$$
H^s_{r,0} \to \mathfrak h^{s+1/2}_+, \, v \mapsto  \mathcal F^+_{1/2} (T^{(1)}_{\overline g_\infty}[\Pi v]).
$$
Since by Lemma \ref{Fredholm}(iii),
$$
H^s_{r,0} \to \mathfrak h^{s+1/2}_+, \, v \mapsto  \mathcal F^+_{1/2} (H_{ \overline g_\infty}[ ({\rm Id} -\Pi) v]) 
$$
is compact, we then conclude that $G(u) : H^s_{r,0} \to \mathfrak h^{s+1/2}_+$ is a Fredholm operator of index $0$.
\end{proof}
All the pieces of the proof of Theorem \ref{th:Phi} are now in place and it remains to put them together.
\begin{proof}[Proof of Theorem \ref{th:Phi}]
Let $s > -1/2$.
 By \cite[Theorem 1.1]{GKT3} (near $0$) and  Proposition \ref{Phi analytic} (away from $0$), 
 the Birkhoff map (cf. \eqref{eq:Phi_n})
$$
\Phi: H^s_{r,0} \to \h^{s+1/2}_{r,0}, \, u \mapsto \big(\big(\overline{\Phi_{-n}(\overline u)}\big)_{n\le -1},\big(\Phi_n(u)\big)_{n\ge 1}\big)
$$ 
is real analytic.
By Proposition \ref{loc diffeo}, $\Phi : H^s_{r,0} \to \h^{s+1/2}_{r,0}$ is a local diffeomorphism. 
Finally,
by \cite[Theorem 1.1]{GK} ($s=0$) and \cite[Theorem 6, Proposition 5 ]{GKT1} ($s > -1/2$),  
$\Phi : H^s_{r,0} \to \h^{s+1/2}_{r,0}$ is bijective.
\end{proof}
To prove Theorem \ref{th:well-posedness}, we follow the strategy  developed in \cite[Section 6]{GKT3}
for showing the results of Theorem \ref{th:well-posedness} near $0$.
 Recall that  the Benjamin-Ono equation is globally well-posed in $H^s_{r,0}$ for any 
$s > -1/2$ (cf. \cite{GKT1} for details and references). For any given $t \in \mathbb R$,  denote by $\mathcal S^t_0$ the 
flow map of the Benjamin-Ono equation on $H^s_{r,0}$, $\mathcal S^t_0 : H^s_{r,0} \to H^s_{r,0}$ and by 
$\mathcal S^t_{B}: \h^{s+1/2}_+ \to \h^{s+1/2}_+$ the version of $\mathcal S^t_0$ obtained, when expressed in 
the Birkhoff coordinates $(\zeta_n)_{n \ge 1}$. 
To describe $\mathcal S^t_{B}$ more explicitly, recall that the nth frequency $\omega_n$, $n \ge 1$, of the Benjamin-Ono 
equation is the real valued, affine function defined on $\h^{s+1/2}_+$ (cf. \cite{GK}, \cite{GKT1}),
\begin{equation}\label{eq:formula_nth_frequency}
\omega_n(\zeta) = n^2 -2\sum_{k=1}^{n} k |\zeta_k|^2
- 2n\sum_{k=n+1}^{\infty} |\zeta_k|^2\, .
\end{equation}
For any initial data  $\zeta(0) \in \h^{s+ 1/2}_+$ with $ s > -1/2$, $\mathcal S^t_B(\zeta(0))$ is given by
\begin{equation}\label{formula for S_B}
\mathcal S^t_B(\zeta(0)):= \big( \zeta_n(0)e^{it\omega_n(\zeta(0))} \big)_{n \ge 1} \, .
\end{equation}
The key ingredient for the proof of Theorem \ref{th:well-posedness}(i)
is the following result, established in \cite[Lemma 6.1]{GKT3}.
\begin{Lem}\label{lem:nowhere locally uniformly continuous}
For any $t \ne 0$ and any $ - 1/2 < s < 0$, 
$\mathcal S^t_B : \h^{s+ 1/2}_+ \to \h^{s+1/2}_+$
is nowhere locally uniformly continuous. In particular,
it is {\em not} locally Lipschitz.
\end{Lem}
\begin{proof}[Proof of Theorem \ref{th:well-posedness}(i)]
The claimed result follows directly from Lemma \ref{lem:nowhere locally uniformly continuous} and 
Theorem \ref{th:Phi}. 
\end{proof}
To outline the proof of Theorem \ref{th:well-posedness}(ii), we first make some preliminary considerations. 
Denote by $\ell^\infty_{c,0}$ the space $\ell^\infty(\Z \setminus\{0\}, \C)$ 
of bounded complex valued sequences $z = (z_n)_{n \ne 0}$, endowed with the supremum norm. 
It is a subspace of $\ell^\infty_c$ and a Banach algebra with respect to the multiplication, 
\begin{equation}\label{eq:l^infty-Banach_algebra}
\ell^\infty_{c,0}\times\ell^\infty_{c,0}\to\ell^\infty_{c,0},\quad(z,w)\mapsto z\bcdot w:=(z_n w_n)_{n\ne 0}.
\end{equation}
Similarly, for any $s\in\R$, the bilinear map
\begin{equation}\label{eq:h*l^infty-multiplication}
\h^s_{c,0}\times\ell^\infty_{c,0}\to\h^s_{c,0},\quad(z,w)\mapsto z\bcdot w,
\end{equation}
is bounded. For any given $T>0$ and $s\in\R$, introduce the complex Banach spaces
\[
\mathcal{C}_{T, s}:=C\big([-T,T],\h^s_{c,0}\big)\quad {\rm and}\quad
\mathcal{C}_{T, \infty} :=C\big([-T,T],\ell^\infty_{c,0}\big),
\] 
endowed with the supremum norm.
Elements of these spaces are denoted by $\xi$, or more explicitly, $\xi(t) = (\xi_n(t))_{n \ne 0}$.
The multiplication in \eqref{eq:l^infty-Banach_algebra} and \eqref{eq:h*l^infty-multiplication} induces
in a natural way the following bounded, bilinear maps (cf. \cite[Section 6]{GKT3})
\begin{equation}\label{eq:multiplication_on_curves}
\mathcal{C}_{T, \infty} \times\mathcal{C}_{T, \infty} \to\mathcal{C}_{T, \infty},\qquad
\mathcal{C}_{T, s} \times\mathcal{C}_{T, \infty}\to\mathcal{C}_{T, s},  \qquad
\mathcal{C}_{T,s} \times \ell^\infty_{c,0}\to\mathcal{C}_{T,s}.
\end{equation}
%For example, the boundedness of the first map in \eqref{eq:multiplication_on_curves} follows from 
%the Banach algebra property of the multiplication in $\ell^\infty_{c,0}$ and the estimate
%\begin{align*}
%\|\xi^{(1)}\bcdot\xi^{(2)}\|_{\mathcal C_{T,\infty}}&=
%\sup_{t\in[-T,T]}\|\xi^{(1)}(t)\bcdot\xi^{(2)}(t)\|_{\ell_c^\infty}\le
%\sup_{t\in[-T,T]}\|\xi^{(1)}(t)\|_{\ell_c^\infty}\|\xi^{(2)}(t)\|_{\ell_c^\infty}\\
%&\le\|\xi^{(1)}\|_{\mathcal C_{T,\infty}}\|\xi^{(2)}\|_{\mathcal C_{T,\infty}},\qquad \xi^{(1)},\xi^{(2)}\in \mathcal C_{T,\infty}.
%\end{align*}
%Hence, $\big(\mathcal{C}_{T, \infty},\bcdot\big)$ is a Banach algebra.
The following lemma can be shown in a straightforward way.

\begin{Lem}\label{lem:exp-analytic}
The map 
$\mathcal{C}_{T,\infty} \to \mathcal{C}_{T,\infty}$, 
$\xi \mapsto 
e^\xi:=\big(e^{\xi_n}\big)_{n\ne 0}$,
is analytic.
\end{Lem}
%\begin{proof}[Proof of Lemma \ref{lem:exp-analytic}]
%Since $\big(\mathcal{C}_{T,\infty},\bcdot\big)$ is a Banach algebra, we see that the series
%$\sum_{k=0}^\infty\frac{1}{k!}\,\underbrace{\xi\boldsymbol{\cdots}\xi}_{k{\rm times}}$
%converges absolutely in $\mathcal{C}_{T,\infty}$. By Taylor's formula, for any given $n\in\Z$ and any $t\in[-T,T]$,
%the $n$th component of the series $\sum_{k=0}^\infty\frac{1}{k!}\,\xi(t) \boldsymbol{\cdots}\xi(t)$ 
%converges to the $n$th component of $e^{\xi(t)}$. Hence,
%\begin{equation}\label{Taylor e^ xi}
%e^\xi=\sum_{k=0}^\infty\frac{1}{k!}\,\underbrace{\xi\boldsymbol{\cdots}\xi}_{k{\rm times}}.
%\end{equation}
%Since for any $k\ge 1$, the map $\mathcal{C}_{T,\infty}\to\mathcal{C}_{T,\infty}$, 
%$\xi \mapsto\underbrace{\xi\boldsymbol{\cdots}\xi}_{k\text{ \rm times}}$, 
%is a (continuous) polynomial map in $\mathcal{C}_{T, \infty}$,
%\eqref{Taylor e^ xi} is the Taylor expansion of $e^\xi$, implying that
%$\xi \mapsto e^\xi$ is analytic.
%\end{proof}

For any $s>-1/2$ and any $n\ge 1$, the $n$th frequency \eqref{eq:formula_nth_frequency} of the
Benjamin-Ono equation extends from $\h_{r,0}^{s+ 1/2}$ to 
an analytic function on $\h_{c,0}^{s + 1/2}$ given by
\begin{equation}\label{eq:omega_n}
\omega_n(\zeta) = n^2+\Omega_n(\zeta),
\end{equation}
where 
\begin{equation}\label{eq:Omega_n}
\Omega_n(\zeta):=-2\sum_{k=1}^{n} k\,\zeta_{-k}\zeta_k
- 2n\sum_{k=n+1}^\infty\zeta_{-k}\zeta_k,\qquad\zeta = (\zeta_k)_{k \ne 0} \in\h_{c,0}^{s+1/2}.
\end{equation}
For $n\le-1$ we set 
\begin{equation}\label{eq:Omega_{-n}}
\Omega_n(\zeta):=-\Omega_{-n}(\zeta),\qquad\zeta\in\h_{c,0}^{s+1/2}.
\end{equation}
In \cite[Lemma 6.3]{GKT3} we prove the following

\begin{Lem}\label{eq:Omega-analytic}
For any $s\ge 0$ the map
\begin{equation}\label{eq:Omega-map}
\Omega : \h_{c,0}^{s+1/2}\to\ell^\infty_{c,0},\quad\zeta\mapsto\big(\Omega_n(\zeta)\big)_{n\ne 0},
\end{equation}
is analytic.
\end{Lem}

We now consider the curves $\xi^{(1)},\xi^{(2)}\in\mathcal{C}_{T,\infty}$, defined by
\begin{equation}\label{eq:gamma1}
\xi^{(1)} : [-T,T]\to\ell^\infty_{c,0},\qquad \xi_n^{(1)}(t):= e^{i\sign(n)\,n^2 t} , \ \ n\ne 0,
\end{equation}
and, respectively, 
\begin{equation}\label{eq:gamma2}
\xi^{(2)} : [-T,T]\to\ell^\infty_{c,0},\qquad \xi_n^{(2)}(t):=t , \ \ n\ne 0.
\end{equation}
The following corollary, proved in \cite[Section 6]{GKT3}, follows from Lemma \ref{lem:exp-analytic}, Lemma \ref{eq:Omega-analytic}, 
and the boundedness of the bilinear maps in \eqref{eq:multiplication_on_curves}.

\begin{Coro}\label{coro:analyticity_of_S_B}
For any $s\ge 0$, the map (cf. \eqref{formula for S_B})
\begin{equation}\label{eq:S_{B,T}}
\mathcal{S}_{B,T} : \h^{s+1/2}_{c,0}\to\mathcal{C}_{T, s+1/2},\quad
\zeta\mapsto\zeta\bcdot \xi^{(1)}\bcdot e^{i\,\Omega(\zeta)\bcdot\xi^{(2)}},
\end{equation}
is analytic.
\end{Coro}

\smallskip

\begin{proof}[Proof of Theorem \ref{th:well-posedness}(ii)] 
Let $s \ge 0$, $T>0$, and $w \in H^s_{r,0}$ and define $\xi:= \Phi(w) \in  \h^{s+1/2}_{r,0}$.
%Denote by $\rm{Iso}(w)$ the isospectral set of $w$, i.e.,  the set of all potentials $u \in H^s_{r,0}$ 
%with $\rm{spec}L_u = \rm{spec}L_w$ and let $\rm{Tor}(\Phi(w)) =\Phi(\rm{Iso}(w))$. Then $\rm{Iso}(w)$ is compact in $H^s_{r,0}$
%and $\rm{Tor}(\Phi(w))$ is compact in $\h^{s+1/2}_{r,0}$.
%
Since the curve $t \mapsto \mathcal S^t_B(\xi) \in \h^{s+1/2}_{r,0}$ is continuous, 
$$
K:= \{ S_B^t(\Phi(w)) \, | \,  |t| \le T \}
$$ 
is compact in $\h^{s+1/2}_{r,0}$.
Choose an open  neighborhood $V$ of $K$  in $\h^{s+1/2}_{c,0}$ so that 
$\Phi^{-1}$ extends to an analytic map $V \to H^s_{c,0}$ and has the property that
$\Phi^{-1} : V \to U:= \Phi^{-1}(V)$ is a bianalytic diffeomorphism.
Since $K$ is compact there exists $\varepsilon > 0$ so that for any $\zeta \in K$
the ball $\{ \zeta' \in \h^{s+1/2}_{c,0} | \ \| \zeta' - \zeta \|_{s+1/2} < \varepsilon \}$ is contained in $V$.
By Corollary \ref{coro:analyticity_of_S_B}, there exists a neighborhood $V_\xi$ of $\xi$ in $ \h^{s+1/2}_{c,0}$
so that for any $\zeta \in V_\xi$ and any $t \in [-T, T]$, $\| \mathcal S_B^t(\zeta) - \mathcal S_B^t(\xi) \|_{s+1/2} < \varepsilon$
and hence $\mathcal S_B^t(\zeta) \in V$. By shrinking $V_\xi$, if needed, we can assume that  $V_\xi \subseteq V$.
Hence $U_w := \Phi^{-1}(V_\xi)$ is well-defined and an open neighborhood of $w$ in $H^s_{c,0}$.
Furthermore, for any $u \in U_w$ and $|t| \le T$, $( \Phi^{-1}\circ S_{B,T}(t)\circ\Phi)(u)$ is well-defined.
The claimed result then follows from Corollary \ref{coro:analyticity_of_S_B}, Theorem \ref{th:Phi},
and from Lemma \ref{lem:push-forward} below.
\end{proof}
For the convenience of the reader, we record the following lemma of \cite[Section 6]{GKT3}, used in the proof of Theorem \ref{th:well-posedness}(ii) .
To state it, we first need to introduce some more notation.
Let $X$ be a complex Banach space with norm $\| \cdot \| \equiv \|\cdot \|_X$. For any $T > 0$, denote by $C\big([-T,T],X\big)$
the Banach space of continuous functions $x: [-T, T] \to X$, endowed with the supremum norm $\|x\|_{T, X} := \sup_{t \in [-T, T]} \| x(t) \|$. 
Furthermore, for any open neighborhood $U$ of $X$, denote by 
$C\big([-T,T],U\big)$ the subset of $C\big([-T,T], X\big)$, consisting of continuous functions $[-T, T ] \to X$ with values in $U$. 
The following lemma can be shown in a straightforward way. For the convenience of the reader it is proved in \cite[Section 6]{GKT3}.
\begin{Lem}\label{lem:push-forward}
Let $f : U\to Y$ be an analytic map from an open neighborhood  $U$ in $X$ where 
$X$ and $Y$ are complex Banach spaces. Then for any $T > 0$, the associated pushforward map
\[
f_* : C\big([-T,T],U\big)\to C\big([-T,T],Y\big),\quad [ t \mapsto x(t)]  \mapsto\big[t\mapsto f(x(t))\big], 
\]
is analytic. 
\end{Lem}
The proof of the Addendum to Theorem \ref{th:well-posedness} (ii), stated in Section \ref{Introduction},
is proved as its local version in \cite{GKT3}.

%\begin{proof}[Proof of Addendum to Theorem \ref{th:well-posedness}(ii)] 
%%\begin{Rem}\label{rem:improved_dependence_on_the_initial_data}
%Corollary \ref{coro:analyticity_of_S_B} continues to hold when $\mathcal{C}_{T, s+1/2}$ is replaced by 
%the Banach space
%$C^k\big([-T,T], \h^{s-2k + 1/2}_{c,0}\big)$, $k \ge 1$, of $k$ times continuously differentiable functions 
%$\xi : [-T, T] \to  \h^{s-2k + 1/2}_{c,0}$. For any $k \ge 1$, this follows from the proof of 
%Theorem \ref{th:well-posedness}(ii) and the fact that the curve $\zeta\bcdot e^{i\xi^{(1)}}$ 
%[resp. $e^{i\,\Omega(\zeta)\bcdot\xi^{(2)}}$], which appears as a factor on the right side of 
%\eqref{eq:S_{B,T}}, belongs to $C^k\big([-T,T],\h^{s-2k +1/2}_{c,0}\big)$ 
%[resp. $C^k\big([-T,T],\ell^\infty_{c,0}\big)$]. 
%By combining this with Theorem \ref{th:Phi} one obtains the claimed result.
%\end{proof}

%%%%%%%%%%%%%%%%%%%%%%%%%%%%%%%%%%%%%%%%%%%%%%%%%%%%%%
%%%%%%%%%%%%%%%%%%%%%%%%%%%%%%%%%%%%%%%%%%%%%%%%%%%%%%

\section{Analyticity of the delta map}\label{sec:the_delta_map}
The goal of this section is to prove Proposition \ref{prop:delta_n-analyticity} of Section \ref{sec:the_Birkhof_map},
saying that for any $w \in H^s_{r,0}$ with $s > -1/2$, the map
$$
\delta : U^\sigma_w \to \ell^1_{> n_2},\quad u\mapsto (\delta_n(u))_{n > n_2} ,
$$
is analytic where $\delta_n(u)$ is given by \eqref{64bis}.
We recall that $U^\sigma \equiv U^\sigma_w$,  $\sigma = \min \{s, 0\}$, is the neighborhood of $w$ in $H^\sigma_{c,0}$, 
satisfying the properties  {\em (NBH1)} - {\em (NBH2)}, introduced in Section \ref{sec.normalized eigenfunctions},
 and that $n_2 \ge n_0 > N$ is given by Lemma \ref{estim1 kappa+mu}.
Note that without loss of generality, we can assume that $-1/2 < s \le 0$, implying that $\sigma = s$
and hence throughout this section we assume that $-1/2 < s \le 0$.

The proof of  Proposition \ref{prop:delta_n-analyticity} is based on a vanishing lemma, which has been established in \cite{GKT3}.
We already know (cf. Section \ref{sec:the_Birkhof_map}) that each component of $\delta$,
\begin{equation}\label{94bis}
\delta_n(u) = \1 P_n(u) \big( SP_{n-1}(u)f_{n-1} - f_n \big) \,  | f_n \2 \in \C, \qquad n > n_2,
\end{equation}
is analytic on $U^s$, where $f_n \equiv f_n(w_N)$. Hence it suffices to show that there exists $n_3 \ge n_2$ so that
$$
\sum_{n > n_3} |\delta_n(u)| < \infty
$$ 
locally uniformly for $u$ in $U^s$. 
First we make some preliminary considerations. Recall that we write $u \in U^s$ as $u = w_N + v$.
For any $\lambda \in \partial D_{n-1}$ and any $n \ge n_2 > N$, one has by the resolvent identity,
\begin{equation}\label{94ter}
(L_u - \lambda)^{-1} 
%= (L_{w_N} - \lambda)^{-1}({\rm Id} - T_v (L_{w_N} - \lambda)^{-1})^{-1}
=  (L_{w_N} - \lambda)^{-1} + (L_u - \lambda)^{-1} T_v (L_{w_N} - \lambda)^{-1}.
\end{equation}
Since $P_{n-1}(u) = -\frac{1}{2\pi i} \oint\limits_{\partial D_{n-1}} (L_u - \lambda)^{-1}  d\lambda$
and $Sf_{n-1} = f_n$, it then follows by Cauchy's theorem that
$$
 -\frac{1}{2\pi i} \oint\limits_{\partial D_{n-1}} S(L_{w_N} - \lambda)^{-1} f_{n-1} d\lambda =
  -\frac{1}{2\pi i} \oint\limits_{\partial D_{n-1}} \frac{1}{n-1 - \lambda}  d\lambda \ f_n = f_n
$$
and hence by \eqref{94bis} and \eqref{94ter},
$$
\delta_n(u) =  -\frac{1}{2\pi i} \oint\limits_{\partial D_{n-1}}  \1 P_n(u) S (L_u - \lambda)^{-1} T_v (L_{w_N} - \lambda)^{-1} f_{n-1} | f_n \2  d\lambda .
$$
Expanding  
\begin{equation}\label{96bis}
(L_u - \lambda)^{-1}  = (L_{w_N} - \lambda)^{-1}({\rm Id} - T_v (L_{w_N} - \lambda)^{-1})^{-1}
\end{equation}
in the formula above in terms of the Neumann series
for $({\rm Id} - T_v (L_{w_N} - \lambda)^{-1})^{-1}$ and writing $P_n$ for $P_n(u)$, one obtains
$$
\delta_n(u) =  - \sum_{m \ge 1} \frac{1}{2\pi i} \oint\limits_{\partial D_{n-1}}  \1 P_n S (L_{w_N} - \lambda)^{-1} (T_v (L_{w_N} - \lambda)^{-1})^m f_{n-1} | f_n \2  d\lambda .
$$
By expressing $(T_v (L_{w_N} - \lambda)^{-1})^m f_{n-1}$ in terms of the basis $f_n$, $n \ge 0$, one obtains
\begin{equation}\label{formula1 delta_n}
\delta_n(u) = \sum_{m\ge 1} \sum_{\tb{k_j \ge 0}{1\le j \le m}} C_{u, n-1}(n-1, k[1,m]) \1 P_n S f_{k_m}  | f_n \2 ,
\end{equation}
where 
$$
k[1, m] := k_1, \ldots, k_m .
$$
The coefficients $C_{u, n-1}(n-1, k[1,m])$ can be computed as follows.
For $m=1$ one has 
$$
\begin{aligned}
 (L_{w_N} - \lambda)^{-1} T_v (L_{w_N} - \lambda)^{-1} f_{n-1}
%& =  (L_{w_N} - \lambda)^{-1} \frac{\Pi (vf_{n-1})}{n-1 - \lambda}\\
  = \sum_{k_1 \ge 0} \frac{\1 vf_{n-1} | f_{k_1} \2}{(n-1 - \lambda)(\lambda_{k_1} - \lambda)} f_{k_1} ,
 \end{aligned}
 $$
 where $\lambda_{k_1} \equiv  \lambda_{k_1}(w_N)$ and where we used that $\lambda_{n-1} = n-1$. Hence
 $$
 \begin{aligned}
 - \frac{1}{2\pi i} \oint\limits_{\partial D_{n-1}} & \1 P_n S (L_{w_N} - \lambda)^{-1} T_v (L_{w_N} - \lambda)^{-1} f_{n-1} | f_n \2  d\lambda \\
& =  \sum_{k_1 \ge 0} C_{u, n-1}(n-1, k_1) \1 P_n S f_{k_1}  | f_n \2
 \end{aligned}
 $$
where
$$
C_{u, n-1}(n-1, k_1) = -  \frac{1}{2\pi i} \oint\limits_{\partial D_{n-1}}   \frac{\1 vf_{n-1} | f_{k_1} \2}{(n-1 - \lambda)(\lambda_{k_1} - \lambda)} d \lambda  .
$$ 
Next let us compute the term $m=2$ in \eqref{formula1 delta_n}. One has
$$
(L_{w_N} - \lambda)^{-1} (T_v (L_{w_N} - \lambda)^{-1})^2 f_{n-1}
=(L_{w_N} - \lambda)^{-1}\!\!\sum_{\tb{k_j \ge 0}{j=1,2}}
\frac{\1 vf_{n-1} | f_{k_1} \2 \1 vf_{k_1} | f_{k_2} \2 }{(n-1 - \lambda)(\lambda_{k_1} - \lambda)} f_{k_2},
$$
yielding
 $$
 \begin{aligned}
& - \frac{1}{2\pi i} \oint\limits_{\partial D_{n-1}}  \1 P_n S (L_{w_N} - \lambda)^{-1} (T_v (L_{w_N} - \lambda)^{-1})^2 f_{n-1} | f_n \2  d\lambda \\
& = \sum_{\tb{k_j \ge 0}{j=1,2}} C_{u, n-1}(n-1, k_1, k_2) \1 P_n S f_{k_2}  | f_n \2
 \end{aligned}
 $$
where 
$$
C_{u, n-1}(n-1, k_1, k_2) =
  -  \frac{1}{2\pi i} \oint\limits_{\partial D_{n-1}}   
  \frac{\1 vf_{n-1} | f_{k_1} \2 \1 vf_{k_1} | f_{k_2} \2 }{(n-1 - \lambda)(\lambda_{k_1} - \lambda) (\lambda_{k_2} - \lambda)} d \lambda \, .
$$
Arguing similarly as in the cases $m=1$ and $m=2$,  for any $m \ge 1$ and $(k_j)_{1 \le j \le m}$, 
the coefficient $C_{u, n-1}(n-1, k[1,m])$  can be computed as
$$
C_{u, n-1}(n-1, k[1,m]) =
  -  \frac{1}{2\pi i} \oint\limits_{\partial D_{n-1}}   
  \frac{ \1 vf_{n-1} | f_{k_1} \2 \cdots  \1 vf_{k_{m-1}} | f_{k_m} \2}{(n-1 - \lambda)\prod_{j=1}^m( \lambda_{k_j} - \lambda)}  d \lambda .
$$
By passing to the variable $\mu:= \lambda -n + 1$ in the contour integral above we get
$$
C_{u, n-1}(n-1, k[1,m]) = A(0, k_{\lambda; n-1}[1, m]) B_u(n-1, k[1,m])
$$
where for any $\ell \ge 1$, $k_j \in \Z$, $1\le j \le \ell$,
$$
B_u(k[1,\ell]) :=   \1 vf_{k_1} | f_{k_2} \2  \1 vf_{k_2} | f_{k_3} \2  \cdots  \1 vf_{k_{\ell-1}} | f_{k_\ell} \2
$$
and any $z_1, \ldots, z_{\ell}$ with $z_j \in \C \setminus \partial D_0,$
\begin{equation}\label{def A(...)}
A(z_1, \ldots, z_\ell) :=  -  \frac{1}{2\pi i} \oint\limits_{\partial D_{0}} \frac{d \mu}{\prod_{j=1}^\ell (z_j - \mu)}  
\end{equation}
and 
$$
k_{\lambda}[1, m] := \lambda_{k_1}, \ldots, \lambda_{k_m} \, , \qquad
k_{\lambda; j}[1, m] :=   \lambda_{k_1} -j, \ldots, \lambda_{k_m} -j \, .
$$
For the sequel, it is convenient to define the corresponding sequences
when $[1, m]$ is replaced by $(1, m]$, $(1, m)$, \ldots , 
$$
k(1, m] := k[2,m],  \qquad k(1, m) := k[2,m-1] , \qquad \ldots \ \ \  .
$$
Altogether, we obtain the following formula,
\begin{equation}\label{formula2 delta_n}
\delta_n(u) = \sum_{m \ge 1}  \sum_{\tb{k_j \ge 0}{1\le j \le m}}A(0, k_{\lambda; n-1}[1, m]) B_u(n-1, k[1,m]) \1 P_n S f_{k_m}  | f_n \2 .
\end{equation}
Now let us turn to the expression $\1 P_n S f_{k_m}  | f_n \2$. An expansion of $P_n$ corresponding to the one of $P_{n-1}$ yields
\begin{equation}\label{100bis}
\1 P_n S f_{k_m}  | f_n \2 =
- \sum_{r \ge 0}  \frac{1}{2\pi i} \oint\limits_{\partial D_{n}}  \1  (L_{w_N} - \lambda)^{-1} (T_v (L_{w_N} - \lambda)^{-1})^r S f_{k_m} | f_n \2  d\lambda \, .
\end{equation}
For the term $r=0$ in the latter series, one has 
$$
\1 (L_{w_N} - \lambda)^{-1}  S f_{k_m} | f_n \2  = \frac{\1 Sf_{k_m} | f_n \2}{n-\lambda} =  \frac{\1 f_{k_m} | f_{n-1} \2}{n-\lambda}
$$
where we used that $S^* f_n = e^{-ix} f_n = f_{n-1}$ since $n-1 \ge N$. Hence by Cauchy's theorem
$$
-  \frac{1}{2\pi i} \oint\limits_{\partial D_{n}} \1 (L_{w_N} - \lambda)^{-1}  S f_{k_m} | f_n \2 d \lambda = \delta_{k_m (n-1)} .
$$
To compute the terms with $r \ge 1$ in the latter series, we distinguish between the case $k_m \ge N$
and the case $k_m < N$.

\smallskip
\noindent
{\em Case $k_m \ge N$.} In this case, $S f_{k_m} = S e^{ik_m x} g_\infty = f_{k_m +1}$ and hence by \eqref{100bis}, 
\begin{align}\label{formula1 P_n etc}
\1 P_n S f_{k_m}  | f_n \2 & = \delta_{k_m( n - 1)} + C_{u,n}(k_m +1, n) \nonumber \\
& + \sum_{r\ge 2} \sum_{\tb{\ell_j \ge 0}{1 \le j \le r}} C_{u,n}(k_m+1, \ell[1,r], n) ,
\end{align}
where passing to the integration variable $\mu:= \lambda - n$, one sees that
$$
\begin{aligned}
 & C_{u,n}(k_m+1, n) =  -  \frac{1}{2\pi i} \oint\limits_{\partial D_{n}}   \frac{ \1 vf_{k_m + 1} | f_{n} \2}{(k_m +1 - \lambda) (n - \lambda)}  d \lambda \\
 & =  -  \frac{1}{2\pi i} \oint\limits_{\partial D_{0}}   \frac{ \1 vf_{k_m + 1} | f_{n} \2 }{(k_m +1 - n - \mu)} \frac{d \mu}{-\mu} 
 =A(k_m+1 -n, 0)B_u(k_m +1, n) .
\end{aligned}
$$
Passing from $r$ to $r-1$ as summation index in the sum \eqref{formula1 P_n etc} one gets for $r \ge 1$,
$$
C_{u,n}(k_m+1, \ell[1,r], n) = A(k_m + 1 -n, \ell_{\lambda; n}[1, r], 0) \cdot B_u(k_m+1, \ell[1,r], n) \, .
$$
Writing $k_{m+j}$ for $\ell_j$, $1 \le j \le r$, and setting $k[j , j'] := \emptyset$ for any $j > j'$, we arrive at the following formula in the case $k_m \ge N$,
\begin{align}\label{formula2 P_n etc}
&\1 P_n S f_{k_m}  | f_n \2  = \delta_{k_m (n-1)}  \, +  \\
& \sum_{r\ge 0} \sum_{\tb{k_j \ge 0}{m <  j \le m+ r}}  A(k_m + 1 -n, \, k_{\lambda; n}(m, m+ r], 0) \cdot B_u(k_m+1, \, k(m, m+r], n) \, . \nonumber
\end{align}

\smallskip
\noindent
{\em Case $k_m < N$.} For convenience we set $\alpha := k_m$. Then $0 \le \alpha < N$.
Since $\1 Sf_\alpha | f_{k+1}  \2 = \1 f_\alpha | f_k \2 = 0$ for any $k\ge N$, one has
\begin{equation}\label{formula3 P_n etc}
\1 P_n S f_{\alpha} | f_n \2 = \sum_{0 \le \beta \le N} \1 Sf_{\alpha} | f_\beta \2 \1 P_n f_\beta | f_n \2 .
\end{equation}
We now expand $\1 P_n f_\beta | f_n \2$, $0 \le \beta \le N$. Since $\1 f_\beta | f_n \2 = 0$ for any $n > N$ one infers 
from \eqref{96bis} that
$$
\1 P_n f_\beta | f_n \2 = \sum_{r \ge 0} \sum_{\tb{\ell_j \ge 0}{1 \le j \le r}} C_{u,n}(\beta, \ell[1,r], n) .
$$
For $r=0$ (recall that by definition, $\ell[1, 0] = \emptyset$)
$$
 C_{u,n}(\beta,  n) =  -  \frac{1}{2\pi i} \oint\limits_{\partial D_{n}}   \frac{ \1 vf_{\beta} | f_{n} \2}{(\lambda_\beta - \lambda) (n - \lambda)}  d \lambda
 = A(\lambda_\beta -n, 0)B_u(\beta, n) 
$$
and for any $r \ge 1$,
$$
C_{u,n}(\beta, \ell[1,r], n) =  
 -  \frac{1}{2\pi i} \oint\limits_{\partial D_{n}}   
  \frac{ \1 vf_{\beta} | f_{\ell_1} \2 \cdots  \1 vf_{\ell_r} | f_n \2}{(\lambda_\beta - \lambda)(n-\lambda)\prod_{j=1}^r( \lambda_{\ell_j} - \lambda)}  d \lambda \, .
$$
By passing to the integration variable $\mu:= \lambda - n$ in the contour integrals, one arrives at
\begin{equation}\label{formula4 P_n etc}
 \1 P_n f_\beta | f_n \2 = \sum_{r \ge 0} \sum_{\tb{\ell_j \ge 0}{1 \le j \le r}}
 A(\lambda_\beta -n, \ell_{\lambda; n}[1, r], 0) \cdot B_u(\beta, \ell[1,r], n) .
\end{equation}
For notational convenience, we suppress $u$ in quantities such as $B_u$ or $\delta_n(u)$ in the sequel.
By combining  \eqref{formula2 delta_n} and \eqref{formula2 P_n etc} - \eqref{formula4 P_n etc}, one sees that 
\begin{equation}\label{105bis}
\delta_n = I_n + II_n
\end{equation} 
where
$$
I_n :=  \sum_{m \ge 1}  \sum_{\tb{k_j \ge 0}{1\le j \le m}}A(0, k_{\lambda; n-1}[1, m]) \cdot B(n-1, k[1,m])  \delta_{k_m (n-1)} ,
$$
$$
II_n :=   \sum_{m \ge 1}  \sum_{\tb{k_j \ge 0}{1\le j \le m}}A(0, k_{\lambda; n-1}[1, m]) \cdot B(n-1, k[1,m]) \big(\1 P_n S f_{k_m}  | f_n \2 - \delta_{k_m (n-1)} \big) .
$$
We rewrite $I_n$ and $II_n$ as
\begin{equation}\label{105ter}
I_n = \delta_n^{(0)} + \mathcal R_n^{(0)} , \qquad II_n = \delta_n^{(1)} + \mathcal R_n^{(1)} ,
\end{equation}
with $\delta_n^{(0)}$, $\delta_n^{(1)}$, $\mathcal R_n^{(0)}$, and $\mathcal R_n^{(1)}$ defined as follows:
the term $\delta_n^{(0)}$ is obtained by passing from $m$ to $d:= m-1$ as summation index and using that for $d=0$, $A(0, 0)\cdot B(n-1, n-1) = 0$, 
\begin{equation}\label{delta^(0)}
 \delta_n^{(0)}  := \sum_{d \ge 1}  \sum_{\tb{k_j \ge N, 1\le j \le d}{(k_{d+1} = n-1)}}A(0, k_{\lambda; n-1}[1, d], 0) \cdot B(n-1, k[1,d], n-1) ,
\end{equation}
whereas
\begin{align}\label{delta^(1)}
 \delta_n^{(1)}  & := \sum_{\tb{m \ge 1}{r \ge 0}}  \sum_{\tb{k_j \ge 0, 1\le j \le m}{k_m \ge N}}  \sum_{\tb{k_j \ge 0}{m < j \le m+r}}
 A(0, k_{\lambda; n-1}[1, m]) \cdot B(n-1, k[1,m])  \nonumber \\
& \cdot A(k_m + 1 -n, \, k_{\lambda; n}(m, m+ r], 0) \cdot B(k_m+1, \, k(m, m+r], n) ,
\end{align}
and 
$$
\mathcal R_n^{(0)} : = \sum_{0 \le  \alpha < N} \mathcal R_{n, \alpha}^{(0)} , \qquad
\mathcal R_n^{(1)} : =   \sum_{\tb{0 \le  \alpha <N}{0 \le \beta \le N}} \1 Sf_{\alpha} | f_\beta \2 \mathcal R_{n, \alpha, \beta}^{(1)} ,
$$
with
\begin{align}\label{mathcal R^(0)}
 \mathcal R_{n, \alpha}^{(0)} & : = A(0, \lambda_\alpha -n +1, 0) \cdot B(n-1, \alpha, n-1) \\
&+ \sum_{d \ge 2} \sum_{i=1}^{d} \sum_{\tb{k_j \ge N, 1\le j <i, k_i=\alpha}{k_j \ge 0, i < j \le d}}
A(0, k_{\lambda; n-1}[1, d], 0) \cdot B(n-1, k[1,d], n-1)  \nonumber 
\end{align}
and
\begin{align}\label{mathcal R^(1)}
 \mathcal R_{n, \alpha, \beta}^{(1)}  & : = \sum_{\tb{m \ge 1}{r \ge 0}}  \sum_{\tb{k_j \ge 0,1\le j < m }{(k_m = \alpha)}}  \sum_{\tb{\ell_j \ge 0}{1\le j \le r}}
 A(0, k_{\lambda; n-1}[1, m), \lambda_\alpha - n + 1)   \nonumber \\
& \cdot B(n-1, k[1,m), \alpha) \cdot A(\lambda_\beta -n, \, \ell_{\lambda; n}[1, r], 0) \cdot B(\beta, \ell[1, r], n) .
\end{align}
We now expand $\delta_n^{(1)}$, given by \eqref{delta^(1)}, further,
\begin{equation}\label{}
\delta_n^{(1)} = \delta_n^{(2)} +  \mathcal R_{n}^{(2,1)} +  \mathcal R_{n}^{(2,2)} ,
\end{equation}
where
\begin{align}\label{delta^(2)}
 \delta_n^{(2)}  &: = \sum_{\tb{m \ge 1}{r \ge 0}}  \sum_{\tb{k_j \ge N}{ 1\le j \le m+r}}  
 A(0, k_{\lambda; n-1}[1, m]) \cdot B(n-1, k[1,m])  \nonumber \\
& \cdot A(k_m + 1 -n, \, k_{\lambda; n}(m, m+ r], 0) \cdot B(k_m+1, \, k(m, m+r], n) . 
\end{align}
The remainders are given by
$$
 \mathcal R_{n}^{(2,1)} := \sum_{0 \le \alpha < N}  \mathcal R_{n, \alpha}^{(2,1)} , \qquad
  \mathcal R_{n}^{(2,2)} := \sum_{0 \le \alpha < N}  \mathcal R_{n, \alpha}^{(2,2)}
$$
where
\begin{align}\label{mathcal R^(2,1)}
& \mathcal R_{n, \alpha}^{(2,1)}  : = 
  \sum_{m \ge 1} \sum_{i=1}^{m-1} \sum_{\tb{k_j \ge N, 1\le j <i, k_i=\alpha} {k_j \ge 0, i < j <m, k_m\ge N} }
A(0, k_{\lambda; n-1}[1, m]) \cdot B(n-1, k[1,m])   \\
&  \cdot  \sum_{r \ge 0} \sum_{\tb{k_j \ge 0}{m < j \le m+r}}A(k_m + 1 -n, \, k_{\lambda; n}(m, m+ r], 0) \cdot B(k_m+1, \, k(m, m+r], n) \nonumber
\end{align}
and
\begin{align}\label{mathcal R^(2,2)}
& \mathcal R_{n, \alpha}^{(2,2)}  : = 
  \sum_{m \ge 1} \sum_{\tb{k_j \ge N}{1 \le j \le m}} 
A(0, k_{\lambda; n-1}[1, m]) \cdot B(n-1, k[1,m])   \\
&  \cdot \sum_{r \ge 0} \sum_{i=1}^{r}   \sum_{\tb{\ell_j \ge N, 1\le j <i}{\ell_i=\alpha, \ell_j \ge 0, i < j \le r} }
A(k_m + 1 -n, \, \ell_{\lambda; n}[1,r], 0) \cdot B(k_m+1, \, \ell[1, r], n) . \nonumber
\end{align}
Combining \eqref{105bis} (definition of $\delta_n$), \eqref{105ter} (definition $I_n$, $II_n$)
with \eqref{delta^(0)} - \eqref{mathcal R^(2,2)},
 we obtain the following expansion of $\delta_n$,
\begin{equation}\label{115bis}
\delta_n =  \delta_n^{(0)}  +  \delta_n^{(2)}  + \mathcal R_n,
\end{equation}
\begin{equation}\label{115ter}
 \mathcal R_n = 
  \sum_{0 \le \alpha < N} \mathcal R_{n, \alpha}^{(0)} + \sum_{\tb{0 \le  \alpha <N}{0 \le \beta \le N}} \1 Sf_{\alpha} | f_\beta \2 \mathcal R_{n, \alpha, \beta}^{(1)} 
  + \sum_{0 \le \alpha < N} ( \mathcal R_{n, \alpha}^{(2,1)} +  \mathcal R_{n, \alpha}^{(2,2)} ).
\end{equation}
Let us now consider $ \delta_n^{(0)}$ and  $\delta_n^{(2)}$ in more detail. We begin with  $\delta_n^{(0)}$.
Introducing $\ell_j:= k_j -n +1$, $1 \le j \le d$, and using that  $f_{k} = g_\infty e^{ikx}$, $\lambda_k = k$ for any $k \ge N$ (since $w_N \in \mathcal U_N$)
and that $|g_\infty(x) | = 1$, it follows that $B(n-1, k[1,d], n-1)$ in the formula \eqref{delta^(0)} for $\delta_n^{(0)}$ is given by
$$
\begin{aligned}
& B(n-1, k[1,d], n-1) = \1 vf_{n-1} | f_{k_1} \2 \cdots \1 vf_{k_d} | f_{n-1} \2 \\
& = \widehat v(k_1 - n +1) \widehat v(k_2 - k_1) \cdots \widehat v(n-1-k_d) = G_v(0, \ell[1, d], 0)
\end{aligned}
$$
where 
$$
G_v(0, \ell[1, d], 0) := \widehat v(\ell_1) \widehat v(\ell_2 - \ell_1) \cdots \widehat v(\ell_d - \ell_{d-1}) \widehat v(-\ell_{d}).
$$
(Recall that $u = w_N + v$ and hence $G_v$ depends on $v$, but not on $w_N$.) As a consequence,
$$
 \delta_n^{(0)}  = \sum_{d \ge 1}  \sum_{\tb{\ell_j > N-n}{1\le j \le d}}A(0, \ell[1, d], 0) \cdot G_v(0, \ell[1, d], 0).
$$
Similarly, we obtain a formula for $\delta_n^{(2)}$. Introducing
$$
\ell_j:= k_j - n +1, \ \ \ 1 \le j \le m, \qquad  \ell_j:= k_j - n, \ \ \ m <  j \le m + r,
$$
$ B(n-1, k[1,m])$ and $ B(k_m+1, k(m,m+r], n)$ in \eqref{delta^(1)} can be computed as
$$
\begin{aligned}
& B(n-1, k[1,m] ) =  \1 vf_{n-1} | f_{k_1} \2 \cdots \1 vf_{k_{m-1}} | f_{m} \2 \\
&=  \widehat v(k_1 - n +1) \widehat v(k_2 - k_1) \cdots \widehat v(k_m -k_{m-1})\\
&=  \widehat v(\ell_1) \widehat v(\ell_2 - \ell_1) \cdots \widehat v(\ell_m - \ell_{m-1}) = G_v(0, \ell[1, m])
\end{aligned}
$$
and 
$$
\begin{aligned}
& B(k_m+1, k(m,m+r], n) =  \1 vf_{k_m+1} | f_{k_{m+1}} \2 \cdots \1 vf_{k_{m+r}} | f_{n} \2 \\
&=\widehat v(k_{m+1} - k_m -1) \widehat v(k_{m+2} - k_{m+1}) \cdots \widehat v(n -k_{m+r})\\
&=  \widehat v(\ell_{m+ 1} - \ell_m) \widehat v(\ell_{m+2} - \ell_{m+1}) \cdots \widehat v(-\ell_{m +r}) = G_v( \ell[m, m+r], 0) .
\end{aligned}
$$
Since 
$$
G_v(0, \ell[1, m]) \cdot G_v( \ell[m, m+r], 0) = G_v(0,  \ell[1, m+r], 0)
$$
it follows that
$$
 \delta_n^{(2)}  = \sum_{\tb{m \ge 1}{r \ge 0}}  \sum_{\tb{\ell_j > N-n,  1\le j \le m}{\ell_j \ge N-n, m< j \le m+r}}  
 A(0, \ell[1, m]) \cdot A(\ell[m, m+r], 0) \cdot G_v(0, \ell[1, m+r], 0).
$$
Arguing as in \cite[Section 5]{GKT3}, we write $  \delta_n^{(2)}  =  \delta_n^{(3)} + \mathcal R^{(3)}_n$
where (taking again into account that $\lambda_k = k$ for $k \ge N$)
$$
 \delta_n^{(3)}  = \sum_{\tb{m \ge 1}{r \ge 0}}  \sum_{\tb{\ell_j > N-n}{1\le j  \le m+r}}  
 A(0, \ell[1, m]) \cdot A(\ell[m, m+r], 0) \cdot G_v(0, \ell[1, m+r], 0).
$$
and
$$
\begin{aligned}
& \mathcal R^{(3)}_n =  \sum_{\tb{m \ge 1}{r \ge 0}}  \sum_{\tb{\ell_j > N-n}{1\le j  \le m}}  
\sum_{i=1}^r  \sum_{\tb{\ell_{m+j} > N-n, 1\le j  <i, \ell_{m+i} = N-n}{\ell_{m+j} \ge N-n, i < j \le r}}  A(0, \ell[1, m]) \\
& \cdot A(\ell[m, m+r], 0) \cdot G_v(0, \ell[1, m+r], 0).
\end{aligned}
$$
We then arrive at 
$$
\delta_n^{(0)} +  \delta_n^{(2)}= \delta_n^{(0)} + \delta_n^{(3)} +  \mathcal R_{n}^{(3)}
$$
where 
$$
 \delta_n^{(0)} + \delta_n^{(3)} = \sum_{d \ge 1} \sum_{\tb{\ell_j > N-n}{1\le j  \le d}} 
 \mathcal D(\ell[1, d])  \cdot G_v(0, \ell[1, d], 0)
$$
and $\mathcal D(\ell[1, d])$ is given by \cite[(72)]{GKT3}, 
$$
\mathcal D(\ell[1, d]) = \sum_{1 \le m \le d} \big( A(0, \ell[1, m])A( \ell[m, d], 0) + A(0, \ell[1, d], 0) \big) .
$$
(Note the change in notation when compared with \cite{GKT3}.) 
By \cite[Lemma 5.1 (Vanishing Lemma)]{GKT3}, $\mathcal D(\ell[1, d]) = 0$. Hence by \eqref{115bis}-\eqref{115ter}, one gets
\begin{equation}\label{formula5 delta_n}
\delta_n = 
  \sum_{0 \le \alpha < N} \mathcal R_{n, \alpha}^{(0)} + \sum_{\tb{0 \le  \alpha <N}{0 \le \beta \le N}} \1 Sf_{\alpha} | f_\beta \2 \mathcal R_{n, \alpha, \beta}^{(1)} 
  + \sum_{0 \le \alpha < N} ( \mathcal R_{n, \alpha}^{(2,1)} +  \mathcal R_{n, \alpha}^{(2,2)} )+  \mathcal R_{n}^{(3)}.
\end{equation}
Arguing as in the proof of \cite[Lemma 5.2]{GKT3} one sees that for some constant $C \ge 1$, only depending on $s$,
$$
\sum_{n > n_2} |  \mathcal R_{n}^{(3)} | \le C \|v\|_s^3, \qquad \forall \, u= w_N + v \in  B^s_{c,0}(w_N, 1/C_{M,s})
$$
The remaining terms in \eqref{formula5 delta_n} are estimated, using the results of Section \ref{sec:Psi}.

\smallskip

\noindent
{\em  Estimate of $\sum_{\tb{0 \le  \alpha <N}{0 \le \beta \le N}} \1 Sf_{\alpha} | f_\beta \2 \mathcal R_{n, \alpha, \beta}^{(1)}$.}
First note that
$$
|  \1 Sf_{\alpha} | f_\beta \2 | \le \frac{1}{2\pi} \int_0^{2\pi} |e^{ix}f_\alpha \overline{ f_\beta} | d x \le \|f_\alpha \| \|f_\beta\| = 1 
$$
and that  by Cauchy-Schwarz,  $\sum_{n > n_2} |\mathcal R_{n, \alpha, \beta}^{(1)}|$ (cf. \eqref{mathcal R^(1)}) can be estimated as 
$$
\sum_{n > n_2} |\mathcal R_{n, \alpha, \beta}^{(1)}| \le \sum_{m \ge 1}(I^{(1)}_{m, \alpha})^{1/2} \cdot \sum_{r \ge 0} (II^{(1)}_{r, \beta})^{1/2}
$$ 
where
$$
I^{(1)}_{m,\alpha}:= \sum_{n > n_2} \big( \sum_{\tb{k_j \ge 0, 1 \le j < m}{k_m=\alpha}}
 |A(0, k_{\lambda; n-1}[1, m), \lambda_\alpha - n + 1)| \cdot |B(n-1, k[1,m), \alpha) |  \big)^2 ,
$$
$$
II^{(1)}_{r,\beta}:=  \sum_{n > n_2} \big(  \sum_{\tb{\ell_j \ge 0}{ 1 \le j \le r}}
| A(\lambda_\beta -n, \, \ell_{\lambda; n}[1, r], 0)| \cdot |B(\beta, \ell[1, r], n)| \big)^2 .
$$
To bound $II^{(1)}_{r,\beta}$, we estimate $\sup_{\mu \in \partial D_0}$ of the absolute value of the integrand 
in the expression for $A(\lambda_\beta -n, \, \ell_{\lambda; n}[1, r], 0)$ (cf. \eqref{def A(...)}), using that
\begin{equation}\label{estimate1 1/(n-N)}
\sup_{\mu \in \partial D_0} \frac{1}{ |\lambda_\beta - n - \mu |} \le \frac{1}{n-\frac13 -N} \le C_1 \frac{1}{n} \le C_1 n^{2s} \, , \qquad \forall n > n_2
\end{equation}
$$
\sup_{\mu \in \partial D_0} \frac{1}{ |\lambda_\ell - n - \mu |} \le \frac{5}{|\ell - n | +1}, \ \ \forall \, \ell  \ge 0, \ n > n_2 .
$$
We obtain
$$
\begin{aligned}
& | A(\lambda_\beta -n, \, \ell_{\lambda; n}[1, r], 0)| \cdot |B(\beta, \ell[1, r], n)| \le C_1 n^s 
\frac{ 5^r |\1 vf_\beta | f_{\ell_1}  \2 |  \cdots  |\1 vf_{\ell_r} | f_n  \2 |   }{\prod_{j=1}^r (|\ell_j -n| +1)}\\
& \le C_1 n^s 
\frac{ 5^r | \widehat{vg_\beta \overline g_{\ell_1}}(\ell_1 - \beta) |   | \widehat{vg_{\ell_1} \overline g_{\ell_2}}(\ell_2 - \ell_1) | \cdots  
 | \widehat{v g_{\ell_r} \overline g_\infty}(n - \ell_r) |  }{\prod_{j=1}^r (|\ell_j -n| +1)} .
\end{aligned}
$$
Arguing as in the proof of Lemma \ref{key lemma} one concludes that $ \sum_{r \ge 0} (II^{(1)}_{r, \beta})^{1/2}$ is bounded by
$$
C_1 \sum_{r \ge 0} 5^r \Big( \sum_{n \ge n_2} n^{2s} \big(
  \sum_{\tb{\ell_j \ge 0}{ 1 \le j \le r}} \frac{ | \widehat{vg_\beta \overline g_{\ell_1}}(\ell_1 - \beta) |   
  %| \widehat{vg_{\ell_1} \overline g_{\ell_2}}(\ell_2 - \ell_1) |
   \cdots   | \widehat{v g_{\ell_r} \overline g_\infty}(n - \ell_r) |  }{\prod_{j=1}^r (|\ell_j -n| +1)}
 \big)^2 \Big)^{1/2} < \infty
$$
locally uniformly for $u=w_N + v \in B^s_{c,0}(w_N, 1/C_{M,s})$. 
Similarly, one can bound $ \sum_{m \ge 1} (I_{m, \alpha})^{1/2}$.
Indeed, for any $0 \le \alpha < N$ one has 
\begin{equation}\label{estimate2 1/(n-N)}
\frac{1}{|\lambda_\alpha - \lambda|} \le \frac{1}{(n-1) - \frac13 -(N-1)} \le \frac{1}{  n -\frac 13 -N} \le C_1 n^{2s} 
\end{equation}
and hence for any $n \ge n_1$
$$
\begin{aligned}
& |A(0, k_{\lambda; n-1}[1, m), \lambda_\alpha - n + 1)| \cdot |B(n-1, k[1,m), \alpha) | \\
& \le C_1 n^s 
\frac{ 5^m | \widehat{vg_\infty \overline g_{k_1}}(k_1 - n + 1) |  
% | \widehat{vg_{k_1} \overline g_{k_2}}(k_2 - k_1) | 
\cdots   | \widehat{v g_{k_{m-1}} \overline g_\alpha}(\alpha - k_{m-1}) |  }{\prod_{j=1}^{m-1} (|k_j -n + 1| +1)} .
\end{aligned}
$$
Again arguing as in the proof of Lemma \ref{key lemma} one concludes that $\sum_{m \ge 0} (I^{(1)}_{m, \alpha})^{1/2}$ 
is bounded by
$$
C_1\!\!\sum_{m \ge 0}\!\!5^m \big(\!\!\sum_{n \ge n_2}\!\!n^{2s} \big(
\!\!\!\sum_{\tb{k_j \ge 0}{ 1 \le j < m}}\!\!\!\frac{ | \widehat{vg_\infty \overline g_{k_1}}(k_1 - n +1) |   
%| \widehat{vg_{k_1} \overline g_{k_2}}(k_2 - k_1) |
...| \widehat{v g_{k_{m-1}} \overline g_\alpha}(\alpha - k_{m-1}) |  }{\prod_{j=1}^{m-1} (|k_j -n + 1| +1)}
 \big)^2 \big)^{1/2}\!<\!\infty
$$
locally uniformly for $u=w_N + v \in B^s_{c,0}(w_N, 1/C_{M,s})$. 

\smallskip

\noindent
{\em  Estimate of $\sum_{0 \le \alpha < N} \mathcal R_{n, \alpha}^{(0)}$.}
Let us first consider the term 
$$A(0, \lambda_\alpha -n +1, 0) \cdot B(n-1, \alpha, n-1)
$$ 
in \eqref{mathcal R^(0)}. 
By Cauchy-Schwarz and \eqref{estimate2 1/(n-N)},
$$
\begin{aligned}
&\sum_{n \ge n_2} |A(0, \lambda_\alpha -n +1, 0)| \cdot |B(n-1, \alpha, n-1)| \\
& \le \sum_{n \ge n_2}\frac{| \widehat{vg_\infty \overline g_{\alpha}}(\alpha - n + 1) | | \widehat{vg_\alpha \overline g_\infty }(n - 1 - \alpha) |  }{n - \frac 13 - N}\\
&\le C_1 \big( \sum_{n \ge n_2} n^{2s} | \widehat{vg_\infty \overline g_{\alpha}}(\alpha - n + 1) |^{2s} \big)^{1/2}
\big( \sum_{n \ge n_2} n^{2s} | \widehat{vg_\alpha \overline g_{\infty}}(n - 1 - \alpha) |^{2s} \big)^{1/2} \\
&\le C_2 \|v\|_s^2
\end{aligned}
$$
where $C_2$ is a constant, only depending on $s$ and $M$. (Here we used Lemma \ref{estimates g_n} to bound $g_\alpha \equiv g_\alpha(w_N)$ 
and $g_\infty \equiv g_\infty(w_N)$.)
By estimating $\sup_{\mu \in \partial D_0}$ of the absolute value of the integrand 
in the expression for $A(0, k_{\lambda; n-1}[1, d], 0)$ (cf. \eqref{def A(...)}) and  using \eqref{estimate2 1/(n-N)},
one obtains in the case $k_i = \alpha$,
$$
\begin{aligned}
&|A(0, k_{\lambda; n-1}[1, d], 0)| \cdot |B(n-1, k[1,d], n-1) | \\
& \le C_1 n^{2s} 5^{d}
 \frac{   |\widehat{vg_\infty \overline g_{k_1}}(k_1 - n + 1) | | \widehat{vg_{k_1} \overline g_{k_2}}(k_2 - k_1) |
\cdots | \widehat{vg_{k_{d}} \overline g_\infty} (n-1 - k_{d}) | }{\prod_{\tb{ 1 \le j \le d}{j \ne i}} (|k_j -n +1| +1)} ,
\end{aligned}
$$
yielding  
\begin{equation}\label{sum R_{n, alpha}^(0) }
\sum_{n \ge n_2} | \mathcal R_{n, \alpha}^{(0)} | \le C_2 \|v\|_s^2  + C_1 \sum_{d \ge 2} 5^{d} \sum_{i=1}^{d-1}
 \sum_{n\ge n_2} n^s \,  I^{(0)}_{n, \alpha}(d, i) \cdot  n^s \, II^{(0)}_{n, \alpha}(d, i)
\end{equation}
where 
$$
\begin{aligned}
& I^{(0)}_{n, \alpha}(d, i) := \sum_{\tb{k_j \ge N}{ 1 \le j < i}}  
\frac{ |\widehat{vg_\infty \overline g_{k_1}}(k_1 - n + 1) | 
%| \widehat{vg_{k_1} \overline g_{k_2}}(k_2 - k_1) |
\cdots | \widehat{vg_{k_{i-1}} \overline g_\alpha} (\alpha - k_{i-1}) | }{\prod_{ 1 \le j <i} (|k_j -n +1| +1)}
\end{aligned}
$$
and
$$
\begin{aligned}
& II^{(0)}_{n, \alpha}(d, i) := \sum_{\tb{k_j \ge 0}{ i < j \le d}}  
\frac{ |\widehat{vg_\alpha \overline g_{k_{i+1}}}(k_{i+1} - \alpha) | 
%| \widehat{vg_{k_1} \overline g_{k_2}}(k_2 - k_1) |
\cdots | \widehat{vg_{k_{d}} \overline g_\infty} (n-1 - k_{d}) | }{\prod_{ i <  j \le d} (|k_j -n +1| +1)}  \, .
\end{aligned}
$$
By Cauchy-Schwarz, 
$$
\begin{aligned}
& \sum_{n\ge n_2} n^s \,  I^{(0)}_{n, \alpha}(d, i) \cdot  n^s \, II^{(0)}_{n, \alpha}(d, i)  \\
& \le  \big( \sum_{n\ge n_2} n^{2s} \,  (I^{(0)}_{n, \alpha}(d, i))^2 \big)^{1/2} \big( \sum_{n\ge n_1} n^{2s} \,  (II^{(0)}_{n, \alpha}(d, i))^2 \big)^{1/2}
\end{aligned} 
$$
and by Lemma \ref{estimate Q} one has for any $1 \le i < d$,
$$
\begin{aligned}
 \big( \sum_{n\ge n_2} n^{2s} \,  (I^{(0)}_{n, \alpha}(d, i))^2 \big)^{1/2} & \le 
  \sup_{0 \le \beta \le N} \| \big( Q^{i-1}[ \big( (|\widehat{v g_\infty \overline{g_k}}|)_{0 \le k \le N} \big)] \big)_\beta \|_s \\
& \le  \big( 2 C_{s, 1}^3 C_{s, 2}^2  (2 + M)^{2 \eta(s)}   \|v\|_s \big)^i  ,
\end{aligned}
$$
$$
\begin{aligned}
 \big( \sum_{n\ge n_2} n^{2s} \,  (II^{(0)}_{n, \alpha}(d, i))^2 \big)^{1/2} & \le 
  \sup_{0 \le \beta \le N} \| \big( Q^{d-i}[ (\big(|\widehat{v g_k \overline{g_\infty}}|)_{0 \le k \le N} \big)] \big)_\beta \|_s \\
& \le  \big( 2 C_{s, 1}^3 C_{s, 2}^2  (2 + M)^{2 \eta(s)}   \|v\|_s \big)^{d- i+1}  ,
\end{aligned}
$$
implying that (use $\sum_{i=1}^{d-1} 1 \le d \le 2^d$)
$$
\begin{aligned}
\sum_{d \ge 2} 5^{d} \sum_{i=1}^{d-1} \sum_{n\ge n_2} n^s \,  I^{(0)}_{n, \alpha}(d, i) \cdot  n^s \, II^{(0)}_{n, \alpha}(d, i) \\
 \le \sum_{d\ge 2}  \big( 20 C_{s, 1}^3 C_{s, 2}^2  (2 + M)^{2 \eta(s)}   \|v\|_s \big)^{d+1} < \infty
 \end{aligned}
$$
uniformly for any $u= w_N + v \in B^s_{c,0}(w_N, 1/C_{M,s})$. When combined with \eqref{sum R_{n, alpha}^(0) }
we proved that for any $0 \le  \alpha <N$, $\sum_{n >n_2} | \mathcal R_{n, \alpha}^{(0)} |  < \infty$
uniformly for any $u= w_N + v \in B^s_{c,0}(w_N, 1/C_{M,s})$.

\smallskip

\noindent
{\em  Estimate of  $ \sum_{0 \le \alpha < N}  \mathcal R_{n, \alpha}^{(2,1)} $.}
Since $k_m \ge N$, the term with $m=1$ in the sum \eqref{mathcal R^(2,1)} vanishes.
Furthermore, for any  $n > n_2$ and  $k \in Z_{\ge 0}$,
$$
\sup_{\mu \in \partial D_0} \frac{1}{ |k -  n - \mu |} \le \frac{5}{|k-n| +1}, \qquad
\sup_{\mu \in \partial D_0} \frac{1}{ |k - n + 1 - \mu |} \le \frac{8}{|k-n| +1} .
$$
It then follows by \eqref{estimate2 1/(n-N)} that
$$
\sum_{n > n_2} \mathcal R_{n, \alpha}^{(2,1)}  \le C_1 \sum_{m \ge 2, r \ge 0} 8^{m + r} \sum_{i=1}^{m-1} 
\sum_{n > n_2}  n^sI^{(2,1)}_{n, \alpha}(i, m, r) \cdot n^s II^{(2,1)}_{n, \alpha}(i, m, r) 
$$
where
$$
 I^{(2,1)}_{n, \alpha}(i, m, r)  :=   \sum_{\tb{k_j \ge N, 1\le j <i}{ k_i=\alpha}} 
 \frac{|\widehat v(k_1 -n +1)| |\widehat v(k_2 - k_1)| \cdots |\widehat{vg_\infty \overline g_\alpha}(\alpha - k_{i-1})|}{\prod_{j=1}^{i-1}(|k_j -n +1| +1)} ,
$$
$$
\begin{aligned}
& I^{(2,1)}_{n, \alpha}(i, m, r) := 
  \sum_{\tb{k_i=\alpha, k_m\ge N} {k_j \ge 0, i < j <m} } \frac{ |\widehat{vg_\alpha \overline g_{k_{i+1}}}(k_{i+1} - \alpha) | 
%| \widehat{vg_{k_1} \overline g_{k_2}}(k_2 - k_1) |
\cdots | \widehat{vg_{k_{m-1}} \overline g_\infty} (k_m - k_{m-1}) | }{\prod_{ i <  j \le m} (|k_j -n | +1)}   \\
& \cdot \sum_{\tb{k_j \ge 0}{m < j \le m+r}}\frac{   | \widehat{vg_{\infty} \overline g_{k_{m+1}}} (k_{m+1} - (k_m+1)) |  
\cdots  | \widehat{v g_{k_{m+r}}\overline g_{\infty} } (n - (k_{m+r})) |}{(|k_m -n| +1)\prod_{ m <  j \le m+r} (|k_j -n| +1)} .
\end{aligned}
$$
By Cauchy-Schwarz,
$$
\begin{aligned}
&\sum_{n > n_2}  n^sI^{(2,1)}_{n, \alpha}(i, m, r) \cdot n^s II^{(2,1)}_{n, \alpha}(i, m, r) \\
& \le
\big(\sum_{n > n_2}  n^{2s} (I^{(2,1)}_{n, \alpha}(i, m, r))^2 \big)^{1/2}
\big( \sum_{n > n_2}  n^{2s} (II^{(2,1)}_{n, \alpha}(i, m, r))^2 \big)^{1/2} 
\end{aligned}
$$
and by Lemma \ref{estimate Q} one has for any $1 \le i < m$,
$$
\begin{aligned}
 \big( \sum_{n\ge n_2} n^{2s} \,  (I^{(2,1)}_{n, \alpha}(i, m, r))^2 \big)^{1/2} & \le 
  \sup_{0 \le \beta \le N} \| \big( Q^{i-1}[ \big( (|\widehat{v g_\infty \overline{g_k}}|)_{0 \le k \le N} \big)] \big)_\beta \|_s \\
& \le  \big( 2 C_{s, 1}^3 C_{s, 2}^2  (2 + M)^{2 \eta(s)}   \|v\|_s \big)^i  ,
\end{aligned}
$$
$$
\begin{aligned}
 \big( \sum_{n\ge n_2} n^{2s} \,  (II^{(2,1)}_{n, \alpha}(i, m, r))^2 \big)^{1/2} & \le 
  \sup_{0 \le \beta \le N} \| \big( Q^{m+ r - i}[ \big( (|\widehat{v {g_k}\overline g_\infty}|)_{0 \le k \le N}  \big)] \big)_\beta \|_s \\
& \le  \big( 2 C_{s, 1}^3 C_{s, 2}^2  (2 + M)^{2 \eta(s)}   \|v\|_s \big)^{m+r -i + 1}  .
\end{aligned}
$$
Hence, taking $d:= m+r$ and $m$ as summation indices and using that 
$$
\sum_{m=2}^d \sum_{i=1}^{m-1} 1 \le d^2 \le 2^{d+1} ,
$$
one concludes that
$$
\begin{aligned}
& \sum_{n > n_2} \mathcal R_{n, \alpha}^{(2,1)}  \le C_1 \sum_{m \ge 2, r \ge 0} 8^{m + r} \sum_{i=1}^{m-1} 
\big( 2 C_{s, 1}^3 C_{s, 2}^2  (2 + M)^{2 \eta(s)}   \|v\|_s \big)^{m+r + 1} \\
& \le C_1 \sum_{d \ge 2} \big( 3 2 C_{s, 1}^3 C_{s, 2}^2  (2 + M)^{2 \eta(s)}   \|v\|_s \big)^{m+r + 1} < \infty ,
\end{aligned}
$$
uniformly  for any $u= w_N + v \in B^s_{c,0}(w_N, 1/C_{M,s})$. 

\smallskip

\noindent
{\em  Estimate of  $ \sum_{0 \le \alpha < N}  \mathcal R_{n, \alpha}^{(2,2)} $.}
Using the same arguments as in the previous paragraph, one shows that
 $\sum_{n > n_2} \mathcal R_{n, \alpha}^{(2,2)}$
is uniformly summable for any $u= w_N + v \in B^s_{c,0}(w_N, 1/C_{M,s})$.

\medskip

\noindent 
Patrick  G\'erard, Laboratoire de Math\'ematiques d'Orsay, CNRS, \\ Universit\'e Paris--Saclay, 91405 Orsay, France, \\
email:  patrick.gerard@universite-paris-saclay.fr

\smallskip
\noindent
Thomas Kappeler, Institut f\"ur Mathematik, Universit\"at Z\"urich, \\
Winterthurerstrasse 190, 8057 Zurich, Switzerland \\
email:  thomas.kappeler@math.uzh.ch

\smallskip
\noindent
Petar Topalov, Department of Mathematics, Northeastern University, \\
567 LA (Lake Hall), Boston, MA 0215, USA \\
email:  p.topalov@northeastern.edu

\end{document}